\documentclass[reqno,hidelinks]{amsart}

\usepackage{a4,amsmath,amssymb,amsthm}
\usepackage[top=1in, bottom=1in, left=1.25in, right=1.25in]{geometry}
\usepackage{xfrac}
\usepackage{latexsym}
\usepackage{mathrsfs}
\usepackage[numbers]{natbib}
\usepackage{leftidx}
\usepackage{inputenc}

\usepackage{amsaddr}

\usepackage[T1]{fontenc}

\usepackage{leftidx}
\usepackage{tikz}

\usepackage{hyperref}
\newcounter{mnotecount}[section]

\usepackage{color}

\usepackage{url}

\numberwithin{equation}{section}

\usepackage[mathscr]{eucal}

\usepackage{graphicx}
\graphicspath{pic/}

\newcommand{\QED}{\hfill\ensuremath{\square}}

\newcommand{\half}{\frac{1}{2}}         %

\newcommand{\veps}{\varepsilon}

\newcommand{\R}{r^2+a^2}
\newcommand{\PR}{r^3-3Mr^2+a^2r+a^2M}

\newcommand{\prb}{\partial_{\rb}}
\newcommand{\di}{\mathrm{d}} 

\newcommand{\Donetwo}{\Omega_{\tb_1,\tb_2}}
\newcommand{\Dinfty}{\Omega_{\tb,\infty}}

\newcommand{\NPplus}{\Upsilon_{+1}}
\newcommand{\NPminus}{\Upsilon_{-1}}
\newcommand{\NPzero}{\Upsilon_{0}}
\newcommand{\NPR}{\widetilde{\Upsilon}}
\newcommand{\NPRplus}{\NPR_{+1}}
\newcommand{\NPRminus}{\NPR_{-1}}
\newcommand{\NPRzero}{\NPR_{0}}

\newcommand{\psiplusc}{\check{\psi}_{+1}}
\newcommand{\psiplus}{\psi_{+1}}
\newcommand{\psiminus}{\psi_{-1}}
\newcommand{\psizero}{\psi_{0}}
\providecommand{\psiminusHigh}[1]{\psiminus^{(#1)}}
\newcommand{\Psiplus}{\Psi_{+1}}
\newcommand{\Psiminus}{\Psi_{-1}}
\providecommand{\PsiminusHigh}[1]{\Psiminus^{(#1)}}

\newcommand{\hatpsizero}{\hat{\psi}_{0}^{\text{rad}}}
\newcommand{\hatpsizeroS}{\hat{\psi}_{0,s}^{\text{rad}}}
\newcommand{\hatpsizeroN}{\hat{\psi}_{0,n}^{\text{rad}}}

\newcommand{\Phiplus}{\Phi_{+1}}
\providecommand{\Phiminus}[1]{\Phi_{-1}^{(#1)}}
\providecommand{\phiminus}[1]{\phi_{-1}^{(#1)}}
\providecommand{\PhiplusHigh}[1]{\Phi_{+1}^{(#1)}}
\providecommand{\tildePhiminus}[1]{\widetilde{\Phi}_{-1}^{(#1)}}
\providecommand{\tildePhiplusHigh}[1]{\widetilde{\Phi}_{+1}^{(#1)}}

\providecommand{\BEplus}[2]{\mathbf{BE}_{#1,+1}^{#2}}
\providecommand{\BEminus}[2]{\mathbf{BE}_{#1,-1}^{#2}}

\providecommand{\InitialEnergyplus}[1]{\mathbb{I}^{#1}_{\Sigmazero,+1}}
\providecommand{\InitialEnergyminus}[1]{\mathbb{I}^{#1}_{\Sigmazero,-1}}

\providecommand{\InizeroEnergyplus}[2]{\mathbb{I}_{\Sigmazero,+1}^{\ell_0,#1,#2}}
\providecommand{\InizeroEnergyminus}[2]{\mathbb{I}_{\Sigmazero,-1}^{\ell_0,#1,#2}}

\providecommand{\NPCP}[1]{\mathbb{Q}_{+1}^{(#1)}}
\providecommand{\NPCN}[1]{\mathbb{Q}_{-1}^{(#1)}}

\newcommand{\Lxi}{\mathcal{L}_{\xi}}
\newcommand{\Leta}{\mathcal{L}_{\eta}}

\newcommand{\curlV}{\mathcal{V}}
\newcommand{\curlY}{\mathcal{Y}}

\newcommand{\TAO}{\mathbf{T}}

\newcommand{\pu}{\partial_u}
\newcommand{\pv}{\partial_v}

\newcommand{\DOC}{\mathcal{D}}
\newcommand{\tb}{\tau}
\newcommand{\pb}{\tilde{\phi}}
\newcommand{\rb}{\rho}
\newcommand{\Hyper}{\Sigma}

\newcommand{\Sigmazero}{\Hyper_{\tb_0}}
\newcommand{\Sigmatb}{\Hyper_{\tb}}
\newcommand{\Sigmatwo}{\Hyper_{\tb_2}}
\newcommand{\Sigmaone}{\Hyper_{\tb_1}}

\newcommand{\Horizon}{\mathcal{H}^+}
\newcommand{\Scri}{\mathcal{I}^+}
\newcommand{\Horizononetwo}{\Horizon_{\tb_1,\tb_2}}
\newcommand{\Scrionetwo}{\Scri_{\tb_1,\tb_2}}

\newcommand{\KDeri}{\mathbb{K}}
\newcommand{\SDeri}{\mathbb{S}}

\newcommand{\CDeri}{\mathbb{D}}
\newcommand{\CDeriphi}{\mathbb{D}_{\partial_{\phi}}}

\newcommand{\PDeri}{\mathbb{B}}
\newcommand{\PSDeri}{\widetilde{\mathbb{B}}}
\newcommand{\RDeri}{\mathbb{H}}
\newcommand{\ScriDeri}{\mathbb{X}_{\Scri}}

\newcommand{\ScriDeripm}{\tilde{\mathbb{X}}_{\Scri}}

\newcommand{\VR}{\hat{V}}
\newcommand{\curlVR}{\hat{\mathcal{V}}}
\newcommand{\edthR}{\mathring{\eth}}

\providecommand{\abs}[1]{\lvert#1\rvert}
\providecommand{\norm}[1]{\lVert#1\rVert}

\providecommand{\absHighOrder}[3]{\abs{#1}_{#2,#3}}
\providecommand{\absSDeri}[2]{\absHighOrder{#1}{#2}{\SDeri}}

\providecommand{\absCDeri}[2]{\absHighOrder{#1}{#2}{\CDeri}}

\providecommand{\absScriDeri}[2]{\abs{#1}_{{#2},{\mathbb{D}_1}}}
\providecommand{\absScriDerit}[2]{\abs{#1}_{{#2},{\mathbb{D}_2}}}

\makeatletter
  \def\moverlay{\mathpalette\mov@rlay}
  \def\mov@rlay#1#2{\leavevmode\vtop{%
     \baselineskip\z@skip \lineskiplimit-\maxdimen
     \ialign{\hfil$#1##$\hfil\cr#2\crcr}}}
\makeatother

\newcommand{\squareS}{\moverlay{\square\cr {\scriptscriptstyle \mathrm S}}}
\newcommand{\squareShat}{\widehat{\squareS}}

\newcommand{\reg}{k}
\newcommand{\ireg}{k_0}
\newcommand{\regl}{k'}

\theoremstyle{plain}
\newtheorem{thm}{Theorem}[section]

\newtheorem{lemma}[thm]{Lemma}
\newtheorem{prop}[thm]{Proposition}

\theoremstyle{definition}
\newtheorem{definition}[thm]{Definition}

\newtheorem{remark}[thm]{Remark}

\title{Almost Price's law in Schwarzschild and decay estimates  in  Kerr for Maxwell field}

\author[S. Ma]{Siyuan Ma$^\dagger$}
\email{siyuan.ma@sorbonne-universite.fr}
\address{$^\dagger$Laboratoire Jacques-Louis Lions,
Sorbonne Université, Campus Jussieu,
4 place Jussieu 75005 Paris, France. }

\begin{document}

%
%
%


\allowdisplaybreaks

\begin{abstract}
We consider in this work the asymptotics of a Maxwell field in Schwarzschild and Kerr spacetimes. In any subextremal Kerr spacetime, we show energy and pointwise decay estimates for all components under an assumption of a basic energy and Morawetz estimate for spin $\pm 1$ components. If restricted to slowly rotating Kerr, we utilize the basic energy and Morawetz estimates proven in an earlier work to further improve these decay estimates such that the total power of decay for all components of Maxwell field is $-7/2$. In the end, depending on if the Newman--Penrose constant vanishes or not, we prove almost sharp Price's law decay $\tb^{-5+}$ (or $\tb^{-4+}$) for Maxwell field and $\tb^{-\ell -4+}$ (or $\tb^{-\ell -3+}$) for any $\ell$ mode of the field towards a static solution on a Schwarzschild background. All estimates are uniform in the exterior of the black hole.
\end{abstract}

\maketitle


\section{Introduction}
In this paper, we prove decay estimates for Maxwell field, a real two-form $\mathbf{F}_{\alpha\beta}$ satisfying the Maxwell equations
\begin{align}\label{eq:MaxwellEqs}
\nabla^{\alpha}\mathbf{F}_{\alpha\beta}&=0 & \nabla_{[\gamma}\mathbf{F}_{\alpha\beta]}&=0,
\end{align}
in the exterior of a subextremal Kerr black hole.

\subsection{Foliation of Kerr spacetimes}
\label{sect:foliation}
The metrics of the subextremal Kerr family of spacetimes $(\mathcal{M},g_{M,a})$ ($|a|< M$), when written in Boyer-Lindquist (B-L) coordinates $(t,r,\theta,\phi)$ \cite{boyer:lindquist:1967}, take the form of
\begin{align}\label{eq:KerrMetricBoyerLindquistCoord}
g_{M,a}= & -\left(1-\tfrac{2Mr}{\Sigma} \right) \di t^2 -\tfrac{2Mar \sin^2\theta}{\Sigma}(\di t \di \phi + \di \phi \di t) \nonumber\\
 & + \tfrac{\Sigma}{\Delta} \di r^2 + \Sigma \di \theta^2 +\tfrac{\sin^2\theta}{\Sigma} \left[(r^2+a^2)^2 -a^2\Delta \sin^2\theta\right]\di \phi^2,
\end{align}
where $M$ and $a$ are the mass and angular momentum per mass of the black hole and the functions
$\Delta=\Delta(r)= r^2 -2Mr +a^2$ and $\Sigma=\Sigma(r,\theta) = r^2+a^2 \cos^2\theta$.
The Schwarzschild metric \cite{schw1916} is obtained by setting $a=0$ in \eqref{eq:KerrMetricBoyerLindquistCoord}. The two roots $r_+=M+\sqrt{M^2-a^2}$    and $
r_-=M-\sqrt{M^2-a^2}$ of  function $\Delta$ correspond to the locations of event horizon $\mathcal{H}$ and Cauchy horizon, respectively. The domain of outer communication (DOC) of a Kerr black hole is denoted as
\begin{equation}\label{def:DOC}
\mathcal{D}=\overline{\{(t,r,\theta,\phi)\in \mathbb{R}\times (r_+,\infty)\times \mathbb{S}^2\}}.
\end{equation}
In the context, we also use a phrase \textquotedblleft{a slowly rotating Kerr spacetime\textquotedblright} which should be referred to as the DOC of a Kerr spacetime with $|a|/M\ll 1$ sufficiently small.

Let $\mu=\mu(r)=\frac{\Delta}{\R}$. Define additionally a tortoise coordinate $r^*$ by
\begin{align}
\di r^*=\mu^{-1}\di r,\qquad r^*(3M)=0.
\end{align}
The B-L coordinate system is convenient when expressing the form of the Kerr metric, but the metric shows a \textquotedblleft{singularity\textquotedblright} in the coefficients in this coordinate system. To justify this is not a real singularity, one shall use a different coordinate system--the ingoing Eddington--Finkelstein coordinate system $(v,r,\theta,\tilde{\phi})$--which is regular at future event horizon $\Horizon$ and is defined by
\begin{equation}\label{def:IngoingEddiFinkerCoord}
\left\{
  \begin{array}{ll}
    \di v=\di t +\di r^*, \\
    \di \pb=\di \phi +a(r^2 +a^2)^{-1}\di r^*,\\
    r=r,\\
    \theta=\theta.\\
  \end{array}
\right.
\end{equation}
To foliate the DOC, let $h=h(r)$ be as in \cite[Equation (1.7)]{andersson2019stability} and define a hyperboloidal time function
\begin{align}\label{def:timefts}
\tb={}v-h.
\end{align}
We call $(\tb,\rb=r, \theta, \pb)$ the hyperboloidal coordinates.
Let $\tb_0\geq 1$, and define for any $\tb_0\leq \tb_1<\tb_2$,
\begin{subequations}
\label{def:domainnotations}
\begin{align}
&\Sigmaone={}\{(\tb, \rb, \theta, \pb)|\tb=\tb_1\}\cap \DOC, \quad \Donetwo={}\bigcup_{\tb\in [\tb_1,\tb_2]}\Sigmatb,\\
&\Scrionetwo={}\lim_{c\to \infty}\{\rb=c\}\cap \Donetwo, \quad
\Horizononetwo={}\Donetwo\cap\Horizon.
\end{align}
\end{subequations}
We fix $\tb_0$ by requiring  $v\geq M$ on $\Sigmazero$ such that $v\geq c (\tb+\rb)$ in $\DOC_{\tb_0,\infty}$.
As discussed in \cite{andersson2019stability},  the level sets of the time function $\tb$ are strictly spacelike with
\begin{align}
c(M)r^{-2}\leq -g(\nabla \tb,\nabla\tb)\leq C(M) r^{-2}
\end{align}
for two positive universal constants $c(M)$ and $C(M)$,
and they cross the future event horizon regularly, and for large $r$, the level sets of $\tb$ are asymptotic to future null infinity $\Scri$.

\subsection{Maxwell equations in Newman--Penrose formalism}

As is shown in \cite{Ma2017Maxwell}, one can project the Maxwell field onto a Kinnersley null tetrad \cite{Kinnersley1969tetradForTypeD}  $(l,n,m,\overline{m})$ and obtain the Newman--Penrose components of the Maxwell field
\begin{equation}\label{eq:MaxwellNPcomponentswithnosuperscript}
\NPplus = \mathbf{F}_{\mu\nu} l^\mu m^\nu ,\ \NPzero= \mathbf{F}_{\mu\nu}(l^{\mu}n^{\nu}+\overline{m}^{\mu}m^{\nu}),\  \NPminus = \mathbf{F}_{\mu\nu} \overline{m}^\mu n^\nu.
\end{equation}
The Kinnersley tetrad written in B--L coordinates is
\begin{align}\label{eq:Kinnersleytetrad}
l^\mu &= \tfrac{1}{\Delta}(r^2+a^2 , \Delta , 0 , a), \notag\\
n^\mu &= \tfrac{1}{2\Sigma} (r^2+a^2 , - \Delta , 0 , a), \notag\\
m^\mu &= \tfrac{1}{\sqrt{2} \bar{\kappa}}\left(i a \sin{\theta},0 , 1, \tfrac{i}{\sin{\theta}}\right),
\end{align}
and $\overline{m}^{\mu}$ the complex conjugates of $m^{\mu}$, with $\bar\kappa$ being the complex conjugate of $\kappa=r-ia\cos\theta$.
Define further spin $s=\pm1$ components
\begin{equation}
\begin{split}
\psiplus= 2^{-1/2}\Delta \NPplus  ,\qquad\psiminus=2^{1/2}\Delta^{-1}\kappa^2\NPminus,
\end{split}
\label{eq:spinsfields}
\end{equation}
and the middle component
\begin{align}
\label{def:middlecompMaxwell}
\psizero={}&\kappa^2\NPzero.
 \end{align}
Denote the regular, future-directed ingoing and outgoing principal null vector fields in B-L coordinates \footnote{The operator $V$ here is $\frac{\Delta}{\R}$ times the operator $V$ in \cite{Ma2017Maxwell}.}
\begin{align}\label{def:VectorFieldYandV}
Y&\triangleq \tfrac{2\Sigma}{\Delta}n^{\mu}\partial_{\mu}
=\tfrac{(r^2+a^2)\partial_t +a\partial_{\phi}}{\Delta}-\partial_r, \ &\ V&\triangleq \tfrac{\Delta}{\R}l^{\mu}\partial_{\mu}= \tfrac{
(\R)\partial_t+a\partial_{\phi}}{\R}
+\tfrac{\Delta}{\R}\partial_r.
\end{align}
The full system of Maxwell equations can be written in a form of first-order differential system:
\begin{subequations}\label{eq:TSIsSpin1Kerr}
\begin{align}
\label{eq:TSIsSpin1KerrAngular0With1}
\sqrt{2}\bar{\kappa}m^{\mu}\partial_{\mu}\psizero={}& 2\kappa^2 Y\left({\kappa}^{-1}{\psiplus}\right),\\
\label{eq:TSIsSpin1KerrAngular0With-1}
\sqrt{2}\kappa \overline{m}^{\mu}\partial_{\mu}\psizero={}&2{\kappa}^2
\tfrac{\R}{\Delta} V\left(\kappa^{-1}{\Delta}\psiminus\right),\\
\label{eq:TSIsSpin1KerrRadipal0With1}
\tfrac{\R}{\kappa^2}V \psizero=
{}&2\Big(\partial_{\theta}
-\tfrac{i}{\sin\theta}\partial_{\phi}
-ia\sin\theta\partial_t
+\cot\theta\Big)\left(\kappa^{-1}{\psiplus}\right),\\
\label{eq:TSIsSpin1KerrRadial0With-1}
Y\psizero
={}&2\kappa^2\Big(\partial_{\theta}
+\tfrac{i}{\sin\theta}\partial_{\phi}
+ia\sin\theta\partial_t+\cot\theta\Big)
\left(\kappa^{-1}{\psiminus}\right).
\end{align}
\end{subequations}

Since the Kinnersley tetrad has singularity at $\Horizon$, we can use instead a regular Hawking--Hartle tetrad as in \cite{Ma2017Maxwell}, with the same $m^{\mu}$, $\tilde{l}^{\mu}=\tfrac{\Delta}{2\Sigma}l^{\mu}$ and $\tilde{n}^{\mu}=\tfrac{2\Sigma}{\Delta}n^{\mu}$, to define regular N--P components $\NPR_i$ which are
\begin{align}\label{def:regularNPComps}
\left\{
  \begin{array}{ll}
    \NPRplus(\mathbf{F}_{\alpha\beta})=\mathbf{F}_{\mu\nu} \tilde{l}^\mu m^\nu
    =\tfrac{\Delta}{2\Sigma}\NPplus
    =\tfrac{1}{\sqrt{2}\Sigma}\psiplus, \\
   \NPRzero(\mathbf{F}_{\alpha\beta})
   =\mathbf{F}_{\mu\nu}(\tilde{l}^{\mu}\tilde{n}^{\nu}+\overline{m}^{\mu}m^{\nu})
   =\NPzero
   =\kappa^{-2}\psizero,\\
    \NPRminus(\mathbf{F}_{\alpha\beta})=\mathbf{F}_{\mu\nu} \overline{m}^\mu \tilde{n}^\nu
    =\tfrac{2\Sigma}{\Delta}\NPminus
    =\tfrac{\sqrt{2}\Sigma}{\kappa^2}\psiminus.\\
  \end{array}
\right.
\end{align}
We consider in this work \underline{only regular Maxwell fields} in the sense that all the regular N--P components are smooth in the hyperboloidal coordinates in the region $\DOC_{\tb_0,\infty}$.

\subsection{TME and BEAM estimates for spin $\pm 1$ components}

It is remarkable that Teukolsky found in \cite{Teu1972PRLseparability} that the spin $s=\pm 1$ components satisfy a decoupled, separable wave equation--the Teukolsky Master Equation (TME)--which in B--L coordinates takes the form of
\begin{align}\label{eq:TME}
& -\left[\tfrac{(r^2+a^2)^2}{\Delta} -a^2 \sin^2{\theta} \right] \tfrac{\partial^2 \psi_{[s]}}{\partial t^2} - \tfrac{4Mar}{\Delta} \tfrac{\partial^2 \psi_{[s]}}{\partial t \partial \phi}-\left[\tfrac{a^2}{\Delta} -\tfrac{1}{\sin^2{\theta}} \right] \tfrac{\partial^2 \psi_{[s]}}{\partial \phi^2}   \notag\\
&  +\Delta^{s} \tfrac{\partial}{\partial r} \left( \Delta^{-s+1} \tfrac{\partial \psi_{[s]}}{\partial r} \right) + \tfrac{1}{\sin{\theta}} \tfrac{\partial}{\partial \theta} \left( \sin{\theta} \tfrac{\partial \psi_{[s]}}{\partial \theta}\right) +2s \left[ \tfrac{a(r-M)}{\Delta} + \tfrac{i \cos{\theta}}{\sin^2{\theta} } \right] \tfrac{\partial \psi_{[s]}}{\partial \phi} \notag\\
&  +2s\left[ \tfrac{M(r^2-a^2)}{\Delta} -r -ia \cos{\theta} \right] \tfrac{\partial \psi_{[s]}}{\partial t}- (s^2 \cot^2{\theta} +s) \psi_{[s]} = 0 .
\end{align}
Note that these N--P scalars satisfy the TME differ with the ones used in \cite{Teukolsky1973I} by a rescaling of $2^{-s/2}\Delta^s$, and the reason we use these scalars lies in the fact that they are both regular at $\Horizon$ from \eqref{def:regularNPComps}. This TME serves as a starting point in obtaining estimates for the spin $\pm 1$ components, from which the full Maxwell field can then be recovered from system \eqref{eq:TSIsSpin1Kerr}.

A robust way of proving decay estimates for a wave equation is to first show a certain type of weak decay estimates, known as a Morawetz estimate. A uniform boundedness of a non-degenerate energy and such a Morawetz estimate are useful as precursors in proving stronger decay estimates. It is shown in our earlier work \cite{Ma2017Maxwell} that such estimates hold true for spin $\pm 1$ components on slowly rotating Kerr backgrounds, and we call such estimates in this paper \textquotedblleft{basic energy and Morawetz estimates (BEAM estimates).\textquotedblright} For convenience of later discussions, we shall introduce a few notations before restating these BEAM estimates.
\begin{definition}
Define $\di^2\mu=\sin\theta \di \theta \wedge \di \pb$, and define the reference volume forms
\begin{subequations}
\begin{align}
\di^3\mu ={}&\di \rb\wedge \di^2\mu,\\
\di^4\mu ={}&\di \tb\wedge\di^3\mu .
\end{align}
\end{subequations}
Given a $1$-form $\nu$, let $\di^3\mu_{\nu}$ denote a Leray $3$-form such that $\nu\wedge\di^3\mu_{\nu}=\di^4\mu$.
\end{definition}

Note that these are convenient reference volume forms in calculations and in stating
the estimates, but not the volume element of DOC or the induced volume form on a $3$-dimensional hypersurface.

\begin{definition}
Let a multi-index $\mathbf{a}$ be an ordered set $\mathbf{a}=(a_1,a_2,\ldots,a_m)$ with all $a_i\in \{1,\ldots, n\}$, $m,n\in \mathbb{Z}^+$ and let $\mathbb{X}=\{X_1, X_2, \ldots, X_n\}$ be a set of spin-weighted operators. Define $|\mathbf{a}|=m$ and define $\mathbb{X}^{\mathbf{a}}=X_{a_1}X_{a_2}\cdots X_{a_m}$. Let $\varphi$ be a spin-weighted scalar, and define its pointwise norm of order $k$, $k\in \mathbb{N}$, as
\begin{align}
\absHighOrder{\varphi}{m}{\mathbb{X}}={}\sqrt{\sum_{\abs{\mathbf{a}}\leq m}\abs{\mathbb{X}^{\mathbf{a}}\varphi}^2} .
\end{align}
\end{definition}

\begin{definition}
Let $\varphi$ be a spin weight $s$ scalar.
Let the spherical edth operators $\edthR$ and $\edthR'$ be as defined in B--L coordinates by
\begin{subequations}
\begin{align}
\edthR\varphi={}&\frac{1}{\sqrt{2}}\partial_{\theta}\varphi
+\frac{i}{\sqrt{2}}\csc\theta\partial_{\phi}\varphi
-\frac{s}{\sqrt{2}}\cot\theta\varphi,\\
\edthR'\varphi={}&\frac{1}{\sqrt{2}}\partial_{\theta}\varphi
-\frac{i}{\sqrt{2}}\csc\theta\partial_{\phi}\varphi
+\frac{s}{\sqrt{2}}\cot\theta\varphi.
\end{align}
\end{subequations}
Define first order differential operators \begin{align}\label{def:curlVop}
\curlY(\cdot)={}&\sqrt{r^2+a^2}Y(\sqrt{r^2+a^2}\cdot),
&
\curlV(\cdot)={}&\sqrt{r^2+a^2}V(\sqrt{r^2+a^2}\cdot).
\end{align}
Define two Killing vector fields
\begin{align}
\label{def:Killingvectors}
\Lxi={}\partial_{\tb}, \quad \Leta=\partial_{\pb}.
\end{align}
Define a set of operators
\begin{subequations}
\begin{align}
\PDeri={}\{Y,V, r^{-1}\edthR,r^{-1}\edthR',\Leta\}
\end{align}
adapted to the Hawking--Hartle tetrad, and its rescaled one
\begin{align}
\PSDeri={}\{rY,rV, \edthR,\edthR'\}.
\end{align}
Define a set of operators
\begin{align}
\CDeri={}\{Y,rV, \edthR,\edthR'\}
\end{align}
adapted to both the hyperboloidal foliation and the set of commutators.
Additionally, define a set of rescaled spherical edth operators
\begin{align}
\SDeri={}\{r^{-1}\edthR,r^{-1}\edthR'\}.
\end{align}
\end{subequations}
\end{definition}
Now we are able to define energy norms and (spacetime) Morawetz norms.
\begin{definition}
\label{def:basicweightednorm}
Let $\varphi$ be a spin-weighted scalar and let $k\in \mathbb{N}$ and $\gamma\in \mathbb{R}$. Let $\Omega$ be a $4$-dimensional subspace of the DOC and let $\Sigma$ be a $3$-dimensional space that can be parameterized by $(\rb,\theta,\pb)$. Define
\begin{subequations}
\begin{align}
\norm{\varphi}_{W_{\gamma}^{\reg}(\Omega)}^2
={}&\int_{\Omega} r^{\gamma}\absCDeri{\varphi}{\reg}^2\di^4\mu,\\
\norm{\varphi}_{W_{\gamma}^{\reg}(\Sigma)}^2
={}&\int_{\Hyper} r^{\gamma}\absCDeri{\varphi}{\reg}^2\di^3\mu,\\
\norm{\varphi}_{W_{\gamma}^{\reg}(\mathbb{S}^2(r))}^2
={}&\int_{\mathbb{S}^2} r^{\gamma}\absSDeri{\varphi}{\reg}^2\di^2\mu.
\end{align}
\end{subequations}
\end{definition}
For convenience of stating the BEAM estimates for spin $\pm 1$ components and the main theorems, we define a few scalars.
\begin{definition}\label{def:scalars}
For any $i\in \mathbb{Z}^+$, define
\begin{subequations}\label{eq:scalars}
\begin{align}
\psiminusHigh{0}={}\psiminus,\qquad
\psiminusHigh{i}={}\curlV^i\psiminusHigh{0}
\end{align}
and their radiation fields
\begin{align}
\label{def:Psiminusi}
\PsiminusHigh{0}={}\sqrt{\R}\psiminusHigh{0},\qquad
\PsiminusHigh{i}={}\sqrt{\R}\psiminusHigh{i},
\end{align}
and define the radiation field of $\psiplus$
\begin{align}
\Psiplus={}\sqrt{\R}\psiplus.
\end{align}
\end{subequations}
\end{definition}
Let us in the end define two initial energies of spin $\pm 1$ components respectively, both of which are crucial in stating the results about the asymptotics of the Maxwell field.

\begin{definition}
\label{def:initialenergyp=2:spinpm1}
Let $3\leq\reg\in \mathbb{Z}^+$. Define on $\Sigmazero$ an initial energy  of spin $+1$ component
\begin{subequations}
\begin{align}
\InitialEnergyplus{\reg}={}
\norm{\Psiplus}_{W_{-2}^{\reg}(\Sigmazero)}
+\norm{rV\Psiplus}_{W_{0}^{\reg-1}(\Sigmazero)},
\end{align}
and an initial energy of spin $-1$ component
\begin{align}
\InitialEnergyminus{\reg}={}\sum_{i=0,1,2}\norm{\PsiminusHigh{i}}_{W_{-2}^{\reg-i}(\Sigmazero)}
+\norm{rV\PsiminusHigh{2}}_{W_{0}^{\reg-3}(\Sigmazero)}.
\end{align}
\end{subequations}
\end{definition}

The BEAM estimates for spin $\pm 1$ components proven in \cite{Ma2017Maxwell} on slowly rotating Kerr backgrounds are as follows.
\begin{thm}\label{thm:BEAM}
In the DOC of a slowly rotating Kerr spacetime $(\mathcal{M},g=g_{M,a})$, given any $0<\delta<1/2$ and any $2\leq \reg\in \mathbb{N}^+$, there exist universal constants $\veps_0=\veps_0(M)$ and $C=C(M,\veps_0,\delta,\reg)$ such that for all $|a|/M\leq \veps_0$ and any solution $\mathbf{F}_{\alpha\beta}$ to the Maxwell equations \eqref{eq:MaxwellEqs}, one has
BEAM estimates in the region $\Donetwo$ for any $\tb_0\leq \tb_1<\tb_2$:
\begin{subequations}\label{eq:BEAM}
\begin{align}
\label{eq:BEAM:-1}
\hspace{4ex}&\hspace{-4ex}
\sum_{i=0,1}\sum_{\abs{\mathbf{a}}\leq \reg-2}\left(
\norm{\PDeri^{\mathbf{a}}\curlV^i\psiminus}_{W_0^1(\Sigmatwo)}
+\norm{\PDeri^{\mathbf{a}}\curlV^i\psiminus}_{W_{-1}^0(\Donetwo)}
+\norm{\PDeri^{\mathbf{a}}\PSDeri(\curlV^i\psiminus)}_{W_{-1}^0(\Donetwo\cap\{r\geq 4M\})}
\right)\notag\\
\leq {}&C\sum_{i=0,1}\sum_{\abs{\mathbf{a}}\leq \reg-2}
\norm{\PDeri^{\mathbf{a}}\curlV^i\psiminus}_{W_0^1(\Sigmaone)},\\
\label{eq:BEAM:+1}
\hspace{4ex}&\hspace{-4ex}
\sum_{\abs{\mathbf{a}}\leq \reg-2}\left(
\norm{\PDeri^{\mathbf{a}}(r^{-\delta}\psiplus)}_{W_0^1(\Sigmatwo)}
+\norm{\PDeri^{\mathbf{a}}Y\psiplus}_{W_0^1(\Sigmatwo)}
\right)\notag\\
\hspace{4ex}&\hspace{-4ex}
+\sum_{\abs{\mathbf{a}}\leq \reg-2}\Big(\norm{\PDeri^{\mathbf{a}}(r^{-\delta}\psiplus)}_{W_{-1}^0(\Donetwo)}
+\norm{\PDeri^{\mathbf{a}}Y\psiplus}_{W_{-1-\delta}^0(\Donetwo)}\notag\\
&\qquad \quad
+\norm{\PDeri^{\mathbf{a}}\PSDeri(r^{-\delta}\psiplus)}_{W_{-1}^0(\Donetwo)}
+\norm{\PDeri^{\mathbf{a}}\PSDeri(Y\psiplus)}_{W_{-1}^0(\Donetwo\cap\{r\geq 4M\})}\Big)\notag\\
\leq {}&C\sum_{\abs{\mathbf{a}}\leq \reg-2}\left(
\norm{\PDeri^{\mathbf{a}}(r^{-\delta}\psiplus)}_{W_0^1(\Sigmaone)}
+\norm{\PDeri^{\mathbf{a}}Y\psiplus}_{W_0^1(\Sigmaone)}\right).
\end{align}
\end{subequations}
\end{thm}

\subsection{Two conditions for spin $\pm 1$ components}

Such BEAM estimates in Theorem \ref{thm:BEAM} are only available for slowly rotating Kerr spacetimes, but are not proven yet for an arbitrary subextremal Kerr spacetime. However, from the experience of proving BEAM estimates for scalar field on subextremal Kerr backgrounds in \cite{dafermos2016decay}, these estimates for spin $\pm 1$ components of Maxwell field are expected to be extended to full subextremal Kerr backgrounds if combined with a mode stability result on the real axis in full subextremal Kerr spacetimes which in turn has been shown in \cite{andersson2017mode,da2019mode}; hence, we are inspired to put forward a BEAM condition:
\begin{definition}\textbf{(BEAM condition to order $\reg$).}
Let $2\leq\reg\in \mathbb{N}^+$. Let $M>0$ and $a$ $(|a|<M)$ be given. The spin $\pm 1$ components satisfy \textquotedblleft{BEAM condition to order $\reg$\textquotedblright} if there exists a constant $0<\delta<1/2$ and a constant $C=C(M, a, \delta, \reg)$ such that the BEAM estimates \eqref{eq:BEAM} hold true in the DOC of a Kerr spacetime $(\mathcal{M},g_{M,a})$.
\end{definition}

Combined with other tools, BEAM estimates can be used to show decay estimates for the energy, from which pointwise behaviours of the field then follow. The late-time asymptotics are relevant to many problems like \emph{the black hole (in)stability} and \emph{Strong Cosmic Censorship}, and there is a heuristic Price's law \cite{Price1972SchwScalar,Price1972SchwIntegerSpin,price2004late} which predicts the sharp upper and lower bounds of the tails of spin fields on a Schwarzschild background.  A novel idea in \cite{angelopoulos2018vector}, which proves almost Price's law for scalar field in Reissner--Nordstr\"{o}m spacetimes, is to show a weighted basic energy has stronger decay rate than the ones which are obtained in former works, and this stronger energy decay enables the authors to prove in a subsequent work \cite{angelopoulos2018late} the sharp upper and lower bounds for the scalar field. Although the methodology therein requires the background to be spherically symmetric and treats only the simplest model--the equation of scalar field, it can in principal be generalized to other spin fields on Kerr backgrounds. A first natural question would be what the analogous basic energy is for higher spin fields in Kerr spacetimes. We propose an appropriate notion of such an basic energy for each spin $\pm 1$ component as follows.
\begin{definition}
Let $\reg\in \mathbb{Z}^+$, let $j\in \mathbb{N}$ and let $p\geq 0$. Define the basic energies with weight $p$ for spin $\pm 1$ components on $\Sigmatb$
\begin{align}
\BEplus{\tb}{\reg,j,p}
={}&\norm{\Lxi^j\Psiplus}_{W_{-2}^{\reg}(\Sigmatb)}
+\norm{rV\Lxi^j\Psiplus}_{W_{p-2}^{\reg-1}(\Sigmatb)},\\
\BEminus{\tb}{\reg,j,p}={}&
\sum_{i=0,1}
\Big(\norm{\Lxi^j\PsiminusHigh{i}}_{W_{-2}^{\reg-i}(\Sigmatb)}
+\norm{rV\Lxi^j\PsiminusHigh{i}}_{W_{p-2}^{\reg-i-1}(\Sigmatb)}\Big).
\end{align}
\end{definition}
On the other hand, while it is routine to obtain pointwise decay from the energy decay for scalar field, it is non-trivial to do so for higher spin fields. We are thus interested in the problem that what (almost) sharp pointwise asymptotics can be achieved given a certain amount of decay of this basic energy; hence, this naturally introduces a condition of decay rate for such basic energies.

\begin{definition}\textbf{(Basic energy $\gamma$-decay condition).}
Let $\gamma \geq 1$ and let $k\in \mathbb{Z}^+$. The spin $+ 1$ and $-1$ components are called to satisfy \textquotedblleft{basic energy $\gamma$-decay condition\textquotedblright} on a subextremal Kerr background $(\mathcal{M},g_{M,a})$ if for any $j\in\mathbb{N}$, there exist constants  $D_{\pm 1}=D_{\pm 1}(M,a,\reg,j)$ such that for any $p\in [0,1]$,
\begin{subequations}
\begin{align}
\BEplus{\tb}{\reg,j,p}\leq{}& D_{+1}\tb^{-\gamma+p-2j} \\
\BEminus{\tb}{\reg,j,p}\leq{}& D_{-1}\tb^{-\gamma-2+p-2j},
\end{align}
\end{subequations}
respectively.
\end{definition}

\subsection{Main theorems}
We are now ready to state the main results of this work. The first result is to use the $r^p$ method initiated in \cite{dafermos2009new} and follow the approach in the part of treating spin $\pm 2$ components of linearized gravity in \cite{andersson2019stability} to prove that the BEAM condition implies $\tb^{-2}$ decay for the basic energy, i.e. basic energy $2$-decay condition, for each of spin $\pm 1$ components.
\begin{thm}\label{thm:1}
\textbf{(BEAM condition implies basic energy $2$-decay condition).} Given the BEAM condition to order $\ireg$ with $\ireg$ suitably large, there exists a constant $j_0=j_0(\reg_0)$ and a constant $\regl(j)$ such that for any $0\leq j\leq j_0$, the basic energy $2$-decay condition is satisfied for both spin $\pm1$ components with $\reg\leq \ireg-\regl(j)$, $D_{+1}=C\InitialEnergyplus{\reg_0}$ and $D_{-1}=C\InitialEnergyminus{\reg_0}$ for some $C=C(M,a,\reg_0,j)$.
\end{thm}

\begin{remark}
As discussed above, such a BEAM condition is currently only valid for slowly rotating Kerr backgrounds. However, from the experience of extending a BEAM estimate of scalar field from slowly rotating Kerr to full subextremal Kerr, it is enough to combine the techniques of proving BEAM estimates for spin $\pm 1$ components in slowly rotating Kerr spacetimes with a mode stability result for spin $s=\pm 1$ TME in any subextremal Kerr spacetime to justify this BEAM condition on any subextremal Kerr background.
\end{remark}

The second main result is to see what the asymptotics of all components of Maxwell field are by assuming basic energy $\gamma$-decay condition. We shall need the following definition.

\begin{definition}
Let $\leftidx^{\star} \mathbf{F}$ be the Hodge dual of the Maxwell field $\mathbf{F}$.
Define the electronic and magnetic charges of a Maxwell field by
\begin{align}
q_{\mathbf{E}}={}&\frac{1}{4\pi}\int_{\mathbb{S}^2(\tb,\rb)}\leftidx^{\star} \mathbf{F}, &q_{\mathbf{B}}={}&\frac{1}{4\pi}\int_{\mathbb{S}^2(\tb,\rb)} \mathbf{F}.
\end{align}
\end{definition}
These two charges are constants at all spheres $\mathbb{S}^2(\tb,\rb)$ and can be calculated from the initial data. See also Lemma \ref{lem:decomp:Maxwellfield}.

\begin{thm}\label{thm:2}
\textbf{(Basic energy $\gamma$-decay condition implies pointwise decay for the full Maxwell field).}
Let $j\in \mathbb{N}$, and let the basic energy $\gamma$-decay condition with a $\gamma\geq 1$, a suitably large $\reg$ and $D_{\pm 1}=D_{\pm 1}(M,a,\reg,j)$ be satisfied for spin $\pm 1$ components in a subextremal Kerr spacetime $(\mathcal{M},g_{M,a})$. For any $\veps\in (0,1/2)$, there exist universal constants $\veps_0=\veps_0(M)>0$, $C=C(\veps)$ and $\regl>0$ such that for all $\abs{a}\leq \veps_0$,
\begin{enumerate}
  \item for spin $\pm 1$ components,
  \begin{subequations}
  \label{eq:thm:2:spinpm1}
  \begin{align}
  \label{eq:thm:2:spinpm1:+1}
  \absCDeri{\Lxi^j\NPRplus}{\reg-\regl}\leq{}& C\times(D_{+ 1}+D_{- 1})^{\half}v^{-3}
  \tb^{-\frac{\gamma-1}{2}+\veps-j}\max\{r^{-\veps}, \tb^{-\veps}\},\\
  \label{eq:thm:2:spinpm1:-1}
  \absCDeri{\Lxi^j\NPRminus}{\reg-\regl}\leq{}& C\times(D_{-1})^{\half}v^{-1}
  \tb^{-\frac{\gamma+3}{2}+\veps-j}
  \max\{r^{-\veps}, \tb^{-\veps}\};
  \end{align}
  \end{subequations}
  \item\label{point2:thm:2}
  for the middle component, there exists a stationary function $\NPRzero^{\text{sta}}$ defined at every point $(\tb,\rb)$ by
$\NPRzero^{\text{sta}}
={\kappa^{-2}}
(q_{\mathbf{E}} + iq_{\mathbf{B}})$
such that
\begin{align}
  \label{eq:thm:2:spin0}
\absCDeri{\Lxi^j(\NPRzero
-\NPRzero^{\text{sta}}))}
{\reg-\regl}
\leq {}&C\times(D_{+ 1}+D_{- 1})^{\half}
v^{-2}\tb^{-\frac{\gamma+1}{2}+\veps-j}\max\{r^{-\veps}, \tb^{-\veps}\}.
\end{align}
\end{enumerate}
On the other hand, there exist universal constants $C$ and $\regl>0$ such that in the exterior region $\{\rb\geq \tb\}$ of any subextremal Kerr spacetime, the above estimates \eqref{eq:thm:2:spinpm1} and \eqref{eq:thm:2:spin0} are valid for $\veps=0$ and a universal constant $C$, and moreover, $(D_{+ 1}+D_{- 1})^{\half}$ can be replaced by $(D_{+ 1})^{\half}$ in \eqref{eq:thm:2:spinpm1:+1}.
\end{thm}

As an application, since the BEAM condition is shown in Theorem \ref{thm:BEAM} on slowly rotating Kerr backgrounds, the above two theorems together prove the following asmptotics for the Maxwell field on slowly rotating Kerr backgrounds.

\begin{thm}\label{thm:3}
\textbf{(Decay estimates for Maxwell field on slowly rotating and subextremal Kerr backgrounds).}
Consider a Maxwell field in the DOC of a subextremal Kerr spacetime $(\mathcal{M},g=g_{M,a})$. Let $j\in \mathbb{N}$ and let $ \ireg\in \mathbb{N}^+$ be suitably large.
\begin{enumerate}
\item For any $\veps\in (0,1/2)$, there exist universal constants $\veps_0=\veps_0(M)>0$,  $\reg'(j)>0$ and $C=C(M,j,\veps,\ireg)$ such that for all $|a|\leq \veps_0$, the estimates in Theorem \ref{thm:2}  hold true with $\reg=\ireg-\regl(j)$, $\gamma=2$, $D_{+1}=C\InitialEnergyplus{\reg_0}$ and $D_{-1}=C\InitialEnergyminus{\reg_0}$.
\item There exist universal constants $\veps_0=\veps_0(M)>0$, $\reg'(j)>0$ and $C=C(M,j,\ireg)$ such that for all $|a|\leq \veps_0$, the estimates in Theorem \ref{thm:2}  hold true in the exterior region $\{\rb\geq \tb\}$ with $\veps=0$, $\reg=\ireg-\regl(j)$, $\gamma=2$, $D_{+1}=C\InitialEnergyplus{\reg_0}$ and $D_{-1}=C\InitialEnergyminus{\reg_0}$.
\item Assume the BEAM condition holds for spin $\pm 1$ component, then there exist universal constants  $\reg'(j)>0$ and $C=C(M,a,j,\ireg)$ such that  the estimates in Theorem \ref{thm:2}  hold true in the exterior region $\{\rb\geq \tb\}$  with $\veps=0$, $\reg=\ireg-\regl(j)$, $\gamma=2$, $D_{+1}=C\InitialEnergyplus{\reg_0}$ and $D_{-1}=C\InitialEnergyminus{\reg_0}$.
\end{enumerate}
\end{thm}

\begin{remark}
We have actually shown the peeling properties for spin $\pm 1$ components in slowly rotating Kerr spacetimes, and also in any subextremal Kerr spacetime but under the BEAM condition. We also show the scalars $\NPRplus$, $\NPRzero-\NPRzero^{\text{sta}}$ and $\NPRminus$ have decay $v^{-2-s}\tb^{-3/2+s+\veps/2}\max\{r^{-\veps}, \tb^{-\veps}\}$ where $s$ is the spin weight, and the total power of decay is $-7/2$.
\end{remark}

As can be seen in the pointwise decay estimates in Theorems \ref{thm:2} and \ref{thm:3}, there is a $\tb^{-\veps}$ loss in the stationary region on Kerr backgrounds. This can be removed in Schwarzschild case, and moreover, we can prove basic energy $\gamma$-decay condition for larger $\gamma$ from which pointwise asymptotics close to Price's law can be achieved.

\begin{definition}
\label{def:energyonSigmazero:Schw}
Let $\ell_0\geq 1$, and let $\reg\geq \ell_0$ be arbitrary. Let $\tildePhiplusHigh{\ell_0-1}$ and $\tildePhiminus{\ell_0+1}$ be defined as in Definition \ref{def:tildePhiplusandminusHigh}. Define on $\Sigmazero$ an energy  of spin $+1$ component
\begin{subequations}
\begin{align}
\InizeroEnergyplus{\reg}{\alpha}={}
\sum_{i=0}^{\ell_0-1}\norm{(r^2V)^i\Psiplus}^2_{W_{-2}^{\reg-i}(\Sigmazero)}+
\norm{r^2V\tildePhiplusHigh{\ell_0-1}}^2_{W_{\alpha}^{\reg-\ell_0}(\Sigmazero\cap\{\rb\geq 3M\})},
\end{align}
and an energy of spin $-1$ component
\begin{align}
\InizeroEnergyminus{\reg}{\alpha}={}
\sum_{i=0}^{\ell_0+1}\norm{(r^2V)^i\Psiminus}^2_{W_{-2}^{\reg+2-i}(\Sigmazero)}+
\norm{r^2V\tildePhiminus{\ell_0+1}}^2_{W_{\alpha}^{\reg-\ell_0}(\Sigmazero\cap\{\rb\geq 3M\})}.
\end{align}
\end{subequations}
\end{definition}
\begin{thm}\label{thm:Schw}
\textbf{(Almost Price's law for Maxwell field on Schwarzschild).}
Consider a Maxwell field on a Schwarzschild spacetime. Let N--P constants $\NPCN{i}$, $i\in \mathbb{Z}^+$, be defined as in Definition \ref{def:NPCs}. Let $\veps\in (0,1/2)$ be arbitrary and let $j\in \mathbb{N}$.
Assume the spin $\pm 1$ components are supported on $\ell\geq\ell_0$ modes\footnote{See Section \ref{sect:decompIntoModes} for mode decompositions. } for $\ell_0\geq 1$.
\begin{enumerate}
\item \label{pt:NPCneq0:l=l0:decay:Schw}
If the $\ell_0$-th N--P constant $\NPCN{\ell_0}$ of $\ell=\ell_0$ mode does not vanish, then the basic energy $\gamma$-decay condition with $\gamma=2\ell_0+1-\veps$ holds for the spin $\pm 1$ components, and there exist universal constants  $C=C(\veps,j,\reg,\ell_0)$ and $\regl(j,\ell_0)>0$ such that
\begin{itemize}
  \item for spin $\pm 1$ components,
  \begin{subequations}
  \begin{align}
  \absCDeri{\Lxi^j\NPRplus}{\reg-\regl(j,\ell_0)}\leq{}& C\times(\InizeroEnergyplus{\reg}{-1-\veps}
  +\InizeroEnergyminus{\reg}{-1-\veps})^{\half}v^{-3}
  \tb^{-\frac{2\ell_0-\veps}{2}-j},\\
  \absCDeri{\Lxi^j\NPRminus}{\reg-\regl(j,\ell_0)}\leq{}& C\times(\InizeroEnergyminus{\reg}{-1-\veps})^{\half}v^{-1}
  \tb^{-\frac{2\ell_0+4-\veps}{2}-j};
  \end{align}
  \end{subequations}
  \item\label{point2:thm:4:l=l0}
  for the middle component, there exists a static function $\NPRzero^{\text{sta}}$ defined at every point $(\tb,\rb)$ by
$\NPRzero^{\text{sta}}
={r^{-2}}
(q_{\mathbf{E}} + iq_{\mathbf{B}})$
such that
\begin{align}
\absCDeri{\Lxi^j(\NPRzero
-\NPRzero^{\text{sta}}))}
{\reg-\regl(j,\ell_0)}
\leq {}&C\times(\InizeroEnergyplus{\reg}{-1-\veps}
  +\InizeroEnergyminus{\reg}{-1-\veps})^{\half}
v^{-2}\tb^{-\frac{2\ell_0+2-\veps}{2}-j}.
\end{align}
\end{itemize}
That is, the scalars $\NPRplus$, $\NPRzero-\NPRzero^{\text{sta}}$ and $\NPRminus$ have decay $v^{-2-s}\tb^{-1+s-\ell_0+\veps/2}$ where $s$ is the spin weight.
\item \label{pt:NPCeq0:l=l0:decay:Schw}
Instead, if the $\ell_0$-th N--P constant $\NPCN{\ell_0}$  of $\ell=\ell_0$ mode vanishes, then the basic energy $\gamma$-decay condition with $\gamma=2\ell_0+3-\veps$ holds for the spin $\pm 1$ components, and the above pointwise decay rates of $\NPRplus$, $\NPRzero-\NPRzero^{\text{sta}}$ and $\NPRminus$ hold by decreasing the power of $\tb$ by $1$ and replacing the initial energy  $\InizeroEnergyplus{\reg}{-1-\veps}$ and
 $\InizeroEnergyminus{\reg}{-1-\veps}$ by $\InizeroEnergyplus{\reg}{1-\veps}$ and
 $\InizeroEnergyminus{\reg}{1-\veps}$, respectively.
 That is, the scalars $\NPRplus$, $\NPRzero-\NPRzero^{\text{sta}}$ and $\NPRminus$ have decay $v^{-2-s}\tb^{-2+s-\ell_0+\veps/2}$ where $s$ is the spin weight.
\end{enumerate}
\end{thm}

\begin{remark}
The components of Maxwell field have homogeneity in the total rate of the pointwise decay, and the total power of decay is $-\ell_0-4+$ or $-\ell_0-3+$ depending on if the $\ell_0$-th N--P constant $\NPCN{\ell_0}$ of $\ell=\ell_0$ mode vanishes or not. In particular, we show that the Maxwell field approaches a static solution in a uniform $v^{-1}\tb^{-4-}$ decay if the data are compactly supported or  decay sufficiently fast towards spatial infinity on an initial future Cauchy surface terminating at spatial infinity. Compared to the Price's law predicted in \cite{Price1972SchwScalar,Price1972SchwIntegerSpin,price2004late,gleiser2008late}, there is only an $\veps$ loss of time decay, and hence this is an almost sharp upper bound for the decay estimate of Maxwell field on Schwarzschild.

In the stationary region where $r$ is finite, the Price's law suggests $\tb^{-2\ell_0-3}$ or $\tb^{-2\ell_0-2}$ pointwise asymptotics depending on if the $\ell_0$-th N--P constant $\NPCN{\ell_0}$ of $\ell=\ell_0$ mode vanishes or not, and our results do \underline{not} imply almost Price's law in this compact region for $\ell_0\geq 2$ modes.
\end{remark}

\begin{remark}
There are currently many works aiming at proving Price's law for integer-spin fields on various backgrounds. See Section \ref{sect:relatedresults} below for explicit references. We would like to point out that the current best results for Maxwell field as we know are $\tb^{-4}$ decay obtained in \cite{donninger2012pointwise} on Schwarzschild and in \cite{metcalfe2017pointwise} on a class of stationary asymptotically flat backgrounds under an assumption of an energy and Morawetz estimate, that is, one less power of decay than Price's law.
\end{remark}

\begin{remark}
The proof going from Theorem \ref{thm:3} to this theorem is in the spirit of \cite{angelopoulos2018vector}, and in view of the work \cite{angelopoulos2018late} in which Price's law for scalar field on Reissner-Nordstr\"{o}m is obtained, the N-P constants found here will be crucial in proving the Price's law and characterizing the precise leading asymptotics of Maxwell field on Schwarzschild.
\end{remark}

\subsection{Related results}
\label{sect:relatedresults}

There is a large amount of literature in wave equations in asymptotically flat spacetimes. Pioneering works on wave equations in Minkowski include \cite{morawetz1968time,klainerman1986null,christodoulou1986global,CK93global,lindblad2010global} and references therein.

Boundedness of scalar field on Schwarzschild was first proven by Wald and Kay--Wald \cite{wald1979note,kay1987linear}. Blue--Soffer \cite{bluesoffer03mora,blue:soffer:integral} and Dafermos--Rodnianski \cite{dafrod09red} obtained integral local energy decay and pointwise decay estimates using a Morawetz type multiplier.  Finster--Kamran--Smoller--Yau \cite{Finster2006} obtained some partial results towards the decay of scalar field on Kerr and Dafermos--Rodnianski \cite{dafermos2011bdedness} showed the first uniform boundedness result in slowly rotating Kerr spacetimes. Integrated local energy decay estimates are extended to slowly rotating Kerr by Andersson--Blue \cite{larsblue15hidden} and Tataru--Tohaneanu \cite{tataru2011localkerr}, and further to subextremal Kerr by Dafermos--Rodnianski--Shlapentokh-Rothman \cite{dafermos2016decay} using a generalization \cite{2015AnHP...16..289S} by Shlapentokh-Rothman of the classic mode stability result \cite{whiting1989mode} by Whiting. In all these works for scalar field on Kerr, the separability of the wave equation or the complete integrability of the geodesic flow found by Carter \cite{Carter1968Separability} is of crucial importance.
The asymptotics for scalar field on these backgrounds are improved by Schlue \cite{schlue2013decay} and Moschidis \cite{moschidis2016r}.
Strichartz estimates are shown in \cite{MMTT,tohaneanu2012strichartz}. In all these works, the decay rate is at most $v^{-1}\tb^{-1/2}$.

Decay behaviours for Maxwell field outsider a Schwarzschild black hole were first obtained by Blue \cite{blue08decayMaxSchw} where the author started with the Fackerell--Ipser equation \cite{fackerell:ipser:EM} satisfied by the middle component, while Pasqualotto \cite{pasqualotto2019spin} proved decay estimates for Maxwell field on a Schwarzschild spacetime starting from the TME of spin $\pm 1$ components. Sterbenz and Tataru \cite{sterbenz2015decayMaxSphSym}  also proved integrated local energy decay estimates for Maxwell field but in more general spherically symmetric stationary spacetimes.  Turning to the Maxwell field outsider a slowly rotating Kerr black hole, Andersson--Blue \cite{larsblue15Maxwellkerr} proved energy and Morawetz estimates for both the fisrt-order Maxwell system and the Fackerell-Ipser wave equation for the middle component, and Ma \cite{Ma2017Maxwell} proved similar estimates by treating only the TME satisfied by spin $\pm 1$ components.
We note also the work \cite{andersson16decayMaxSchw} by Andersson--B\"{a}ckdahl--Blue where  a conserved energy current is obtained in Schwarzschild after contracting a superenergy tensor with a Killing vector $\partial_t$, and Gudapati constructed in \cite{gudapati2017positive,gudapati2019conserved} a conserved, positive definite energy for axially symmetric Maxwell field on Kerr and Kerr-de Sitter backgrounds.

The tails of integer-spin fields on Schwarzschild are analyzed by Price in \cite{Price1972SchwScalar,Price1972SchwIntegerSpin} and later by Price--Burko \cite{price2004late}, and these heuristic asymptotics are called Price's law. In particular, they also discussed the tails of a fixed $\ell$ mode. Generalizations of $\ell$-dependent Price's law to Kerr spacetimes are discussed in \cite{gleiser2008late}. In these works, they suggest that for any fixed $\ell$ mode of integer spin fields in a Schwarzschild spacetime, it falls off as $\tb^{-2\ell-3}$ at any finite radius as $\tb\to \infty$ provided the initial data decay sufficiently fast (or have compact support) on an initial future Cauchy surface terminating at spatial infinity.
Donninger--Schlag--Soffer  first proved in \cite{donninger2011proof} $\tb^{-2\ell-2}$ decay for a fixed $\ell$ mode of scalar field on Schwarzschild and then obtained in \cite{donninger2012pointwise} $\tb^{-3}$, $\tb^{-4}$ and $\tb^{-6}$ for scalar field, Maxwell field and gravitational perturbations on Schwarzschild, respectively.
Tataru \cite{tataru2013local} proved $\tb^{-3}$ decay for scalar field on a class of stationary asymptotically flat backgrounds under an assumption that an energy and Morawetz estimate holds, and subsequently, Metcalfe--Tataru--Tohaneanu extended this result and obtained  $\tb^{-3}$ decay for scalar field in \cite{metcalfe2012price} and $\tb^{-4}$ decay for Maxwell field \cite{metcalfe2017pointwise} in a class of non-stationary asymptotically flat spacetimes under a similar assumption. Angelopoulos--Aretakis--Gajic proved in \cite{angelopoulos2018vector} an almost Price's law for scalar field on a class of spherically symmetric, stationary spacetimes including Schwarzschild and subextremal Reissner--Nordstr\"{o}m spacetimes, and further obtained in \cite{angelopoulos2018late} the precise $t^{-3}$ leading order term in the asymptotic profiles for the first time. For the radiation field, apart from the $\tb^{-2}$ leading term, they are able to calculate in \cite{angelopoulos2019logarithmic} the subleading $\tb^{-3}\log \tb$ term. More recently, Angelopoulos--Aretakis--Gajic announced a result saying that the leading order term of $\ell$ mode of scalar field on a subextremal Reissner--Nordstr\"{o}m background can be expressed explicitly, hence justifying the Price's law, and Hintz
\cite{hintz2020sharp} computed the $\tb^{-3}$ leading order term for the scalar field on both Schwarzschild and subextremal Kerr spacetimes and obtained $\tb^{-2\ell-3}$ upper bound for a fixed $\ell$ mode on Schwarzschild.

Linear stability of Schwarzschild spacetimes is shown by Dafermos--Holzegel--Rodnianski \cite{dafermos2019linear} and Hung--Keller--Wang \cite{hung2017linearstabSchw} and, more recently, by Hung \cite{Hung18odd, Hung19even} and Johnson \cite{johnson2019linear} in a harmonic gauge,
while linear stability of any subextremal Reissner--Nordstr\"{o}m spacetime is proven by Giorgi \cite{Giorgi2019linearRNsmallcharge,Giorgi2019linearRNfullcharge}.
The energy estimates and decay estimates for linearized gravity on Schwarzschild are also obtained by Andersson--Blue--Wang \cite{Jinhua17LinGraSchw}.

Energy estimate and integrated local energy estimate for TME of spin $\pm 2$ components of linearized gravity are obtained by Ma \cite{Ma17spin2Kerr} and Dafermos--Holzegel--Rodnianski \cite{dafermos2019boundedness}. In order to generalize these estimates to full subextremal range of Kerr spacetimes, a crucial ingredient, i.e. a generalization of Whiting's mode stability result, has been proven by Andersson--Ma--Paganini--Whiting \cite{andersson2017mode} and da Costa \cite{da2019mode}. Note also the work \cite{finster2016linear} by Finster--Smoller which
discusses the stability problem for each azimuthal mode solution to spin $\pm 2$ TME.
The linear stability for slowly rotating Kerr spacetimes are proven by Andersson--B\"{a}ckdahl--Blue--Ma \cite{andersson2019stability} and H\"{a}fner--Hintz--Vasy
\cite{hafner2019linear}.

Lindblad--Tohaneanu \cite{lindblad2018global,lindblad2020local} obtained a global existence result for a quasilinear wave equation on Schwarzschild and Kerr backgrounds.
In the end, we mention two nonlinear stability results by Klainerman--Szeftel
\cite{kla2015globalstabwavemapKerr} for Schwarzschild under polarized axisymmetry  and by Hintz--Vasy \cite{HintzKds2018} for  slowly rotating Kerr-de Sitter.

\subsection*{Acknowledgment}
The author would like to thank the co-authors Lars Andersson, Thomas B\"{a}ckdahl and Pieter Blue in an earlier joint work for helpful discussions and valuable insights which inspire this current work. The author acknowledges the support by the ERC grant ERC-2016 CoG 725589 EPGR, as well as the hospitality of Institut Mittag-Leffler in the fall semester 2019 where part of the work was done.

\section{General conventions and basic estimates}

\subsection{Conventions}
$\mathbb{N}$ is denoted as the natural number set $\{0,1,\ldots\}$, and $\mathbb{Z}^+$ the positive integer set. Denote $\Re(\cdot)$ as the real part.

We denote a universal constant by $C$. If it depends on a set of parameters $\mathbf{P}$, we denote it by $C(\mathbf{P})$. We use regularity parameters, generally denoted by $\reg$, and $\regl$ which is a universal constant. Also, $\regl(\mathbf{P})$ means a regularity constant depending on the parameters in the set $\mathbf{P}$.

We say $F_1\lesssim F_2$ if there exists a universal constant $C$ such that $F_1\leq CF_2$. Similarly for $F_1\gtrsim F_2$. If both $F_1\lesssim F_2$ and $F_1\gtrsim F_2$ hold, we say $F_1\sim F_2$.

Let $\mathbf{P}$ be a set of parameters. We say $F_1\lesssim_{\mathbf{P}} F_2$ if there exists a universal constant $C(\mathbf{P})$ such that $F_1\leq C(\mathbf{P})F_2$. Similarly for $F_1\gtrsim_{\mathbf{P}} F_2$. If both $F_1\lesssim_{\mathbf{P}} F_2$ and $F_1\gtrsim_{\mathbf{P}} F_2$ hold, then we say $F_1\sim_{\mathbf{P}} F_2$.

For any $\alpha \in \mathbb{N}$, we say a function $f(r,\theta,\pb)$ is $O(r^{-\alpha})$ if it is a sum of two smooth functions $f_1(\theta,\pb) r^{-\alpha}$ and $f_2(r,\theta,\pb)$ satisfying that for any $j\in \mathbb{N}$,
$\abs{(\partial_r)^j f_2}\leq C_j r^{-\alpha-1-j}$. Therefore, in particular, if $f$ is $O(1)$, then $\abs{\partial_r f }\leq Cr^{-2}$.

Let  $\chi_1$ be a standard smooth cutoff function which is decreasing, $1$ on $(-\infty,0)$, and $0$ on $(1,\infty)$, and let $\chi=\chi_1((R_0-r)/M)$ with $R_0$ suitably large and to be fixed in the proof.  So $\chi=1$ for $r\geq R_0$ and vanishes identically for $r\leq R_0-M$.

\subsection{Further definitions}

Following \cite{andersson2019stability}, we define a wave operator acting on a spin weight $s$ scalar.
\begin{definition}\label{def:squareShat}
Define a spin-weighted wave operator
\begin{align}
\label{eq:squareShat}
\squareShat_{s} = {}&-(\R)YV
+2\edthR\edthR'
+2a\Lxi\Leta+a^2 \sin^2 \theta\Lxi^2
-2ias\cos\theta \Lxi
\notag\\
&+\frac{2ar}{\R}\Leta
-s^2\frac{2Mr^3+a^2r^2-4a^2Mr+a^4}{(\R)^2}.
\end{align}
\end{definition}

\begin{remark}\label{rem:squareShatandLs}
This is the same as \cite[Equation (2.36)]{andersson2019stability} except for a sign difference because of the signature convention difference in the metric.
Compared with the  wave operator $\mathbf{L}_{s}$ in \cite[Equation (1.32)]{Ma2017Maxwell}, we have for any spin weight $s$ scalar $\varphi$ that
\begin{align}
\label{eq:BoxHatandLs}
\squareShat_{s} (\sqrt{\R}\varphi)={}&
\sqrt{\R}\bigg(\mathbf{L}_{s}+\frac{s^2(r^4-2Mr^3+6a^2Mr-a^4)}{(\R)^2}-s\bigg)\varphi.
\end{align}
\end{remark}

\begin{definition}
Let $\tb_2>\tb_1\geq \tb_0$  and let $r_2>r_1\geq r_+$. Define
\begin{subequations}
\label{def:domainnotations:subdomains}
\begin{align}
\Sigmaone^{r_1}={}&\Sigmaone\cap \{r\geq r_1\}, & \Donetwo^{r_1}={}&\Donetwo\cap\{r\geq r_1\},\\
\Sigmaone^{r_1,r_2}={}&\Sigmaone\cap \{r_1\leq r\leq r_2\}, &\Donetwo^{r_1,r_2}={}&\Donetwo\cap\{r_1\leq r\leq r_2\},\\
\Sigmaone^{\leq r_1}={}&\Sigmaone\cap \{r_+\leq r\leq r_1\}, & \Donetwo^{\leq r_1}={}&\Donetwo\cap\{r_+\leq r\leq r_1\}.
\end{align}
\end{subequations}
\end{definition}

\subsection{Decomposition into modes for spin weight $s$ scalars}
\label{sect:decompIntoModes}

For any spin weight $s$ scalar $\varphi$, we can decompose it into modes $\varphi=\sum\limits_{\ell_0=\abs{s}}^{\infty}\varphi^{\ell=\ell_0}$, with each mode
$\varphi^{\ell=\ell_0}=\sum\limits_{m=-\ell_0}^{\ell_0}\varphi_{m\ell_0}
(\tb,\rb)Y_{m\ell_0}^{s}(\cos\theta)e^{im\pb}$.
Here, $\left\{Y_{m\ell}^{s}(\cos\theta)e^{im\pb}\right\}_{m,\ell}$ are the eigenfunctions, called as "\emph{spin-weighted spherical harmonics}", of a self-adjoint operator
$2\edthR\edthR'$
on $L^2(\sin\theta d\theta)$, form a complete orthonormal basis on $L^2(\sin\theta \di\theta\di \pb)$ and have eigenvalues $-\Lambda_{\ell}=-(\ell+s)(\ell-s+1)$ defined by
\begin{equation}
\label{eq:eigenvalueSWSHO}
2\edthR\edthR'(Y_{m\ell}^{s}(\cos\theta)e^{im\pb})=
-\Lambda_{\ell}
Y_{m\ell}^{s}(\cos\theta)e^{im\pb}.
\end{equation}
In particular,
\begin{align}
\label{eq:l=l0mode:eigenvalue}
2\edthR\edthR'\varphi^{\ell=\ell_0}
={}-(\ell_0+s)(\ell_0-s+1)\varphi^{\ell=\ell_0}, \quad
2\edthR'\edthR\varphi^{\ell=\ell_0}
={}-(\ell_0-s)(\ell_0+s+1)\varphi^{\ell=\ell_0}.
\end{align}

\subsection{Simple estimates}
The following simple Hardy's inequality will be useful.
\begin{lemma}
Let $\varphi$ be a spin weight $s$ scalar. Then for any $r'>r_+$,
\begin{align}
\label{eq:Hardy:trivial}
\int_{r_+}^{r'}\abs{\varphi}^2\di r
\lesssim{}&\int_{r_+}^{r'}\mu^2r^2\abs{\partial_r\varphi}^2\di r
+(r'-r_+)\abs{\varphi(r')}^2.
\end{align}
In particular, if $\lim\limits_{r\to \infty} r\abs{\varphi}^2 =0$, then
\begin{align}
\label{eq:Hardy:trivial:1}
\int_{r_+}^{\infty}\abs{\varphi}^2\di r
\lesssim{}&\int_{r_+}^{\infty}\mu^2r^2\abs{\partial_r\varphi}^2\di r.
\end{align}
\end{lemma}

\begin{proof}
It follows easily by integrating the following equation
\begin{align}
\partial_r((r-r_+)\abs{\varphi}^2)=\abs{\varphi}^2
+2(r-r_+)\Re(\bar{\varphi}\partial_r\varphi)
\end{align}
from $r_+$ to $r'$ and applying Cauchy-Schwarz to the last product term.
\end{proof}

We will also use the following standard Hardy's inequality, cf. \cite[Lemma 4.30]{andersson2019stability}.
\begin{lemma}[One-dimensional Hardy estimates]
\label{lem:HardyIneq}
Let $\alpha \in \mathbb{R}\setminus \{0\}$  and $h: [r_0,r_1] \rightarrow \mathbb{R}$ be a $C^1$ function.
\begin{enumerate}
\item \label{point:lem:HardyIneqLHS} If $r_0^{\alpha}\vert h(r_0)\vert^2 \leq D_0$ and $\alpha<0$, then
\begin{subequations}
\label{eq:HardyIneqLHSRHS}
\begin{align}\label{eq:HardyIneqLHS}
-2\alpha^{-1}r_1^{\alpha}\vert h(r_1)\vert^2+\int_{r_0}^{r_1}r^{\alpha -1} \vert h(r)\vert ^2 \di r \leq \frac{4}{\alpha^2}\int_{r_0}^{r_1}r^{\alpha +1} \vert \partial_r h(r)\vert ^2 \di r-2\alpha^{-1}D_0.
\end{align}
\item \label{point:lem:HardyIneqRHS} If $r_1^{\alpha}\vert h(r_1)\vert^2 \leq D_0$ and $\alpha>0$, then
\begin{align}\label{eq:HardyIneqRHS}
2\alpha^{-1}r_0^{\alpha}\vert h(r_0)\vert^2+\int_{r_0}^{r_1}r^{\alpha -1} \vert h(r)\vert ^2 \di r \leq \frac{4}{\alpha^2}\int_{r_0}^{r_1}r^{\alpha +1} \vert \partial_r h(r)\vert ^2 \di r +2\alpha^{-1}D_0.
\end{align}
\end{subequations}
\end{enumerate}
\end{lemma}

Recall the following Sobolev-type estimates from \cite[Lemmas 4.32, 4.33]{andersson2019stability}.
\begin{lemma}
\label{lem:Sobolev}
Let $\varphi$ be a spin weight $s$ scalar. Then
\begin{align}
\label{eq:Sobolev:1}
\sup_{\Sigmatb}\abs{\varphi}^2\lesssim_{s}{} \norm{\varphi}_{W_{-1}^3(\Sigmatb)}^2.
\end{align}
If $\alpha\in (0,1]$, then
\begin{align}
\label{eq:Sobolev:2}
\sup_{\Sigmatb}\abs{\varphi}^2\lesssim_{s,\alpha} {} (\norm{\varphi}_{W_{-2}^3(\Sigmatb)}^2
+\norm{rV\varphi}_{W_{-1-\alpha}^2(\Sigmatb)}^2)^{\half}
(\norm{\varphi}_{W_{-2}^3(\Sigmatb)}^2
+\norm{rV\varphi}_{W_{-1+\alpha}^2(\Sigmatb)}^2)^{\half}.
\end{align}
If $\lim\limits_{{\tb\to\infty}}\abs{r^{-1}\varphi}=0$ pointwise in $(\rb,\theta,\pb)$, then
\begin{align}
\label{eq:Sobolev:3}
\abs{r^{-1}\varphi}^2\lesssim_{s} {}\norm{\varphi}_{W_{-1}^3(\DOC_{\tb,\infty})}
\norm{\Lxi\varphi}_{W_{-1}^3(\DOC_{\tb,\infty})}.
\end{align}
\end{lemma}

\begin{prop}
\label{prop:basicesti}
\begin{enumerate}
\item \label{prop:basicesti:pt2}
Let $\varphi$ be a spin weight $s$ scalar and be supported on $\ell\geq \ell_0$ modes. Then
\begin{align}
\label{eq:ellip:highermodes}
\hspace{4ex}&\hspace{-4ex}
\int_{\mathbb{S}^2}\bigg(\abs{\edthR'\varphi}^2
-\frac{(\ell_0+s)(\ell_0-s+1)}{2}\abs{\varphi}^2\bigg) \di^2\mu\notag\\
={}&\int_{\mathbb{S}^2}\bigg(\abs{\edthR\varphi}^2
-\frac{(\ell_0-s)(\ell_0+s+1)}{2}\abs{\varphi}^2\bigg) \di^2\mu \geq 0.
\end{align}
In particular, let $\varphi$ be an arbitrary spin weight $s$ scalar, then
\begin{align}
\label{eq:ellipestis}
\int_{\mathbb{S}^2}\bigg(\abs{\edthR'\varphi}^2
-\frac{s+\abs{s}}{2}\abs{\varphi}^2\bigg) \di^2\mu =\int_{\mathbb{S}^2}\bigg(\abs{\edthR\varphi}^2
-\frac{\abs{s}-s}{2}\abs{\varphi}^2\bigg) \di^2\mu \geq 0.
\end{align}
\item \label{prop:basicesti:pt1}
Let $m\in \mathbb{Z}^+$ and let a multiindex $\mathbf{a}$ $(\abs{\mathbf{a}}= m)$ be given. Let $\CDeriphi=\CDeri\cup \{\partial_{\phi}\}$. Then
\begin{align}
[V,\CDeri^{\mathbf{a}}]={}&
\sum_{\abs{\mathbf{a}_1}\leq m-1} f_{\mathbf{a}_1}\CDeriphi^{\mathbf{a}_1} V
+\sum_{\abs{\mathbf{a}_2}\leq m} h_{\mathbf{a}_2}\CDeriphi^{\mathbf{a}_2},
\end{align}
where for each $\mathbf{a}_1$ and $\mathbf{a}_2$, $f_{\mathbf{a}_1}= O(1)$ and $h_{\mathbf{a}_2}=O(r^{-2})$.
\end{enumerate}
\end{prop}
\begin{proof}
A proof of point \ref{prop:basicesti:pt2} can be found in \cite[Lemma 4.25]{andersson2019stability} together with the fact that $\edthR\edthR'=\edthR'\edthR-s$.

As to point \ref{prop:basicesti:pt1}, we consider first $m=1$. This is manifest since $[V,\edthR]=0$, $[V,\edthR']=0$, $[V,rV]=\frac{\Delta}{\R}V$, and
\begin{align}
\label{eq:YVcommutator}
[V,Y]=
\frac{-4ar}{(\R)^2}\partial_{\phi}-\frac{2M(r^2-a^2)}{(\R)^2}Y
= aO(r^{-3})\partial_{\phi}+MO(r^{-2})Y.
\end{align}
The general $m\in \mathbb{Z}^+$ cases follow by induction and the following fact:
For any $X_1, X_2$ lying in the span of $\CDeriphi$ with $O(1)$ coefficients, by direct calculations, there exists a vector field $X_3$ also lying in the span of $\CDeriphi$ with $O(1)$ coefficients such that $[r^{-2}X_1,X_2]= {}r^{-2} X_3$.
\end{proof}

\subsection{A general $r^p$ estimate for a spin-weighted wave equation}
We present an $r^p$ estimate for a spin-weighted wave equation with a source term.

\begin{definition}\label{def:scriandothernorm}
\begin{enumerate}
  \item
  Define a few sets of operators acting on spin weight $s$ scalars
\begin{subequations}
\label{def:mathbbD1and2}
\begin{align}
\mathbb{D}_1={}&\{\Lxi, \edthR,\edthR', rV\},\\
\mathbb{D}_2={}&\left\{
                  \begin{array}{ll}
                    \{\Lxi, \edthR, rV\}, & \hbox{if $s>0$;} \\
                    \{\Lxi, \edthR', rV\}, & \hbox{if $s\leq 0$,}
                  \end{array}
                \right.\\
\ScriDeri={}&\{\Lxi, \edthR, \edthR'\},\\
\ScriDeripm={}&\left\{
                  \begin{array}{ll}
                    \{\Lxi, \edthR\}, & \hbox{if $s>0$;} \\
                    \{\Lxi, \edthR'\}, & \hbox{if $s\leq 0$.}
                  \end{array}
                \right.
\end{align}
\end{subequations}
Define scri fluxes for any $\tb_2>\tb_1$
\begin{subequations}
\begin{align}
\norm{\varphi}_{F^{\reg}(\Scrionetwo)}^2
={}&\int_{\Scrionetwo}  \absScriDeri{\Lxi\varphi}{\reg-1}^2\di \tb\di^2\mu,\\
\norm{\varphi}_{W^{\reg}(\Scrionetwo)}^2
={}&\int_{\Scrionetwo}  \absScriDeri{\ScriDeri\varphi}{\reg-1}^2\di \tb\di^2\mu,\\
\norm{\varphi}_{\tilde{W}^{\reg}(\Scrionetwo)}^2
={}&\int_{\Scrionetwo}  \absScriDerit{\ScriDeripm\varphi}{\reg-1}^2\di \tb\di^2\mu.
\end{align}
\end{subequations}

  \item Define a Teukolsky angular operator
\begin{align}
\TAO={}&2\edthR\edthR' +a^2\sin^2\theta\Lxi^2 -2ias\cos\theta\Lxi,
\end{align}
and a spherical operator
\begin{align}
\mathbf{S}=\TAO\vert_{a=0}=2\edthR\edthR'.
\end{align}
Define sets of operators
\begin{align}
\KDeri_1={}\{Y, rV\},\qquad
\KDeri_2={}\{\Lxi, \TAO, rV\},
\end{align}
and for $\mathbf{n}=(n_1,n_2,n_3)$ with $\abs{\mathbf{n}}=n_1+n_2+2n_3$,
\begin{align}
\KDeri_2^\mathbf{n}=\{\Lxi^{n_1}(rV)^{n_2}\TAO^{n_3}\}.
\end{align}
Define a norm square in a $4$-dimensional subspace of the DOC as in Definition \ref{def:basicweightednorm}
\begin{align}
\norm{\varphi}_{\hat{W}_{\gamma}^{\reg}(\Omega)}^2
={}&\sum_{\abs{\mathbf{a}}+2{\mathbf{b}}=\reg}
\int_{\Omega} r^{\gamma}\abs{\KDeri_1^{\mathbf{a}}\mathbf{S}^{{b}}\varphi}^2\di^4\mu.
\end{align}
Similarly, one can define $\norm{\varphi}_{\hat{W}_{\gamma}^{\reg}(\Sigma)}^2$ on a $3$-dimensional space.
\end{enumerate}
\end{definition}

\begin{prop}
\label{prop:wave:rp}
Let $\reg\in \mathbb{N}$, $\abs{s}\leq 2$\footnote{Although only the Maxwell field $(s=\pm 1)$ is considered in this work, this proposition applies to general half integer spin weight $s$ which corresponds to the extreme components of different spin fields.}, and $p\in [0,2]$. Let $\delta\in (0,1/2)$ be arbitrary. Let $\varphi$ and $\vartheta=\vartheta(\varphi)$ be spin weight $s$ scalars satisfying
\begin{align}
\label{eq:wave:rp}
\squareShat_{s}\varphi -b_V V\varphi -b_{\phi}\Leta\varphi -b_0\varphi=\vartheta.
\end{align}
Let $b_{V}$, $b_{\phi}$ and $b_0$ be smooth real functions of $r$ such that
\begin{enumerate}
\item $\exists b_{V,-1}\geq 0$ such that $b_V=b_{V,-1} r +O(1)$,
\item $b_{\phi}=O(r^{-1})$, and
\item $\exists b_{0,0}\in \mathbb{R}$ such that $b_0=b_{0,0}+O(r^{-1})$ and $b_{0,0}+s+\abs{s}\geq 0$.
\end{enumerate}
Then there are constants $\hat{R}_0=\hat{R}_0(p,b_0,b_{\phi},b_V)$, $c= c(p,\hat{R}_0,b_0,b_{\phi},b_V)$ and $C=C(p,\hat{R}_0,b_0,b_{\phi},b_V)$ such that for all $R_0\geq \hat{R}_0$ and $\tb_2>\tb_1\geq \tb_0$,
\begin{enumerate}
\item for $p\in (0, 2)$,
    \begin{align}\label{eq:rp:less2}
\hspace{4ex}&\hspace{-4ex}
\norm{rV\varphi}^2_{W_{p-2}^\reg(\Sigmatwo^{R_0})}
+\norm{\varphi}^2_{W_{-2}^{\reg+1}(\Sigmatwo^{R_0})}
+\norm{\varphi}^2_{W_{p-3}^{\reg+1}(\Donetwo^{R_0})}
+\norm{Y\varphi}^2_{W_{-1-\delta}^{\reg}(\Donetwo^{R_0})}
+\norm{\varphi}_{F^{\reg+1}(\Scrionetwo)}^2\notag\\
\leq {}&C\bigg(\norm{rV\varphi}^2_{W_{p-2}^\reg(\Sigmaone^{R_0-M})}
+\norm{\varphi}^2_{W_{-2}^{\reg+1}(\Sigmaone^{R_0-M})}
+\norm{\varphi}^2_{W_{0}^{\reg+1}(\Sigmatwo^{R_0-M,R_0})}\bigg.\notag\\
&\bigg.\quad
+\norm{\varphi}^2_{W_{0}^{\reg+1}(\Donetwo^{R_0-M,R_0})}
+\norm{\vartheta}^2_{W_{p-3}^{\reg}(\Donetwo^{R_0-M})}
\bigg);
\end{align}
\item\label{pt:2:prop:wave:rp} for $p=2$ and $b_{0,0}+s+\abs{s}> 0$,
\begin{align}\label{eq:rp:p=2}
\hspace{4ex}&\hspace{-4ex}
c\Big(\norm{rV\varphi}^2_{W_{0}^\reg(\Sigmatwo^{R_0})}
+\norm{\varphi}^2_{W_{-2}^{\reg+1}(\Sigmatwo^{R_0})}
+\norm{\varphi}^2_{W_{-1-\delta}^{\reg+1}(\Donetwo^{R_0})}
+\norm{rV\varphi}^2_{W_{-1}^{\reg}(\Donetwo^{R_0})}
+\norm{\varphi}_{{W}^{\reg+1}(\Scrionetwo)}^2\Big)
\notag\\
\leq {}&C\bigg(\norm{rV\varphi}^2_{W_{0}^\reg(\Sigmaone^{R_0-M})}
+\norm{\varphi}^2_{W_{-2}^{\reg+1}(\Sigmaone^{R_0-M})}
+\norm{\varphi}^2_{W_{0}^{\reg+1}(\Sigmatwo^{R_0-M,R_0})}\bigg.\notag\\
&\bigg.\quad +\norm{\varphi}^2_{W_{0}^{\reg+1}(\Donetwo^{R_0-M,R_0})}
+\norm{\vartheta}^2_{W_{-1-\delta}^{\reg}(\Donetwo^{R_0-M})}\bigg)\notag\\
&
+Ca^2\int_{\tb_1}^{\tb_2}
\tb^{1+\delta}
\norm{\Lxi\varphi}^2_{W_{-2}^{\reg+1}(\Sigmatb^{R_0})}
\di \tb
- \sum_{\abs{\mathbf{a}}\leq \reg}\int_{\Donetwo^{R_0}}\Re\Big(V\overline{\mathbb{D}_1^{\mathbf{a}}\varphi}
\mathbb{D}_1^{\mathbf{a}}\vartheta\Big) \di^4 \mu.
\end{align}
\item for $p=2$ and $b_{0,0}+s+\abs{s}=0$,
\begin{align}\label{eq:rp:p=2:2}
\hspace{4ex}&\hspace{-4ex}
c\Big(\norm{rV\varphi}^2_{W_{0}^\reg(\Sigmatwo^{R_0})}
+\norm{\varphi}^2_{W_{-2}^{\reg+1}(\Sigmatwo^{R_0})}
+\norm{\varphi}^2_{W_{-1-\delta}^{\reg+1}(\Donetwo^{R_0})}
+\norm{rV\varphi}^2_{W_{-1}^{\reg}(\Donetwo^{R_0})}
+\norm{\varphi}_{\tilde{W}^{\reg+1}(\Scrionetwo)}^2\Big)
\notag\\
\leq {}&C\bigg(\norm{rV\varphi}^2_{W_{0}^\reg(\Sigmaone^{R_0-M})}
+\norm{\varphi}^2_{W_{-2}^{\reg+1}(\Sigmaone^{R_0-M})}
+\norm{\varphi}^2_{W_{0}^{\reg+1}(\Sigmatwo^{R_0-M,R_0})}
\bigg.\notag\\
&\bigg.\quad +\norm{\varphi}^2_{W_{0}^{\reg+1}(\Donetwo^{R_0-M,R_0})}
+\norm{\vartheta}^2_{W_{-1-\delta}^{\reg}(\Donetwo^{R_0-M})}\bigg)\notag\\
&
+Ca^2\int_{\tb_1}^{\tb_2}
\tb^{1+\delta}
\norm{\Lxi\varphi}^2_{W_{-2}^{\reg+1}(\Sigmatb^{R_0})}
\di \tb
- \sum_{\abs{\mathbf{a}}\leq \reg}\int_{\Donetwo^{R_0}}\Re\Big(V\overline{\mathbb{D}_2^{\mathbf{a}}\varphi}
\mathbb{D}_2^{\mathbf{a}}\vartheta\Big) \di^4 \mu.
\end{align}
\end{enumerate}
\end{prop}

\begin{proof}
The case $p\in(0,2)$ has been proven in \cite[Lemmas 5.5 and 5.6]{andersson2019stability}. Note that as discussed in Remark \ref{rem:squareShatandLs}, there is a sign difference between the operator $\squareShat$ in this work with the one in \cite{andersson2019stability}, and this is responsible for the sign change in some terms in \eqref{eq:wave:rp}.

For $p=2$ and $b_{0,0}+s+\abs{s}> 0$, we consider only $\reg=0$ case, since $\reg\in \mathbb{Z}^+$ case can be shown following the same argument in \cite[Lemma 5.6]{andersson2019stability} by commuting the wave equation with $\mathbb{D}_1$.
Define $c_0=s^2\frac{2Mr^3+a^2r^2+4a^2Mr+a^4}{(\R)^2}$, then $c_0= O(r^{-1})$. Equation \eqref{eq:wave:rp} can be expanded utilizing equation \eqref{eq:squareShat} as
\begin{align}
\label{eq:wave:generalrp}
\vartheta={}&-(\R)YV\varphi
+(2\edthR\edthR'+s+\abs{s})\varphi
-b_{V,-1} rV\varphi
-(b_{0,0}+s+\abs{s})\varphi\notag\\
&
+2a\Lxi\Leta\varphi
+a^2 \sin^2 \theta\Lxi^2\varphi
-2ias\cos\theta \Lxi\varphi\notag\\
&
-(b_{V}-rb_{V,-1})V\varphi
+\big(\tfrac{2ar}{\R}-b_{\phi}\big)\Leta\varphi
-(b_0-b_{0,0}+c_0)\varphi.
\end{align}
We multiply equation \eqref{eq:wave:generalrp} by $-2\chi^2 V\overline{\varphi}$, take the real part, and obtain
\begin{align}
\label{eq:wavetimesrp=2}
&-4\Re(\edthR(\edthR'\varphi \chi^2 V\overline{\varphi}))
+V(\chi^2(2\abs{ \edthR'\varphi}^2-(s+\abs{s})\varphi^2
+(b_{0,0}+s+\abs{s})\abs{\varphi}^2))
\notag\\
&+Y((\R)\chi^2\abs{V\varphi}^2)
+(\partial_r(\chi^2(\R))+2\chi^2 rb_{V,-1})\abs{V\varphi}^2\notag\\
&-\tfrac{2\Delta}{\R}\partial_r(|\chi |^2)
(\abs{\edthR'\varphi}^2+b_{0,0}\abs{\varphi}^2)-F_{\varphi}
=-2\chi^2 \Re(V\overline{\varphi}\vartheta),
\end{align}
where $
F_{\varphi}=2\chi^2 \Re( V\overline{\varphi}\times \text{last two lines of } \eqref{eq:wave:generalrp})$.
By integrating \eqref{eq:wavetimesrp=2}
over $\Donetwo$ with the volume element $\di^4\mu$ and from point \ref{prop:basicesti:pt2} of Proposition \ref{prop:basicesti}, one finds the integrals of the first two lines in \eqref{eq:wavetimesrp=2} give positive contribution of $c(
\norm{rV\varphi}^2_{W_{0}^0(\Sigmatwo^{R_0})}
+\norm{rV\varphi}^2_{W_{-1}^{0}(\Donetwo^{R_0})})$, the integral of the first term in the third line is supported in $r\in [R_0-M,R_0]$, and an application of Cauchy-Schwarz implies the integrals of all the sub-terms in $F_{\varphi}$ coming from the last line of \eqref{eq:wave:generalrp} are bounded by
$C\big(\norm{rV\varphi}^2_{W_{-2}^{0}(\Donetwo^{R_0-M})}+
\norm{\varphi}^2_{W_{-2}^{1}(\Donetwo^{R_0-M})}\big)$.
The integral of the remaining sub-terms in $F_{\varphi}$ equals
\begin{align}
\label{eq:wave:p=2:error123}
\int_{\Donetwo}
\left(-4\chi^2\Re(V\overline{\varphi}ia\cos \theta \Lxi\varphi)
+4a\chi^2\Re(V\overline{\varphi}
\Lxi\Leta\varphi)
+2a^2\sin^2\theta\chi^2\Re(V\overline{\varphi}
\Lxi^2\varphi)\right)\di^4 \mu
\end{align}
We separate this integral domain into $\Donetwo^{R_0-M,R_0}$ and $\Donetwo^{R_0}$, and find the integral  over $\Donetwo^{R_0-M,R_0}$ are easily seen to be bounded by the RHS of \eqref{eq:rp:p=2:2}, while for the integrals over $\Donetwo^{R_0}$ we apply Cauchy-Schwarz to bound them by
\begin{align}
\label{eq:generalrp:troubleterm:1}
\hspace{4ex}&\hspace{-4ex}
\veps\int_{\tb_1}^{\tb_2}\frac{1}{\tb^{1+\delta}}\int_{\Sigmatb^{R_0}}
\abs{rV\varphi}^2\di^4\mu
+\frac{Ca^2}{\veps}\int_{\tb_1}^{\tb_2}
\tb^{1+\delta}\int_{\Sigmatb^{R_0}}
r^{-2}\left(\abs{\Lxi\varphi}^2
+\abs{
\Lxi^2\varphi}^2
+\abs{
\Lxi\Leta\varphi}^2\right)\di^4\mu\notag\\
\lesssim{}&
\veps\sup_{\tb\in[\tb_1,\tb_2]}\int_{\Sigmatb^{R_0}}
\abs{rV\varphi}^2\di^3\mu
+\frac{a^2}{\veps}\int_{\tb_1}^{\tb_2}
\tb^{1+\delta}\int_{\Sigmatb^{R_0}}
r^{-2}\left(\abs{\Lxi\varphi}^2
+\abs{
\Lxi^2\varphi}^2
+\abs{
\Lxi\Leta\varphi}^2\right)\di^4\mu.
\end{align}
By taking $\veps$ small enough, the first term above is absorbed by the LHS of the integral of \eqref{eq:wavetimesrp=2}. Adding additionally a sufficiently large multiple of the estimate \eqref{eq:rp:less2} with $p=2-\delta$ to the integral form of \eqref{eq:wavetimesrp=2} such that one can absorb $C\big(\norm{rV\varphi}^2_{W_{-2}^{0}(\Donetwo^{R_0})}+
\norm{\varphi}^2_{W_{-2}^{1}(\Donetwo^{R_0})}\big)$, the estimate \eqref{eq:rp:p=2} is proved.

For $p=2$ and $b_{0,0}+s+\abs{s}= 0$, we consider $\reg=0$ case first.  The way of arguing is the same as the case $p=2$ and $b_{0,0}+s+\abs{s}> 0$ and the integral \eqref{eq:wave:p=2:error123} can be similarly treated.
The estimate \eqref{eq:rp:p=2:2} is thus proved.

For $p=2$, $b_{0,0}+s+\abs{s}= 0$ and general $\reg\geq 1$ case, it is enough to treat $s\leq 0$ case, the $s>0$ being analogous.
We shall prove below that
\begin{align}\label{eq:rp:p=2:2:v2}
\hspace{4ex}&\hspace{-4ex}
c\sum_{\abs{\mathbf{a}}\leq \reg}
\Big(\norm{rV \mathbb{X}^{\mathbf{a}}\varphi}^2_{W_{0}^0(\Sigmatwo^{R_0})}
+\norm{\mathbb{X}^{\mathbf{a}}\mathbb{D}_2\varphi}^2_{W_{-2}^{0}(\Sigmatwo^{R_0})}
+\norm{rV\mathbb{X}^{\mathbf{a}}\varphi}^2_{W_{-1}^{0}(\Donetwo^{R_0})}
+\norm{\mathbb{X}^{\mathbf{a}}\mathbb{D}_1\varphi}^2_{W_{-1-\delta}^{0}(\Donetwo^{R_0})}\Big)
\notag\\
\leq {}&C\sum_{\abs{\mathbf{a}}\leq \reg}\bigg(\norm{rV\mathbb{X}^{\mathbf{a}}\varphi}^2_{W_{0}^0(\Sigmaone^{R_0})}
+\norm{\mathbb{X}^{\mathbf{a}}\mathbb{D}_2\varphi}^2_{W_{-2}^{0}(\Sigmaone^{R_0})}
+\norm{\mathbb{X}^{\mathbf{a}}\vartheta}^2_{W_{-1-\delta}^{0}(\Donetwo^{R_0-M})}\notag\\
&\quad +\norm{\mathbb{X}^{\mathbf{a}}\mathbb{D}_1\varphi}^2_{W_{0}^{0}(\Donetwo^{R_0-M,R_0})}
+\sum_{\tb\in\{\tb_1,\tb_2\}}
\norm{\mathbb{X}^{\mathbf{a}}\mathbb{D}_1\varphi}^2_{W_{0}^{0}(\Sigmatb^{R_0-M,R_0})}
\bigg)\notag\\
&
+Ca^2\int_{\tb_1}^{\tb_2}
\tb^{1+\delta}
\norm{\Lxi\varphi}^2_{W_{-2}^{\reg+1}(\Sigmatb^{R_0})}
\di \tb\notag\\
&- \sum_{\abs{\mathbf{a}}\leq \reg}\int_{\Donetwo^{R_0}}\Re\Big(V\overline{\mathbb{X}^{\mathbf{a}}\varphi}
\mathbb{X}^{\mathbf{a}}\vartheta \Big)\di^4 \mu
\end{align}
for $\mathbb{X}=\mathbb{D}_2$. Tthe estimate \eqref{eq:rp:p=2:2} then clearly follows from point \ref{prop:basicesti:pt2} of Proposition \ref{prop:basicesti} and an inequality from the Hardy's inequality \eqref{eq:HardyIneqLHSRHS}
\begin{align}
\norm{\varphi}^2_{W_{-2}^{\reg+1}(\Sigmatwo^{R_0})}
\lesssim_{\reg,s}{}
\sum_{1\leq\abs{\mathbf{a}}\leq \reg+1}\norm{\mathbb{D}_2^{\mathbf{a}}\varphi}^2_{W_{-2}^{0}(\Sigmatwo^{R_0})}
+\norm{\varphi}^2_{W_{-2}^{0}(\Sigmatwo^{R_0-M, R_0})}.
\end{align}
Let $\mathbb{X}_1=\{\Lxi\}$ and
$\mathbb{X}_2=\{\Lxi,\edthR'\}$. We divide the proof into four steps. First, the case $\mathbb{X}=\mathbb{X}_1$ manifestly holds true since the Killing vector $\Lxi$ commutes with the wave equation. Second, consider $\mathbb{X}=\mathbb{X}_2$.
Commuting equation \eqref{eq:wave:rp} with $\edthR'$ gives a wave equation of $\edthR'\varphi$ in the form of \eqref{eq:wave:rp} but with
\begin{align}
\label{eq:commutator:edthR'withageneralwave}
&b_{V,\edthR'}={}b_{V}, \quad b_{\phi,\edthR'}=b_{\phi}, \quad b_{0,\edthR'}= b_0-2(s-1),\notag\\
&\vartheta_{\edthR'}
\triangleq{}\vartheta(\edthR'\varphi)={}
\edthR'\vartheta
+\frac{1}{\sqrt{2}}\left(\partial_{\theta}(a^2 \sin^2\theta) \Lxi^2\varphi -\partial_{\theta}(2ias\cos\theta)\Lxi\varphi\right).
\end{align}
All the assumptions are satisfied, and in particular, it holds that $b_{0,\edthR',0} + (s-1)+\abs{s-1}=b_{0,0}-2s+2=-2(s-1)>0$ such that the estimate \eqref{eq:rp:p=2} can be applied to $\edthR'\varphi$. We are thus left to estimate $\int_{\Donetwo^{R_0-M}}-2\chi^2 \Re(V\overline{\edthR'\varphi}\vartheta_{\edthR'})\di^4\mu$, which is bounded by
\begin{align}
\int_{\Donetwo^{R_0-M}}-2\chi^2 \Re(V\overline{\edthR'\varphi}\edthR'\vartheta)\di^4\mu
+C\abs{a}\int_{\Donetwo^{R_0-M}}\Big(\Big|V\overline{\edthR'\varphi}
\Lxi^2\vartheta\Big|+
\Big|V\overline{\edthR'\varphi}
\Lxi\vartheta\Big|\Big)\di^4\mu.
\end{align}
The second integral can be similarly estimated as \eqref{eq:generalrp:troubleterm:1}, and this proves the estimate \eqref{eq:rp:p=2:2} in the case that $\mathbb{X}=\mathbb{X}_2$ with only one angular derivative. One can increase the number of angular derivative by iterating the above steps. Third, consider the case $\mathbb{X}=\mathbb{D}_2$ with at most one $rV$ derivative. We commute equation \eqref{eq:wave:rp} with $rV$ and obtain
\begin{align}
\label{eq:coeffis:commuterV}
&b_{V,rV}={}b_V+\tfrac{2(\R)}{r}, \quad b_{\phi, rV}= b_{\phi}-\tfrac{4ar}{\R}, \quad b_{0, rV}=b_0+1,\notag\\
&\vartheta_{rV}\triangleq {}\vartheta(rV\varphi)= rV\vartheta
-\tfrac{(r^2-a^2)(\Delta-2Mr)}{\R} YV\varphi
-\tfrac{r\Delta}{2(\R)}\partial_r (b_{\phi} +c_{\phi})\Leta\varphi\notag\\
&\qquad -\Big(\tfrac{r\Delta}{2(\R)}\partial_r (r^{-1}b_V)
 +\tfrac{4Mr}{\R}-\tfrac{2r^2+a^2}{r^2}\Big)
rV\varphi
-\tfrac{r\Delta}{2(\R)}\partial_r(b_0+c_0)\varphi.
\end{align}
Note that all the assumptions are satisfied as well and, moreover, $b_{0,rV}+s+\abs{s}>0$.
By isolating $YV\varphi$ from the wave equation \eqref{eq:wave:generalrp}, one can rewrite it as a weighted sum with $O(1)$ coefficients of terms $\edthR\edthR'\varphi$, $rV\varphi$, $a\Lxi\Leta\varphi$, $a^2\Lxi^2\varphi$, $a\Lxi\varphi$, $r^{-1}\Leta\varphi$ and $r^{-1}\varphi$, hence $\vartheta_{rV}-rV\vartheta$ is also a weighted sum of those terms with $O(1)$ coefficients. One can then expand the integral term
$- \sum\limits_{\abs{\mathbf{a}}\leq \reg-1}\int_{\Donetwo^{R_0}}
\Re\big(V\overline{\mathbb{X}^{\mathbf{a}}(rV\varphi)}
\mathbb{X}^{\mathbf{a}}\vartheta_{rV} \big)\di^4 \mu$ into a sum of sub-integrals of the above mentioned terms to achieve bounds. The sub-integrals from $a\Lxi\Leta\varphi$, $a\Lxi\varphi$ and $a^2\Lxi^2\varphi$ can be bounded as \eqref{eq:generalrp:troubleterm:1}, the sub-integrals from $r^{-1}\Leta\varphi$, $rV\varphi$ and $r^{-1}\varphi$ can be easily bounded by Cauchy-Schwarz and applying the estimates from the second step, and the remaining sub-integral from $\edthR\edthR'\varphi$ is manifestly bounded after integration by parts in $V$ and $\edthR$. This then closes the proof in the case $\mathbb{X}=\mathbb{D}_2$ with at most one $rV$ derivative. Fourth, consider the general case $\mathbb{X}=\mathbb{D}_2$. In view of the third step, this requires to commute $rV$ more times, and it follows from iterating the discussions in the third step.
\end{proof}

In the case that the scalar $\varphi$ is supported on $\ell\geq \ell_0$ modes, one can modify the assumptions in Proposition \ref{prop:wave:rp} and obtain the following statements.
\begin{prop}
\label{prop:wave:rp:highmodes}
Let $\ell_0\geq \abs{s}$, and let $\varphi$ be a spin weight $s$ scalar supported on $\ell\geq \ell_0$ modes. Let the same assumptions in Proposition \ref{prop:wave:rp} hold true except that the third assumption is replaced by an assumption that there exists $b_{0,0}\in \mathbb{R}$ such that $b_0=b_{0,0}+M O(r^{-1})$ and $b_{0,0}+(\ell_0+s)(\ell_0-s+1)\geq 0$. Then,
\begin{enumerate}
\item\label{pt:rp:highmodes:1} the estimate \eqref{eq:rp:less2} holds;
 \item\label{pt:rp:highmodes:2} the estimate \eqref{eq:rp:p=2} is valid for $p=2$ and $b_{0,0}+(\ell_0+s)(\ell_0-s+1)> 0$;
\item\label{pt:rp:highmodes:3} for $p=2$, $b_{0,0}+(\ell_0+s)(\ell_0-s+1)= 0$, and $\reg$ an even, non-negative integer,
\begin{align}\label{eq:rp:p=2:2:Highmodes}
\hspace{4ex}&\hspace{-4ex}
c\Big(\norm{rV\varphi}^2_{W_{0}^\reg(\Sigmatwo^{R_0})}
+\norm{\varphi}^2_{W_{-2}^{\reg+1}(\Sigmatwo^{R_0})}
+\norm{\varphi}^2_{W_{-1-\delta}^{\reg+1}(\Donetwo^{R_0})}
+\norm{rV\varphi}^2_{W_{-1}^{\reg}(\Donetwo^{R_0})}\Big)
\notag\\
\leq {}&C\bigg(\norm{rV\varphi}^2_{W_{0}^\reg(\Sigmaone^{R_0-M})}
+\norm{\varphi}^2_{W_{-2}^{\reg+1}(\Sigmaone^{R_0-M})}
+\norm{\varphi}^2_{W_{-2}^{\reg+1}(\Sigmatwo^{R_0-M,R_0})}\bigg.\notag\\
&\bigg.\quad +\norm{\varphi}^2_{W_{0}^{\reg+1}(\Donetwo^{R_0-M,R_0})}
+\norm{\vartheta}^2_{W_{-1-\delta}^{\reg}(\Donetwo^{R_0-M})}\bigg)\notag\\
&
+Ca^2\int_{\tb_1}^{\tb_2}
\tb^{1+\delta}
\norm{\Lxi\varphi}^2_{W_{-2}^{\reg+1}(\Sigmatb^{R_0})}
\di \tb\notag\\
&- \sum_{\abs{\mathbf{n}}\leq \reg}\int_{\Donetwo^{R_0}}\Re\Big(V\overline{\KDeri_2^{\mathbf{n}}\varphi}
\KDeri_2^{\mathbf{n}}\vartheta\Big) \di^4 \mu.
\end{align}
\end{enumerate}
\end{prop}

\begin{proof}
The entire proof is similar to the one in Proposition \ref{prop:wave:rp}, and we only give necessary remarks.

In the case of Point \ref{pt:rp:highmodes:1} where $p \in (0,2)$, recall that in the proof of \cite[Lemmas 5.5 and 5.6]{andersson2019stability}, one of the main facts used is that the eigenvalues of the operator $2\edthR\edthR' +s+\abs{s}$ acting on spin weight $s$ scalar $\varphi$ are non-positive. This enables one to obtain non-negative contribution of both the energy at $\Sigmatwo^{R_0}$ and the integral on $\Donetwo^{R_0}$ after utilizing the multiplier and integrating over $\Donetwo$. Assume the scalar $\varphi$ is supported on $\ell\geq \ell_0$ modes, and in order to prove Point \ref{pt:rp:highmodes:1}, it suffices to show the eigenvalues of the operator $2\edthR\edthR' +(\ell_0+s)(\ell_0-s+1)$ are non-positive, a fact which follows from \eqref{eq:ellip:highermodes}.

In the case of Point \ref{pt:rp:highmodes:2}, the eigenvalues of the operator $2\edthR\edthR' +(\ell_0+s)(\ell_0-s+1)$ are strictly negative, hence the same way of arguing as in Point \ref{pt:2:prop:wave:rp} of Proposition \ref{prop:wave:rp} applies and yields the desired estimate.

In the case of Point \ref{pt:rp:highmodes:3}, the $\reg=0$ case of inequality \eqref{eq:rp:p=2:2:Highmodes} is straightforward. To show $\reg=2$ case, one finds first this estimate holds for $\Lxi\varphi$, $\Lxi^2\varphi$ and $\TAO \varphi$ since the Killing vector or tensors $\Lxi$, $\Lxi^2$ and $\TAO$ commute with the wave equation. By commuting with $rV$, one obtains the functions $b_{V,rV}$ $b_{\phi,rV}$, $b_{0,rV}$ and $\vartheta_{rV}$ as in \eqref{eq:coeffis:commuterV}. This then falls into the case in Point \ref{pt:rp:highmodes:2}, and the estimate \eqref{eq:rp:p=2} applies to $rV\varphi$. The way of estimating the error terms arising from $\vartheta_{rV}$ is exactly the same as the one in the discussions below equation \eqref{eq:coeffis:commuterV}. Based on this, it is manifest that one can commute further with $\Lxi$ to obtain estimates for $\Lxi(rV)$. One can also achieve estimates for $\edthR' (rV)$ since as can be seen from \eqref{eq:commutator:edthR'withageneralwave}, commuting $\edthR'$ with the equation of $rV\varphi$ only introduces a new error term of $\frac{1}{\sqrt{2}}\left(\partial_{\theta}(a^2 \sin^2\theta) \Lxi^2(rV\varphi ) -\partial_{\theta}(2ias\cos\theta)\Lxi(rV\varphi)\right)$  which can be bounded by the above estimates.
Moreover, repeating the discussions of commuting $rV$, we obtain the estimate for $(rV)^2\varphi$. These together prove the $\reg=2$ case. Iterating the above discussions yields the general case where $\reg$ is an even, non-negative integer.
\end{proof}


\section{BEAM condition implies basic energy 2-decay condition}


Theorem \ref{thm:1}, which says BEAM condition implies basic energy 2-decay condition in a subextremal Kerr spacetime, is proven for  both spin $\pm 1$ components in this section. It follows from Propositions \ref{prop:BEDC:Phiplus} and \ref{prop:BEDC:Psiminus} below.
\subsection{Spin $+1$ component}


\begin{lemma}
Let
\begin{align}
\label{def:Phi+1}
\Phiplus={}&\Delta^{-1}(\R)\Psiplus.
\end{align}
It satisfies a wave equation
\begin{align}\label{eq:Phi+1}
\squareShat_{+1}\Phiplus={}&
\tfrac{2(r^3-3Mr^2+a^2r+a^2M)}{\Delta}
{V\Phiplus}
-\tfrac{4ar}{\R}\Leta\Phiplus\notag\\
&
+\Big(-2+\tfrac{10Mr^3+2a^2r^2-22a^2Mr+2a^4}{(\R)^2}
+\tfrac{a^4(M^2-a^2)}{\Delta(\R)^2}\Big)\Phiplus.
\end{align}
\end{lemma}
\begin{remark}
One can relate this scalar with $\phi_{+1}^0$ in \cite{Ma2017Maxwell} by $\Phiplus= \Delta^{-1}(\R)^{5/2}\phi_{+1}^0$.
\end{remark}
\begin{proof}
By taking $\delta=0$ in \cite[Equation (3.10)]{Ma2017Maxwell}, one finds $\grave{\phi}_{+1}^0$ defined therein is equal to $(\R)^{-1/2}\Phiplus$ and the equation \cite[Equation (3.10)]{Ma2017Maxwell} reduces to
\begin{align*}
\hspace{4ex}&\hspace{-4ex}\left(\Sigma \Box_g+\tfrac{2i\cos\theta}{\sin^2 \theta}\Leta-\cot^2 \theta+1-2ia\cos \theta \Lxi+\tfrac{4ar}{\R}\partial_{\phi}\right)
((\R)^{-1/2}\Phiplus)\notag\\
={}&\tfrac{2(r^3-3Mr^2+a^2r+a^2M)}{\Delta(\R)^{1/2}}
{V\Phiplus}
+\Big(\tfrac{10Mr^3+2a^2r^2-14a^2Mr+2a^4}{(\R)^{5/2}}
+\tfrac{a^4(M^2-a^2)}{\Delta(\R)^{5/2}}\Big)\Phiplus.
\end{align*}
Expanding the wave operator gives
\begin{align}\label{eq:Phi+1wave}
0={}&(2\edthR\edthR' +2)\Phiplus-(\R) YV\Phiplus
-\tfrac{2(r^3-3Mr^2+a^2r+a^2M)}{\Delta}
{V\Phiplus}\notag\\
&-2ia\cos \theta \Lxi\Phiplus
+2a\Lxi\Leta\Phiplus+a^2 \sin^2 \theta\Lxi^2\Phiplus
\notag\\
&
+\tfrac{6ar}{\R}\Leta\Phiplus
-\Big(\tfrac{12Mr^3+3a^2r^2-18a^2Mr+3a^4}{(\R)^2}
+\tfrac{a^4(M^2-a^2)}{\Delta(\R)^2}\Big)\Phiplus.
\end{align}
The statement then follows from \eqref{eq:squareShat}.
\end{proof}

\subsubsection{$r^p$ estimates for rescaled spin $+1$ component}
Recall that the BEAM condition for spin $+1$ component is assumed.
\begin{definition}
\label{def:Ffts:Phiplus:-1to2}
Define for convenience that
\begin{align}
E^{\reg}_{\Sigmatb}(\Psiplus)={}\sum\limits_{\abs{\mathbf{a}}\leq \reg-2}\left(
\norm{\PDeri^{\mathbf{a}}(r^{-\delta}\psiplus)}_{W_0^1(\Sigmatb)}
+\norm{\PDeri^{\mathbf{a}}Y\psiplus}_{W_0^1(\Sigmatb)}
\right).
 \end{align}
 Define $F(\reg,p,\tb,\Psiplus)$ as follows\footnote{This corresponds to $F(i,\alpha,t)$ in \cite[Lemma 5.2]{andersson2019stability}. We make these corresponding changes since $i$ is the regularity, $\alpha$ is the $p$ weight, and $t$ is the time function.}
\begin{subequations}
\label{eq:def:Ffts:Phiplus:-1to2}
\begin{align}
\label{def:Ffts:Phiplus:1:p-10}
F(\reg,p,\tb,\Psiplus)={}&0, \qquad\text{for } p\in [-1,0),\\
\label{def:Ffts:Phiplus:1:p02}
F(\reg,p,\tb,\Psiplus)={}& \norm{rV\Psiplus}^2_{W_{p-2}^{\reg-3}(\Sigmatb)}
+(1-\delta(p))\norm{\Psiplus}^2_{W_{-2}^{\reg-2}(\Sigmatb)}
\notag\\
&
+\delta(p)\norm{\Psiplus}^2_{W_{-2-\delta}^{\reg-2}(\Sigmatb)}
+E^{\reg}_{\Sigmatb}(\Psiplus),
\quad \text{for } p\in [0,2),\\
\label{def:Ffts:Phiplus:2}
F(\reg,2,\tb,\Psiplus)={}&\norm{rV\Psiplus}^2_{W_{0}^{\reg-1}(\Sigmatb)}
+\norm{\Psiplus}^2_{W_{-2}^{\reg}(\Sigmatb)}.
\end{align}
\end{subequations}
\end{definition}
We shall now prove global $r^p$ estimates for spin $+1$ component.
\begin{prop}
\label{prop:rpplusglobal}
Let $\delta>0$ and let $\reg\geq 3$. Let $\delta_{(x)}$ be a function which equals $1$ at $x=0$ and vanishes elsewhere.
Then for any $\tb_2>\tb_1\geq \tb_0$ and $p\in [0,2)$,
\begin{align}\label{eq:rpplusglobal:less2}
F(\reg,p,\tb_2,\Psiplus)
+\norm{\Psiplus}^2_{W_{p-3}^{\reg-3}(\Donetwo)}
\leq {}&CF(\reg,p,\tb_1,\Psiplus).
\end{align}
\end{prop}

\begin{proof}
We consider only $\reg=0$ case, as the proof for $\reg\geq 1$ case is the same. We put equation \eqref{eq:Phi+1} into the form of \eqref{eq:wave:rp} and find $b_{V,-1}=2> 0$, $b_{\phi}=O(r^{-1})$ and $b_{0,0}+1+1=0$, therefore all assumptions in Proposition \ref{prop:wave:rp} are satisfied and the source term  is $\vartheta(\Psiplus)=0$.  This implies that there are constants $\hat{R}_0=\hat{R}_0(p)$ and $C=C(p,\hat{R}_0)$ such that for all $R_0\geq \hat{R}_0$, $\tb_2>\tb_1\geq \tb_0$ and $p\in (0,2)$,
\begin{align}\label{eq:rpplus:less2}
\hspace{4ex}&\hspace{-4ex}
\norm{rV\Psiplus}^2_{W_{p-2}^0(\Sigmatwo^{R_0})}
+\norm{\Psiplus}^2_{W_{-2}^{1}(\Sigmatwo^{R_0})}
+\norm{\Psiplus}^2_{W_{p-3}^{1}(\Donetwo^{R_0})}
+\norm{Y\Psiplus}^2_{W_{-1-\delta}^{0}(\Donetwo^{R_0})}
\notag\\
\leq {}&C\bigg(\norm{rV\Psiplus}^2_{W_{p-2}^0(\Sigmaone^{R_0-M})}
+\norm{\Psiplus}^2_{W_{-2}^{1}(\Sigmaone^{R_0-M})}
+\norm{\Psiplus}^2_{W_{0}^{1}(\Donetwo^{R_0-M,R_0})}\bigg).
\end{align}
The estimate \eqref{eq:rpplusglobal:less2} for $p\in (0,2)$  then follows by adding the BEAM estimates \eqref{eq:BEAM:+1} to the above estimate.
 To show the estimate \eqref{eq:rpplusglobal:less2} for $p=0$, one needs to go back to the proof of Proposition \ref{prop:wave:rp}. Similarly to the $p=2$ case, we multiply equation \eqref{eq:wave:generalrp} by $-2\chi^2 (\R)^{-1}V\overline{\Psiplus}$, take the real part and arrive at
\begin{align}
\label{eq:wavetimesrp=0}
\hspace{4ex}&\hspace{-4ex}
-4\Re(\edthR(\edthR'\Psiplus \chi^2 (\R)^{-1}V\overline{\Psiplus}))
+Y(\chi^2\abs{V\Psiplus}^2)\notag\\
\hspace{4ex}&\hspace{-4ex}
+V(\chi^2(\R)^{-1}(2\abs{ \edthR'\Psiplus}^2
-(s+\abs{s})\abs{\Psiplus}^2
+(b_{0,0}+s+\abs{s})\abs{\Psiplus}^2))
\notag\\
\hspace{4ex}&\hspace{-4ex}
+(\partial_r(\chi^2)+2\chi^2 r(\R)^{-1}b_{V,-1})\abs{V\Psiplus}^2\notag\\
\hspace{4ex}&\hspace{-4ex}-\tfrac{2\Delta}{\R}\partial_r(|\chi |^2(\R)^{-1})
(\abs{\edthR'\Psiplus}^2+b_{0,0}\abs{\Psiplus}^2)
-F_{+1}^{p=0}
={}0,
\end{align}
where $F_{+1}^{p=0}$ equals the real part of $2\chi^2 (\R)^{-1}V\overline{\Psiplus}$ times the last two lines of \eqref{eq:wave:generalrp} but with $\varphi=\Psiplus$.
By integrating \eqref{eq:wavetimesrp=0}
over $\Donetwo$ with the volume form $\di^4\mu$ and from the fact that $b_{V,-1}>0$, one finds the integral of the third line gives positive contribution of $c
\norm{rV\Psiplus}^2_{W_{-3}^{0}(\Donetwo^{R_0})}$. By an application of the Cauchy-Schwarz inequality,
\begin{align}
\label{eq:errorterm:F+1:p=0}
\hspace{4ex}&\hspace{-4ex}\int_{\Donetwo}F_{+1}^{p=0}\di^4 \mu
\lesssim{}\norm{rV\Psiplus}^2_{W_{-4}^{1}(\Donetwo^{R_0-M})}
+\norm{\Leta\Psiplus}^2_{W_{-4}^{0}(\Donetwo^{R_0-M})}\notag\\
&\qquad \qquad
+\norm{\Psiplus}^2_{W_{-4}^{0}(\Donetwo^{R_0-M})}
+\sum_{\abs{\mathbf{a}}\leq 1}\norm{\PDeri^{\mathbf{a}}\PSDeri(r^{-\delta}\psiplus)}_{W_{-1}^0(\Donetwo^{R_0-M})}.
\end{align}
Adding a large multiple of BEAM estimate \eqref{eq:BEAM:+1} to the integral form of \eqref{eq:wavetimesrp=0}, one finds the RHS of \eqref{eq:errorterm:F+1:p=0} can all be absorbed by taking $R_0$ large enough, and we are thus led to
\begin{align}
\hspace{4ex}&\hspace{-4ex}
\norm{rV\Psiplus}^2_{W_{-2}^{0}(\Sigmatwo)}
+\norm{\Psiplus}^2_{W_{-2-2\delta}^{1}(\Sigmatwo)}
+E^{3}_{\Sigmatwo}(\Psiplus)\notag\\
\hspace{4ex}&\hspace{-4ex}
+\norm{rV\Psiplus}^2_{W_{-3}^{0}(\Donetwo)}
+\norm{\Psiplus}^2_{W_{-3-2\delta}^{0}(\Donetwo)}\notag\\
\lesssim{}&\norm{rV\Psiplus}^2_{W_{-2}^{0}(\Sigmaone)}
+\norm{\Psiplus}^2_{W_{-2-2\delta}^{1}(\Sigmaone)}
+E^{3}_{\Sigmaone}(\Psiplus).
\end{align}
An application of  \cite[Lemma 4.30, point (1)]{andersson2019stability} allows the LHS to further bound over $\norm{\Psiplus}^2_{W_{-2}^{0}(\Sigmatwo)}
+\norm{\Psiplus}^2_{W_{-3}^{0}(\Donetwo)}$. In the end, making a rescaling $2\delta\mapsto \delta$ closes the proof.
\end{proof}


\subsubsection{Basic energy $2$-decay condition for spin $+1$ component}\label{sect:BED2:spin+1}


We are now ready to prove the basic energy $2$-decay condition for spin $+1$ component.
The $r^p$ estimates proven above can be interpreted as an inequality saying that for all $p\in [-1,2)$ and $\tb_2>\tb_1\geq \tb_0$,
\begin{align}
\label{eq:BEDC:Phiplus:0}
F(\reg,p,\tb_2,\Psiplus)
+\int_{\tb_1}^{\tb_2}F(\reg-3,p-1,\tb,\Psiplus)\di \tb
\lesssim F(\reg,5/3,\tb_1,\Psiplus).
\end{align}

\begin{prop}
\label{prop:BEDC:Phiplus:1}
Let $j\in \mathbb{N}$ and let $\reg\geq 6j+4$. Assume the BEAM condition to order $\reg$ is satisfied for spin $+1$ component, then for any $p\in [0,5/3]$ and any $\tb\geq\tb_0$,
\begin{align}
\label{eq:BEDC:Phiplus}
F(\reg-8j-6,p,\tb,\Lxi^j\Psiplus)\lesssim {}&\tb^{-\frac{5}{3}(1+j)+p}F(\reg,5/3,\tb_0,\Psiplus).
\end{align}
\end{prop}

\begin{proof}
An application of \cite[Lemma 5.2]{andersson2019stability} then implies for any $p\in [0,5/3]$,
\begin{align}
F(\reg-6,p,\tb,\Psiplus)\lesssim {}&\tb^{-5/3+p}F(\reg,2,\tb_0,\Psiplus).
\end{align}
This proves \eqref{eq:BEDC:Phiplus} in the case of $j=0$. To show the general $j$ case, we prove it by induction. Assume it holds true for $j$, and we consider $j+1$ case. One can utilize equation \eqref{eq:Phi+1wave} and use the replacement $Y=\frac{\R}{\Delta}\big(2\Lxi+\frac{2a}{\R}\Leta-V\big)$ away from horizon to rewrite $r^2V\Lxi \Psiplus$ as a weighted sum of $(rV)^2\Psiplus$, $\edthR\edthR'\Psiplus$,
$\Lxi^2\Psiplus$, $\Lxi\Leta\Psiplus$, $rV\Psiplus$, $r^{-1}\Leta (rV)\Psiplus$, $r^{-1}\Leta \Psiplus$, $\Lxi\Psiplus$ and $r^{-1}\Psiplus$ all with $O(1)$ coefficients. Therefore,
\begin{align}
\hspace{4ex}&\hspace{-4ex}
F(\reg-8-8j,5/3,\tb,\Lxi^{j+1}\Psiplus)\notag\\
={}&
\norm{rV\Lxi^{j+1}\Psiplus}^2_{W_{5/3-2}^{\reg-8-8j}(\Sigmatb)}
+\norm{\Lxi^{j+1}\Psiplus}^2_{W_{-2}^{\reg-7-8j}(\Sigmatb)}
+E^{\reg-5-8j}_{\Sigmatb}(\Lxi^{j+1}\Psiplus)\notag\\
\lesssim{}&\norm{r^2V\Lxi(\Lxi^{j}\Psiplus)}^2_{W_{-7/3}^{\reg-8-8j}(\Sigmatb)}
+\norm{\Lxi^{j+1}\Psiplus}^2_{W_{-2}^{\reg-7-8j}(\Sigmatb)}
+E^{\reg-4-8j}_{\Sigmatb}(\Lxi^{j}\Psiplus)\notag\\
\lesssim{}&\norm{\Lxi^{j}\Psiplus}^2_{W_{-2}^{\reg-6-8j}(\Sigmatb)}
+E^{\reg-4-8j}_{\Sigmatb}(\Lxi^{j}\Psiplus)
\lesssim{}F(\reg-6-8j,0,\tb,\Lxi^{j}\Psiplus)\notag\\
\lesssim{}&\tb^{-\frac{5}{3}(1+j)}F(\reg,5/3,\tb_0,\Psiplus).
\end{align}
Since $\Lxi$ is a symmetry operator for the TME, the estimate \eqref{eq:BEDC:Phiplus:0} is still valid if we replace $\Psiplus$ by $\Lxi^{j+1}\Psiplus$. Again, an application of \cite[Lemma 5.2]{andersson2019stability} then implies for any $p\in [0,5/3]$,
\begin{align}
F(\reg-8(j+1)-6,p,\tb,\Lxi^{j+1}\Psiplus)\lesssim{}&
\tb^{-5/3+p}F(\reg-8(j+1),5/3,\tb/2,\Lxi^{j+1}\Psiplus)\notag\\
\lesssim{}&
\tb^{-\frac{5}{3}-\frac{5}{3}(j+1)+p}F(\reg,5/3,\tb_0,\Lxi^{j}\Psiplus),
\end{align}
as claimed.
\end{proof}

\begin{prop}
\label{prop:BEDC:Phiplus:2}
Let $\delta>0$ and let $\reg\geq 3$.  There are constants $\hat{R}_0=\hat{R}_0(p)$, $C=C(p,\hat{R}_0)$ and $\regl\geq 0$ such that for all $R_0\geq \hat{R}_0$ and $\tb_2>\tb_1\geq \tb_0$,
\begin{align}\label{eq:rpplusglobal:p=2}
F(\reg,2,\tb_2,\Psiplus)
+\int_{\tb_1}^{\tb_2}F(\reg-2,1,\tb,\Psiplus)\di \tb
\lesssim F(\reg+\regl,2,\tb_1,\Psiplus).
\end{align}
Moreover, the above estimate holds true if replacing $\Psiplus$ by $\Lxi^j\Psiplus$ for any $j\in \mathbb{Z}^+$.
\end{prop}

\begin{proof}
From Proposition \ref{prop:wave:rp},  there are constants $\hat{R}_0=\hat{R}_0(p)$ and $C=C(p,\hat{R}_0)$ such that for all $R_0\geq \hat{R}_0$ and $\tb_2>\tb_1\geq \tb_0$,
\begin{align}\label{eq:rpplus:p=2}
\hspace{4ex}&\hspace{-4ex}
\norm{rV\Psiplus}^2_{W_{\reg}^0(\Sigmatwo^{R_0})}
+\norm{\mathbb{D}_2\Psiplus}^2_{W_{-2}^{\reg}(\Sigmatwo^{R_0})}
+\norm{\Psiplus}^2_{W_{-1-\delta}^{\reg+1}(\Donetwo^{R_0})}
+\norm{rV\Psiplus}^2_{W_{-1}^{\reg}(\Donetwo^{R_0})}
\notag\\
\leq {}&C\bigg(\norm{rV\Psiplus}^2_{W_{0}^{\reg}(\Sigmaone^{R_0-M})}
+\norm{\mathbb{D}_2\Psiplus}^2_{W_{-2}^{\reg}(\Sigmaone^{R_0-M})}
+\norm{\Psiplus}^2_{W_{0}^{\reg+1}(\Donetwo^{R_0-M, R_0})}\notag\\
&\quad
+\sum_{\tb\in\{\tb_1,\tb_2\}}\norm{\Psiplus}^2_{W_{0}^{\reg+1}
(\Sigmatb^{R_0-M, R_0})}
+\int_{\tb_1}^{\tb_2}
\tb^{1+\delta}
\norm{\Lxi\Psiplus}^2_{W_{-2}^{\reg+1}(\Sigmatb^{R_0})}
\di \tb\bigg).
\end{align}
Adding the BEAM estimates \eqref{eq:BEAM:+1} to the above gives
\begin{align}\label{eq:rpplusglobal:p=2:mid}
\hspace{4ex}&\hspace{-4ex}
\norm{rV\Psiplus}^2_{W_{0}^\reg(\Sigmatwo)}
+\norm{\Psiplus}^2_{W_{-2}^{\reg+1}(\Sigmatwo)}
+\norm{\Psiplus}^2_{W_{-1-\delta}^{\reg}(\Donetwo)}
+\norm{rV\Psiplus}^2_{W_{-1}^{\reg-1}(\Donetwo)}
+E^{\reg+3}_{\Sigmatwo}(\Psiplus)
\notag\\
\leq {}&C\bigg(\norm{rV\Psiplus}^2_{W_{0}^{\reg}(\Sigmaone)}
+\norm{\Psiplus}^2_{W_{-2}^{\reg+1}(\Sigmaone)}
+E^{\reg+3}_{\Sigmaone}(\Psiplus)
 +\int_{\tb_1}^{\tb_2}
\tb^{1+\delta}
\norm{\Lxi\Psiplus}^2_{W_{-2}^{\reg+1}(\Sigmatb^{R_0})}
\di \tb\bigg).
\end{align}
From Proposition \ref{prop:BEDC:Phiplus:1}, the last line of \eqref{eq:rpplusglobal:p=2:mid} is bounded by $\norm{rV\Psiplus}^2_{W_{0}^{\reg+\regl}(\Sigmaone)}
+\norm{\Psiplus}^2_{W_{-2}^{\reg+\regl+1}(\Sigmaone)}$ for some $\regl\geq 0$.
The estimate \eqref{eq:rpplusglobal:p=2} for $p=2$ then follows.
\end{proof}

The basic energy $2$-decay condition follows from the following result.
\begin{prop}
\label{prop:BEDC:Phiplus}
Let $j\in \mathbb{N}$ and let $\reg$ suitably large. Assume the BEAM condition to order $\reg$ is satisfied for spin $+1$ component, then there is a constant $\regl(j)$ such that for any $p\in [0,2]$ and any $\tb\geq\tb_0$,
\begin{align}
\label{eq:BEDC:Phiplus}
F(\reg-\regl(j),p,\tb,\Lxi^j\Psiplus)\lesssim {}&\tb^{-2-2j+p}F(\reg,2,\tb/2,\Psiplus)\lesssim \tb^{-2-2j+p}F(\reg,2,\tb_0,\Psiplus).
\end{align}
\end{prop}

\begin{proof}
The estimate \eqref{eq:BEDC:Phiplus:0} for $p\in [0,2)$ and the estimate \eqref{eq:rpplusglobal:p=2} for $p=2$ together imply that
for any $p\in [0,2]$,
\begin{align}
F(\reg-4,p,\tb,\Psiplus)\lesssim {}&\tb^{-2+p}F(\reg+2\regl,2,\tb_0,\Psiplus).
\end{align}
This proves the estimate \eqref{eq:BEDC:Phiplus} for $j=0$. To show general $j\geq 0$ cases, one uses an induction for $j\geq 0$. The same way of arguing in the proof of Proposition \ref{prop:BEDC:Phiplus:1} applies here after replacing $5/3$ by $2$ and $j$-dependent constants by $\regl(j)$, eventually proving the estimate \eqref{eq:BEDC:Phiplus}.
\end{proof}


\subsection{Spin $-1$ component}


\begin{definition}
\label{def:Phiminusi}
Define the operators\footnote{The operator $\VR$ in this work is the same as the operator $V$ in \cite{Ma2017Maxwell}.}
\begin{align}
\VR={}\frac{\R}{\Delta}V, \qquad \curlVR={}(\R)\VR,
\end{align}
and define scalars
\begin{subequations}
\label{eq:Phiminusi:def}
\begin{align}
\Phiminus{0}={}&\Delta/(r^2+a^2) \Psiminus, \\ \Phiminus{i}={}&\curlVR^i \Phiminus{0}, \quad \text{for}\quad i=1,2,\\
\phiminus{i}={}&(\R)^{-1/2}\Phiminus{i},\quad \text{for}\quad i=0,1,2.
\end{align}
\end{subequations}
\end{definition}

\begin{remark}\label{rem:Phiminusandphi-1}
The scalars $\phiminus{i}$ here are the same as the scalars $\phi_{-1}^i$ in \cite{Ma2017Maxwell} for $i=0,1$.
\end{remark}

We derive the wave system for these variables, which is an extended system of \cite[Equations (1.31)]{Ma2017Maxwell}.
\begin{lemma}
The scalars defined as in Definition \ref{def:Phiminusi} satisfy a wave system
\begin{subequations}
\label{eq:Phi-1012}
\begin{align}
\label{eq:Phi-10}
\squareShat_{-1}\Phiminus{0}
={}&-\tfrac{2(\PR)}{(\R)^2}\Phiminus{1}
+\tfrac{2(r^4-Mr^3+a^2r^2+3a^2Mr)}{(\R)^2}\Phiminus{0}
+\tfrac{4ar}{\R}\partial_{\phi}\Phiminus{0},\\
\label{eq:Phi-11}
\squareShat_{-1}\Phiminus{1} ={}&\tfrac{2(r^4-Mr^3+a^2r^2+3a^2Mr)}{(\R)^2}\Phiminus{1}
+\tfrac{2a^2(\PR)}{(\R)^2}\Phiminus{0}
-\tfrac{2a(r^2-a^2)}{\R}\partial_{\phi}\Phiminus{0},\\
\label{eq:Phi-12}
\squareShat_{-1}\Phiminus{2}={}&\tfrac{2(r^3-3Mr^2+a^2r+a^2M)}{\Delta}
{V\Phiminus{2}}
+\tfrac{10Mr^3+2a^2r^2-22a^2Mr+2a^4}{(\R)^2}\Phiminus{2}
-\tfrac{8ar}{\R}\partial_{\phi}\Phiminus{2}
\notag\\
&+\tfrac{4a^2(\PR)}{(\R)^2}\Phiminus{1}
\notag\\
&
-\tfrac{2ar^2}{\R}\partial_{\phi}\Phiminus{1}+(\R)\partial_r\Big(\tfrac{2a^2(\PR)}{(\R)^2}\Big)\Phiminus{0}\notag\\
&
-(\R)\partial_r\Big(\tfrac{2a(r^2-a^2)}{\R}\Big)\partial_{\phi}\Phiminus{0}.
\end{align}
\end{subequations}
\end{lemma}

\begin{proof}
The first two subequations are manifest from \cite[Equations (1.31)]{Ma2017Maxwell} in view of Remarks \ref{rem:squareShatandLs} and \ref{rem:Phiminusandphi-1}. We apply operator $\curlVR$ to equation \eqref{eq:Phi-11}, and from the commutation relation \cite[(A.1)]{Ma2017Maxwell}, Remark \ref{rem:squareShatandLs} and Definition \ref{def:squareShat}, one obtains a wave equation of $\Phiminus{2}$:
\begin{align}\label{eq:Phi-12wave}
0={}&2\edthR\edthR'\Phiminus{2}-(\R) YV\Phiminus{2}
-\tfrac{2(r^3-3Mr^2+a^2r+a^2M)}{\Delta}
{V\Phiminus{2}}\notag\\
&
+2ia\cos \theta \partial_t\Phiminus{2}
+\tfrac{10ar}{\R}\partial_{\phi}\Phiminus{2}
+\left(2a\partial_{t\phi}^2+a^2 \sin^2 \theta\partial_{tt}^2\right)\Phiminus{2}\notag\\
&-\tfrac{12Mr^3+3a^2r^2-18a^2Mr+3a^4}{(\R)^2}
\Phiminus{2}\notag\\
&-\tfrac{4a^2(\PR)}{(\R)^2}\Phiminus{1}
-(\R)\partial_r\Big(\tfrac{2a^2(\PR)}{(\R)^2}\Big)\Phiminus{0}\notag\\
&+\tfrac{2ar^2}{\R}\partial_{\phi}\Phiminus{1}
+(\R)\partial_r\Big(\tfrac{2a(r^2-a^2)}{\R}\Big)\partial_{\phi}\Phiminus{0}.
\end{align}
This is the expanded form of \eqref{eq:Phi-12} in view of equation \eqref{eq:squareShat}.
\end{proof}


\subsubsection{$r^p$ estimates for extended spin $-1$ system}


\begin{prop}
\label{prop:rpglobal:Phiminus}
Let $\mathbf{Q}_1=\{0,1\}$ and $\mathbf{Q}_2=\{0,1,2\}$, and define $l(i)=\max(0,i-1)$. Assume the BEAM condition to order $\reg+2$ for spin $-1$ component is satisfied, then  for any $j\in \{1,2\}$,
\begin{itemize}
\item for $p\in (0,2)$,
\begin{align}\label{eq:rpminusglobal:01:less2}
\hspace{4ex}&\hspace{-4ex}
\sum_{i\in \mathbf{Q}_j}\Big(\norm{rV\PsiminusHigh{i}}^2_{W_{p-2}^{\reg-l(i)}(\Sigmatwo)}
+\norm{\PsiminusHigh{i}}^2_{W_{-2}^{\reg+1-l(i)}(\Sigmatwo)}\notag\\
\hspace{4ex}&\hspace{-4ex} \qquad\quad
+\norm{\PsiminusHigh{i}}^2_{W_{p-3}^{\reg-l(i)}(\Donetwo)}
+\norm{Y\PsiminusHigh{i}}^2_{W_{-1-\delta}^{\reg-1-l(i)}(\Donetwo)}\Big)
\notag\\
\lesssim {}&\sum_{i\in \mathbf{Q}_j}\Big(\norm{rV\PsiminusHigh{i}}^2_{W_{p-2}^{\reg-l(i)}(\Sigmaone)}
+\norm{\PsiminusHigh{i}}^2_{W_{-2}^{\reg+1-l(i)}(\Sigmaone)}\Big);
\end{align}
\item for $p=2$,
\begin{align}\label{eq:rpminusglobal:01:p=2}
\hspace{4ex}&\hspace{-4ex}
\sum_{i\in \mathbf{Q}_1}\Big(\norm{rV\PsiminusHigh{i}}^2_{W_{0}^{\reg}(\Sigmatwo)}
+\norm{\PsiminusHigh{i}}^2_{W_{-2}^{\reg+1}(\Sigmatwo)}\notag\\
\hspace{4ex}&\hspace{-4ex}\qquad\quad
+\norm{\PsiminusHigh{i}}^2_{W_{-1-\delta}^{\reg}(\Donetwo)}
+\norm{rV\PsiminusHigh{i}}^2_{W_{-1}^{\reg-1}(\Donetwo)}
\Big)
\notag\\
\lesssim {}&\sum_{i\in \mathbf{Q}_1}
\bigg(\norm{rV\PsiminusHigh{i}}^2_{W_{0}^{\reg-l(i)}(\Sigmaone)}
+\norm{\PsiminusHigh{i}}^2_{W_{-2}^{\reg+1-l(i)}(\Sigmaone)}
+a^2\int_{\tb_1}^{\tb_2}
\tb^{1+\delta}
\norm{\Lxi\PsiminusHigh{i}}^2_{W_{-2}^{\reg+1}(\Sigmatb^{R_0})}
\di \tb\bigg);
\end{align}
\item for $p=2$,
\begin{align}\label{eq:rpminusglobal:012:p=2}
\hspace{4ex}&\hspace{-4ex}
\sum_{i\in \mathbf{Q}_2}\Big(\norm{rV\PsiminusHigh{i}}^2_{W_{0}^{\reg-l(i)}(\Sigmatwo)}
+\norm{\PsiminusHigh{i}}^2_{W_{-2}^{\reg+1-l(i)}(\Sigmatwo)}
\notag\\
\hspace{4ex}&\hspace{-4ex}\qquad\quad
+\norm{\PsiminusHigh{i}}^2_{W_{-1-\delta}^{\reg-l(i)}(\Donetwo)}
+\norm{rV\PsiminusHigh{i}}^2_{W_{-1}^{\reg-1-l(i)}(\Donetwo)}
\Big)
\notag\\
\lesssim {}&\sum_{i\in \mathbf{Q}_2}
\Big(\norm{rV\PsiminusHigh{i}}^2_{W_{0}^{\reg-l(i)}(\Sigmaone)}
+\norm{\PsiminusHigh{i}}^2_{W_{-2}^{\reg+1-l(i)}(\Sigmaone)}\Big)\notag\\
&
+
a^2\int_{\tb_1}^{\tb_2}
\tb^{1+\delta}
\bigg(\sum_{i\in \mathbf{Q}_2}\norm{\Lxi\PsiminusHigh{i}}^2_{W_{-2}^{\reg+1-l(i)}
(\Sigmatb)}
+
\norm{\PsiminusHigh{1}}^2_{W_{-2}^{\reg}(\Sigmatb)}
+
\norm{\PsiminusHigh{0}}^2_{W_{-2}^{\reg}(\Sigmatb)}
\bigg)
\di \tb.
\end{align}
\end{itemize}
\end{prop}

\begin{proof}
We put each subequation in system \eqref{eq:Phi-1012} into the form of \eqref{eq:wave:rp} and find
\begin{enumerate}
\item $b_{V,-1}(\Phiminus{0})=b_{V,-1}(\Phiminus{1})=0$, $b_{V,-1}(\Phiminus{2})=2\geq 0$,
\item $b_{\phi}(\Phiminus{i})=MO(r^{-1})$ for all $i\in \{0,1,2\}$,
\item $b_{0,0}(\Phiminus{0})=b_{0,0}(\Phiminus{1})=2>0$ and $b_{0,0}(\Phiminus{2})=0$.
\end{enumerate}
All the assumptions in Proposition \ref{prop:wave:rp} are satisfied and the source terms are
\begin{subequations}
\begin{align}
\vartheta(\Phiminus{0})={}&-\tfrac{2(\PR)}{(\R)^2}\Phiminus{1}
=O(r^{-1})\Phiminus{1},\\
\vartheta(\Phiminus{1})={}&\tfrac{2a^2(\PR)}{(\R)^2}\Phiminus{0}
-\tfrac{2a(r^2-a^2)}{\R}\partial_{\phi}\Phiminus{0}\notag\\
={}&-2a\partial_{\phi}\Phiminus{0}
+a^2 O(r^{-1})\Phiminus{0}
+a^3O(r^{-2})\partial_{\phi}\Phiminus{0},\\
\vartheta(\Phiminus{2})
={}&\tfrac{4a^2(\PR)}{(\R)^2}\Phiminus{1}
-(\R)\partial_r\Big(\tfrac{2a(r^2-a^2)}{\R}\Big)\partial_{\phi}\Phiminus{0}\notag\\
&-\tfrac{2ar^2}{\R}\partial_{\phi}\Phiminus{1}
+(\R)\partial_r\Big(\tfrac{2a^2(\PR)}{(\R)^2}\Big)\Phiminus{0}
\notag\\
={}&-2a\partial_{\phi}\Phiminus{1}
-2a^2\Phiminus{0}
+a^3O(r^{-2})\partial_{\phi}\Phiminus{1}\notag\\
&
+a^2O(r^{-1})\Phiminus{1} +a^3O(r^{-1})\partial_{\phi}\Phiminus{0}
+Ma^2O(r^{-1})\Phiminus{0} .
\end{align}
\end{subequations}
We first show $r^p$ estimates near infinity: For $p\in (0,2)$,
\begin{subequations}
\begin{align}\label{eq:rpminus:01:less2}
\hspace{4ex}&\hspace{-4ex}
\sum_{i\in \mathbf{Q}_j}\Big(\norm{rV\Phiminus{i}}^2_{W_{p-2}^{\reg-l(i)}(\Sigmatwo^{R_0})}
+\norm{\Phiminus{i}}^2_{W_{-2}^{\reg+1-l(i)}(\Sigmatwo^{R_0})}\notag\\
\hspace{4ex}&\hspace{-4ex}\qquad
+\norm{\Phiminus{i}}^2_{W_{p-3}^{\reg+1-l(i)}(\Donetwo^{R_0})}
+\norm{Y\Phiminus{i}}^2_{W_{-1-\delta}^{\reg-l(i)}(\Donetwo^{R_0})}\Big)
\notag\\
\lesssim {}&\sum_{i\in \mathbf{Q}_j}\bigg(\norm{rV\Phiminus{i}}^2_{W_{p-2}^{\reg-l(i)}(\Sigmaone^{R_0-M})}
+\norm{\Phiminus{i}}^2_{W_{-2}^{\reg+1-l(i)}(\Sigmaone^{R_0-M})}\bigg.\notag\\
&\bigg.\quad \qquad +\norm{\Phiminus{i}}^2_{W_{0}^{\reg+1-l(i)}(\Donetwo^{R_0-M,R_0})}
+
\norm{\Phiminus{i}}^2_{W_{0}^{\reg+1-l(i)}
(\Sigmatwo^{R_0-M,R_0})}\bigg),
\end{align}
and for $p=2$,
\begin{align}\label{eq:rpminus:01:p=2}
\hspace{4ex}&\hspace{-4ex}
\sum_{i\in \mathbf{Q}_1}\Big(\norm{rV\Phiminus{i}}^2_{W_{0}^{\reg}(\Sigmatwo^{R_0})}
+\norm{\Phiminus{i}}^2_{W_{-2}^{\reg+1}(\Sigmatwo^{R_0})}\notag\\
&\quad
+\norm{\Phiminus{i}}^2_{W_{-1-\delta}^{\reg+1}(\Donetwo^{R_0})}
+\norm{rV\Phiminus{i}}^2_{W_{-1}^{\reg}(\Donetwo^{R_0})}
\Big)
\notag\\
\lesssim {}&\sum_{i\in \mathbf{Q}_1}\bigg(\norm{rV\Phiminus{i}}^2_{W_{0}^{\reg}
(\Sigmaone^{R_0-M})}
+\norm{\Phiminus{i}}^2_{W_{-2}^{\reg+1}(\Sigmaone^{R_0-M})}\bigg.\notag\\
&\bigg.\quad \qquad +\norm{\Phiminus{i}}^2_{W_{0}^{\reg+1}(\Donetwo^{R_0-M,R_0})}
+
\norm{\Phiminus{i}}^2_{W_{0}^{\reg+1}(\Sigmatwo^{R_0-M,R_0})}\notag\\
&\bigg.\quad\qquad
+a^2\int_{\tb_1}^{\tb_2}
\tb^{1+\delta}
\norm{\Lxi\Phiminus{i}}^2_{W_{-2}^{\reg+1}(\Sigmatb^{R_0})}
\di \tb\bigg),\\
\label{eq:rpminus:012:p=2}
\hspace{4ex}&\hspace{-4ex}
\sum_{i\in \mathbf{Q}_2}\Big(\norm{rV\Phiminus{i}}^2_{W_{0}^{\reg-l(i)}(\Sigmatwo^{R_0})}
+\norm{\Phiminus{i}}^2_{W_{-1-\delta}^{\reg+1-l(i)}(\Donetwo^{R_0})}
+\norm{rV\Phiminus{i}}^2_{W_{-1}^{\reg-l(i)}(\Donetwo^{R_0})}
\Big)
\notag\\
\hspace{4ex}&\hspace{-4ex}
+\sum_{i=0,1}\norm{\PsiminusHigh{i}}^2_{W_{-2}^{\reg+1}(\Sigmatwo^{R_0})}
+\norm{\mathbb{D}_2\PsiminusHigh{2}}^2_{W_{-2}^{\reg-1}(\Sigmatwo^{R_0})}
\notag\\
\lesssim {}&
\sum_{i=0,1}\norm{\PsiminusHigh{i}}^2_{W_{-2}^{\reg+1}(\Sigmaone^{R_0-M})}
+\norm{\mathbb{D}_2\PsiminusHigh{i}}^2_{W_{-2}^{\reg-1}(\Sigmaone^{R_0-M})}\notag\\
&+\sum_{i\in \mathbf{Q}_2}\bigg(\norm{rV\Phiminus{i}}^2_{W_{0}^{\reg-l(i)}(\Sigmaone^{R_0})} +\norm{\Phiminus{i}}^2_{W_{0}^{\reg+1-l(i)}(\Donetwo^{R_0-M,R_0})}
+
\norm{\Phiminus{i}}^2_{W_{0}^{\reg+1-l(i)}(\Sigmatwo^{R_0-M,R_0})}\bigg)\notag\\
&+
a^2\int_{\tb_1}^{\tb_2}
\tb^{1+\delta}
\bigg(\sum_{i\in \mathbf{Q}_2}\norm{\Lxi\PsiminusHigh{i}}^2_{W_{-2}^{\reg+1-l(i)}(\Sigmatb^{R_0})}
+
\norm{\PsiminusHigh{1}}^2_{W_{-2}^{\reg+1}(\Sigmatb^{R_0})}
+
\norm{\PsiminusHigh{0}}^2_{W_{-2}^{\reg+1}(\Sigmatb^{R_0})}
\bigg)
\di \tb.
\end{align}
\end{subequations}
To prove these $r^p$ estimates near infinity, we shall apply the estimates in Proposition \ref{prop:wave:rp} to each subequation of system \eqref{eq:Phi-1012} and estimate the terms involving the source terms.
For $p\in (0,2)$,
\begin{subequations}
\label{eq:Phi-1012:less2:error}
\begin{align}
\norm{\vartheta(\Phiminus{0})}^2_{W_{p-3}^{\reg}(\Donetwo^{R_0-M})}
\lesssim{}&
R_0^{-2}\norm{\Phiminus{1}}^2_{W_{p-3}^{\reg}(\Donetwo^{R_0-M})},\\
\norm{\vartheta(\Phiminus{1})}^2_{W_{p-3}^{\reg}(\Donetwo^{R_0-M})}
\lesssim{}&\norm{\Phiminus{0}}^2_{W_{p-3}^{\reg+1}(\Donetwo^{R_0-M})},\\
\label{eq:Phi-12:less2:error}
\norm{\vartheta(\Phiminus{2})}^2_{W_{p-3}^{\reg-1}(\Donetwo^{R_0-M})}
\lesssim{}&\sum_{i=0,1}\norm{\Phiminus{i}}^2_{W_{p-3}^{\reg}(\Donetwo^{R_0-M})}.
\end{align}
\end{subequations}
Define for $i=0,1$ that $I(\Phiminus{i})= -\sum\limits_{\abs{\mathbf{a}}\leq \reg}\int_{\Donetwo^{R_0}}\Re\Big(V\overline{\mathbb{D}_1^{\mathbf{a}}\Phiminus{i}}
\mathbb{D}_1^{\mathbf{a}}\vartheta(\Phiminus{i})\Big) \di^4 \mu$, and define $I(\Phiminus{2})= -\sum\limits_{\abs{\mathbf{a}}\leq \reg-1}\int_{\Donetwo^{R_0}}\Re\Big(
V\overline{\mathbb{D}_2^{\mathbf{a}}\Phiminus{2}}
\mathbb{D}_2^{\mathbf{a}}\vartheta(\Phiminus{2})\Big) \di^4 \mu$.
For $p=2$, we have
\begin{subequations}
\label{eq:Phi-1012:p2:error}
\begin{align}
\label{eq:Phi-10:p2:error}
I(\Phiminus{0})\leq{}&C\veps\norm{rV\Phiminus{0}}^2_{W_{-1}^{\reg}(\Donetwo^{R_0})}
+C\veps^{-1}R_0^{-1}\norm{\Phiminus{1}}^2_{W_{-2}^{\reg}(\Donetwo^{R_0})},\\
\label{eq:Phi-11:p2:error}
I(\Phiminus{1})\leq{}&C\norm{\Phiminus{0}}^2_{W_{-3}^{\reg+1}(\Donetwo^{R_0})}
+2a\sum_{\abs{\mathbf{a}}\leq \reg}\int_{\Donetwo^{R_0}}
\Re\Big(V\overline{\mathbb{D}_1^{\mathbf{a}}\Phiminus{1}}
\mathbb{D}_1^{\mathbf{a}}\partial_{\phi}\Phiminus{0}\Big) \di^4 \mu,\\
\label{eq:Phi-12:p2:error}
I(\Phiminus{2})\leq{}&C\sum_{i=0,1}
\norm{\Phiminus{i}}^2_{W_{-3}^{\reg}(\Donetwo^{R_0})}
+2a\sum_{\abs{\mathbf{a}}\leq \reg-1}\int_{\Donetwo^{R_0}}
\Re\Big(V\overline{\mathbb{D}_2^{\mathbf{a}}\Phiminus{2}}
\mathbb{D}_2^{\mathbf{a}}\partial_{\phi}\Phiminus{1}\Big) \di^4 \mu\notag\\
&+2a^2\sum_{\abs{\mathbf{a}}\leq \reg-1}\int_{\Donetwo^{R_0}}
\Re\Big(V\overline{\mathbb{D}_2^{\mathbf{a}}\Phiminus{2}}
\mathbb{D}_2^{\mathbf{a}}\Phiminus{0}\Big) \di^4 \mu.
\end{align}
\end{subequations}
Denote the last term in \eqref{eq:Phi-11:p2:error} by $I_1(\Phiminus{1})$ and the last two terms in \eqref{eq:Phi-12:p2:error} by $I_1(\Phiminus{2})$, respectively. For the term $I_1(\Phiminus{1})$, we perform integration by parts in $V$, use the estimate in point \ref{prop:basicesti:pt1} of Proposition \ref{prop:basicesti} to commute $V$ with $\CDeri^{\mathbf{a}}$,  rewrite $\partial_{\phi} V\Phiminus{0}=\frac{\Delta}{(\R)^2}\partial_{\phi} \Phiminus{1}$ by Definition \ref{def:Phiminusi}, then perform an integration by parts in $\partial_{\phi}$, and finally use the Cauchy-Schwarz inequality, arriving at
\begin{subequations}
\label{eq:Phi-1012:p2:erroresti}
\begin{align}
\label{eq:Phi-11:p2:erroresti}
I_1(\Phiminus{1})\leq {}&C\veps\Big(\norm{\Phiminus{1}}_{{W}_{0}^{\reg+1}(\Scrionetwo)}^2
+ \norm{\Phiminus{1}}^2_{W_{-2}^{\reg+1}(\Sigmatwo^{R_0})}
+\norm{\Phiminus{1}}_{{W}^{\reg+1}(\Scrionetwo)}^2\Big)\notag\\
&
+C\veps^{-1}\Big(\norm{\Phiminus{0}}_{{W}_{0}^{\reg+1}(\Scrionetwo)}^2
+\norm{\Phiminus{0}}^2_{W_{-2}^{\reg+1}(\Sigmatwo^{R_0})}
+\norm{\Phiminus{0}}_{{W}^{\reg+1}(\Scrionetwo)}^2\Big)\notag\\
&+CR_0^{-1+\delta}\Big(\norm{\Phiminus{0}}^2_{W_{-1-\delta}^{\reg+1}(\Donetwo^{R_0})}
+\norm{\Phiminus{1}}^2_{W_{-1-\delta}^{\reg+1}(\Donetwo^{R_0})}
\Big).
\end{align}
A simple application of the Cauchy-Schwarz inequality yields
\begin{align}
\label{eq:Phi-12:p2:erroresti1}
I_1(\Phiminus{2})\lesssim{}&
R_0^{-1+\delta}\sum_{i=0,1,2}
\norm{\Phiminus{i}}^2_{W_{-1-\delta}^{\reg}(\Donetwo^{R_0})}
+\veps\int_{\tb_1}^{\tb_2}\frac{1}{\tb^{1+\delta}}
\norm{rV\Phiminus{2}}^2_{W_{0}^{\reg-1}(\Sigmatb^{R_0})}\di\tb\notag\\
&
+\frac{a^2}{\veps}\int_{\tb_1}^{\tb_2}
\tb^{1+\delta}\Big(\norm{\Phiminus{0}}^2_{W_{-2}^{\reg}(\Sigmatb^{R_0})}
+\norm{\Phiminus{1}}^2_{W_{-2}^{\reg}(\Sigmatb^{R_0})}\Big)
\di\tb\notag\\
\lesssim{}&
\veps\sup_{\tb\in[\tb_1,\tb_2]}\norm{rV\Phiminus{2}}^2_{W_{0}^{\reg-1}(\Sigmatb^{R_0})}
+R_0^{-1+\delta}\sum_{i=0,1,2}
\norm{\Phiminus{i}}^2_{W_{-1-\delta}^{\reg}(\Donetwo^{R_0})}
\notag\\
&
+\frac{a^2}{\veps}\int_{\tb_1}^{\tb_2}
\tb^{1+\delta}\Big(\norm{\Phiminus{0}}^2_{W_{-2}^{\reg}(\Sigmatb^{R_0})}
+\norm{\Phiminus{1}}^2_{W_{-2}^{\reg}(\Sigmatb^{R_0})}\Big)
\di\tb.
\end{align}
\end{subequations}

Consider first the system of equations \eqref{eq:Phi-10} and \eqref{eq:Phi-11}. For $p\in (0,2)$, one applies Proposition \ref{prop:wave:rp} to each equation and uses the first two estimates of \eqref{eq:Phi-1012:less2:error} for the source term. By adding a sufficiently large amount of the estimate for $\Phiminus{0}$ to the estimate of $\Phiminus{1}$ and then taking $\hat{R}_0$ sufficiently large, one can absorb the error terms arising from the source terms and obtain the estimate \eqref{eq:rpminus:01:less2}. Turn to $p=2$. We apply the estimate \eqref{eq:rp:p=2} of Proposition \ref{prop:wave:rp} to both of the equations \eqref{eq:Phi-10} and \eqref{eq:Phi-11} and fix $\veps$ small enough by requiring the terms with $\veps$ coefficient on the RHS of \eqref{eq:Phi-10:p2:error} and
\eqref{eq:Phi-11:p2:erroresti} to be absorbed. In the next step, one adds an $A_0$ multiple of the estimate \eqref{eq:rp:p=2} for $\varphi=\Phiminus{0}$ to the estimate \eqref{eq:rp:p=2} for $\varphi=\Phiminus{1}$ to obtain a new estimate and requires $A_0\gg \veps^{-1}$ such that the terms with $\veps^{-1}$ coefficient on the RHS of \eqref{eq:Phi-11:p2:erroresti} is absorbed. In the end, one chooses $\hat{R}_0$ sufficiently large such that for any $R_0\geq \hat{R}_0$, the remaining terms on the RHS of \eqref{eq:Phi-10:p2:error} and
\eqref{eq:Phi-11:p2:erroresti} are absorbed by the LHS of this new estimate, which then finishes the proof of \eqref{eq:rpminus:01:p=2}.

Consider next the full system \eqref{eq:Phi-1012}. For $p\in (0,2)$, the estimate \eqref{eq:rp:less2} can be applied to $\varphi=\Phiminus{2}$, and from the estimate \eqref{eq:Phi-12:less2:error} for the source term and the already proven estimate \eqref{eq:rpminus:01:less2} for $\Phiminus{0}$ and $\Phiminus{1}$, the estimate \eqref{eq:rpminus:01:less2} for $j=2$ manifestly holds. Similarly, for $p=2$, one applies the estimate \eqref{eq:rp:p=2:2} to $\varphi=\Phiminus{2}$. Given that the estimate \eqref{eq:rpminus:01:p=2} has been proven and from the estimates \eqref{eq:Phi-12:less2:error}, \eqref{eq:Phi-12:p2:error} and \eqref{eq:Phi-12:p2:erroresti1}, the estimate \eqref{eq:rpminus:012:p=2} follows easily.

After adding the BEAM estimates to \eqref{eq:rpminus:01:less2}, \eqref{eq:rpminus:01:p=2} and \eqref{eq:rpminus:012:p=2}, respectively, and in view of the fact that for any $\tb$,
\begin{align}
\hspace{4ex}&\hspace{-4ex}
\sum_{i=0,1,2}\norm{rV\PsiminusHigh{i}}^2_{W_{0}^{\reg-l(i)}(\Sigmatb)}
+\sum_{i=0,1}\norm{\PsiminusHigh{i}}^2_{W_{-2}^{\reg+1}(\Sigmatb)}
+\norm{\mathbb{D}_2\PsiminusHigh{2}}^2_{W_{-2}^{\reg-1}(\Sigmatb)}\notag\\
\sim{}&\sum_{i=0,1,2}\Big(\norm{rV\PsiminusHigh{i}}^2_{W_{0}^{\reg-l(i)}(\Sigmatb)}
+\norm{\PsiminusHigh{i}}^2_{W_{-2}^{\reg+1-l(i)}(\Sigmatb)}\Big),
\end{align}
the estimates \eqref{eq:rpminusglobal:01:less2},  \eqref{eq:rpminusglobal:01:p=2} and \eqref{eq:rpminusglobal:012:p=2} then follow.
\end{proof}

\subsubsection{Basic energy $2$-decay condition for spin $-1$ component}

\begin{definition}
\label{def:Fterm:Psiminus:1and2level}
Define
\begin{subequations}
\begin{align}
F^{(2)}(\reg,p,\tb,\Psiminus)={}&0, \quad \text{for } p\in [-1,0),\\
F^{(2)}(\reg,p,\tb,\Psiminus)={}&\sum_{i=0,1,2}
\Big(\norm{rV\PsiminusHigh{i}}^2_{W_{p-2}^{\reg-2-l(i)}(\Sigmatb)}
+\norm{\PsiminusHigh{i}}^2_{W_{-2}^{\reg-1-l(i)}(\Sigmatb)}\Big), \quad \text{for } p\in [0,2],
\end{align}
\end{subequations}
and define
\begin{subequations}
\begin{align}
F^{(1)}(\reg,p,\tb,\Psiminus)={}&0, \quad \text{for } p\in [-1,0),\\
F^{(1)}(\reg,p,\tb,\Psiminus)={}&\sum_{i=0,1}
\Big(\norm{rV\PsiminusHigh{i}}^2_{W_{p-2}^{\reg-2}(\Sigmatb)}
+\norm{\PsiminusHigh{i}}^2_{W_{-2}^{\reg-1}(\Sigmatb)}\Big), \quad \text{for } p\in [0,2].
\end{align}
\end{subequations}
\end{definition}

\begin{prop}
\label{prop:BEDC:Psiminus}
Let $j\in \mathbb{N}$ and let $\reg\in \mathbb{N}$. Define $l(i)=\max(0,i-1)$.
Assume the BEAM condition to order $\reg+2$ for spin $-1$ component is satisfied for spin $-1$ component, then there exists a constant $\regl(j)$ such that for any $p\in [0,2]$ and any $\tb\geq\tb_0$,
\begin{align}
\label{eq:BEDC:Psiminus}
\hspace{4ex}&\hspace{-4ex}
\sum_{i=0,1}
\Big(\norm{rV\Lxi^j\PsiminusHigh{i}}^2_{W_{p-2}^{\reg-\regl(j)+1}(\Sigmatb)}
+\norm{\Lxi^j\PsiminusHigh{i}}^2_{W_{-2}^{\reg-\regl(j)+2}(\Sigmatb)}\Big)
\notag\\
\hspace{4ex}&\hspace{-4ex}
+\sum_{i=0,1}
\Big(\norm{rV\Lxi^j\PsiminusHigh{i}}^2_{W_{p-3}^{\reg-\regl(j)}(\Dinfty)}
+\norm{\Lxi^j\PsiminusHigh{i}}^2_{W_{-2}^{\reg-\regl(j)+1}(\Dinfty)}\Big)
\notag\\
\lesssim {}&\tb^{-4-2j+p}
\sum_{i=0,1,2}
\Big(\norm{rV\PsiminusHigh{i}}^2_{W_{0}^{\reg-l(i)}(\Sigmazero)}
+\norm{\PsiminusHigh{i}}^2_{W_{-2}^{\reg+1-l(i)}(\Sigmazero)}\Big).
\end{align}
\end{prop}

\begin{proof}
The $r^p$ estimates \eqref{eq:rpminusglobal:01:less2} for the set $\mathbf{Q}_1$ then implies for any $p\in [0,2)$,
\begin{align}
\label{eq:BEDC:Phiminus:1:v1}
F^{(1)}(\reg,p,\tb_2,\Psiminus)
+\int_{\tb_1}^{\tb_2}F^{(1)}(\reg-1,p-1,\tb,\Psiminus)\di \tb
\lesssim F^{(1)}(\reg,p,\tb_1,\Psiminus).
\end{align}
Note that $p=0$ case follows from the BEAM estimates.
An application of \cite[Lemma 5.2]{andersson2019stability} then gives for any $p\in [0,5/3]$,
\begin{align}
F^{(1)}(\reg-2,p,\tb,\Psiminus)\lesssim {}&\tb^{-5/3+p}F^{(1)}(\reg,5/3,\tb/2,\Psiminus).
\end{align}
Similarly to the spin $+1$ component, we can obtain better energy decay for $\Lxi$ derivative. One can utilize equation \eqref{eq:Phi-10} and the definition of the spin-weighted wave operator $\squareShat_{s}$ in \eqref{eq:squareShat} and use the replacement $Y=\frac{\R}{\Delta}\big(2\Lxi+\frac{2a}{\R}\Leta-V\big)$ away from horizon to rewrite $r^2V\Lxi \PsiminusHigh{0}$ as a weighted sum  with $O(1)$ coefficients of $r^{-1}\Leta \PsiminusHigh{0}$ and terms of the form $X_2X_1\PsiminusHigh{0}$ with $X_1,X_2\in \CDeri\cup\{1\}$. There is a similar expansion for each of $r^2 V\PsiminusHigh{i}$, but involving also the $\PsiminusHigh{i'}$ with $i'<i$. These together give
\begin{align}
\label{eq:F1LxiPsiminus}
F^{(1)}(\reg,5/3,\tb,\Lxi\Psiminus)
\lesssim{}F^{(1)}(\reg,2,\tb,\Lxi\Psiminus)
\lesssim F^{(1)}(\reg+1,0,\tb,\Psiminus).
\end{align}
Therefore, for all $p\in [0,5/3]$,
\begin{align}
\hspace{4ex}&\hspace{-4ex}
F^{(1)}(\reg-2-5,p,\tb,\Lxi\Psiminus)\notag\\
\lesssim {}&
\tb^{-5/3+p}F^{(1)}(\reg-5,5/3,\tb/2,\Lxi\Psiminus)\notag\\
\lesssim{}&\tb^{-5/3+p}F^{(1)}(\reg-4,0,\tb/2,\Psiminus)\notag\\
\lesssim{}&\tb^{-10/3+p}F^{(1)}(\reg-2,5/3,\tb/4,\Psiminus).
\end{align}
We can then use this to estimate the last term in \eqref{eq:rpminusglobal:01:p=2} for $p=2$ by $\tb_1^{-4/3+\delta}F^{(1)}(\reg+\reg',5/3,\tb_1/4,\Psiminus)$ for some $\reg'>0$.
One thus concludes that for any $p\in [0,2]$,
\begin{align}
\label{eq:BEDC:Phiminus:1}
F^{(1)}(\reg,p,\tb_2,\Psiminus)
+\int_{\tb_1}^{\tb_2}F^{(1)}(\reg-1,p-1,\tb,\Psiminus)\di \tb
\lesssim F^{(1)}(\reg+\reg',p,\tb_1,\Psiminus),
\end{align}
and
\begin{align}
\label{eq:edcay:Phiminus:1:step1}
F^{(1)}(\reg,p,\tb,\Lxi^j\Psiminus)
\lesssim{}&\tb^{-2+p-2j}F^{(1)}(\reg+\reg'(j),2,\tb/2,\Psiminus),
\end{align}
where the proof for general $j$ cases is the same as the one of spin $+1$ component.

The $r^p$ estimate \eqref{eq:rpminusglobal:01:less2} for the set $\mathbf{Q}_2$ implies for $p\in (0,2)$,
\begin{align}
\label{eq:BEDC:Phiminus:2}
F^{(2)}(\reg,p,\tb_2,\Psiminus)
+\int_{\tb_1}^{\tb_2}F^{(2)}(\reg-1,p-1,\tb,\Psiminus)\di \tb
\lesssim F^{(2)}(\reg,p,\tb_1,\Psiminus).
\end{align}
An application of \cite[Lemma 5.2]{andersson2019stability} then gives for any $p\in [1,5/3]$,
\begin{align}
F^{(2)}(\reg-1,p,\tb,\Psiminus)\lesssim {}&\tb^{-5/3+p}F^{(2)}(\reg,5/3,\tb_0,\Psiminus).
\end{align}
We go back to the estimate \eqref{eq:rpminusglobal:01:less2} with $p=1$, $j=2$, $\tb_1=\tb$ and $\tb_2=2\tb$, then an application of the mean-value principle suggests that there exists a $\tb'\in [\tb,2\tb]$ such that
\begin{align}
\sum_{i=0,1,2}
\norm{\PsiminusHigh{i}}^2_{W_{-2}^{\reg-2-l(i)}(\Hyper_{\tb'})}
\lesssim{}& \tau^{-1}F^{(2)}(\reg,1,\tb,\Psiminus)\lesssim \tb^{-5/3}F^{(2)}(\reg+1,5/3,\tb_0,\Psiminus).
\end{align}
From the definitions of $\PsiminusHigh{i}$ in \eqref{def:Psiminusi}, one has for any $\tb\geq \tb_0$,
\begin{align}
\label{eq:phiminus:equivalenceenergy:01:012}
\sum_{i=0,1}
\Big(\norm{rV\PsiminusHigh{i}}^2_{W_{0}^{\reg-3}(\Hyper_{\tb})}
+\norm{\PsiminusHigh{i}}^2_{W_{-2}^{\reg-2}(\Hyper_{\tb})}\Big)
\sim{}&
\sum_{i=0,1,2}
\norm{\PsiminusHigh{i}}^2_{W_{-2}^{\reg-2-l(i)}(\Hyper_{\tb})}.
\end{align}
The estimate \eqref{eq:edcay:Phiminus:1:step1} hence implies that for any $p\in [0,2]$,
\begin{align}
F^{(1)}(\reg,p,\tb,\Lxi^j\Psiminus)\lesssim{}&
\tb^{-11/3+p-2j}F^{(2)}(\reg+\regl(j),5/3,\tb/2,\Psiminus)\notag\\
\lesssim{}&\tb^{-11/3+p-2j}F^{(2)}(\reg+\regl(j),5/3,\tb_0,\Psiminus).
\end{align}

Similarly, we can obtain better energy decay for $\Lxi$ derivative. One achieves
\begin{align}
\label{eq:F2LxiPsiminus}
F^{(2)}(\reg-3j,5/3,\tb,\Lxi\Psiminus)
\lesssim{}&F^{(2)}(\reg+1-3j,0,\tb,\Psiminus).
\end{align}
Therefore, for all $p\in [1,5/3]$,
\begin{align}
\hspace{4ex}&\hspace{-4ex}
F^{(2)}(\reg-2-\omega(p)-3,p,\tb,\Lxi\Psiminus)\notag\\
\lesssim {}&
\tb^{-5/3+p}F^{(2)}(\reg-5,5/3,\tb/2,\Lxi\Psiminus)\notag\\
\lesssim{}&\tb^{-5/3+p}F^{(2)}(\reg-4,0,\tb/2,\Psiminus)\notag\\
\lesssim{}&\tb^{-10/3+p}F^{(2)}(\reg-2,5/3,\tb/4,\Psiminus)
\lesssim\tb^{-10/3+p}F^{(2)}(\reg-2,5/3,\tb_0,\Psiminus).
\end{align}
This implies there exists a constant $\regl>0$ such that
\begin{align}
\hspace{4ex}&\hspace{-4ex}
\int_{\tb_1}^{\tb_2}
\tb^{1+\delta}
\Big(\sum_{i=0,1,2}\norm{\Lxi\PsiminusHigh{i}}^2_{W_{-2}^{\reg}(\Sigmatb)}
+
\norm{\PsiminusHigh{1}}^2_{W_{-2}^{\reg}(\Sigmatb)}
+
\norm{\PsiminusHigh{0}}^2_{W_{-2}^{\reg}(\Sigmatb)}
\Big)
\di \tb\notag\\
 \lesssim{}&\tb_1^{-4/3+\delta}F^{(2)}(\reg+\regl,5/3,\tb_1/2,\Psiminus).
\end{align}
Thus, together with the estimate \eqref{eq:BEDC:Phiminus:2} for $p\in (0,2)$, we have for $p\in (0,2]$ that
\begin{align}
\label{eq:BEDC:Phiminus:2:p=2}
F^{(2)}(\reg,p,\tb_2,\Psiminus)
+\int_{\tb_1}^{\tb_2}F^{(2)}(\reg-1,p-1,\tb,\Psiminus)\di \tb
\lesssim F^{(2)}(\reg+\regl,p,\tb_1,\Psiminus).
\end{align}
This estimate also holds for $p=0$ from the estimate \eqref{eq:BEDC:Phiminus:1} with $p=2$ and the relation that
\begin{align}
F^{(1)}(\reg,2,\tb,\Psiminus) \sim F^{(2)}(\reg,0,\tb,\Psiminus).
\end{align}
Following the same argument above, one has
\begin{align}
\label{eq:BEDC:Phiminus:2:v1}
F^{(1)}(\reg,2,\tb,\Psiminus)\lesssim{}
F^{(2)}(\reg,0,\tb,\Psiminus)
\lesssim{}&\tb^{-2}F^{(2)}(\reg+\regl,2,\tb_2,\Psiminus).
\end{align}
The $r^p$ estimates in Proposition \ref{prop:rpglobal:Phiminus} for the set $\mathbf{Q}_1$ then implies
\begin{align}
\label{eq:BEDC:Phiminus:1}
F^{(1)}(\reg,p,\tb_2,\Psiminus)
+\int_{\tb_1}^{\tb_2}F^{(1)}(\reg-1,p-1,\tb,\Psiminus)\di \tb
\lesssim F^{(1)}(\reg,p,\tb_1,\Psiminus).
\end{align}
Combining this estimate with \eqref{eq:edcay:Phiminus:1:step1} then proves the desired estimate \eqref{eq:BEDC:Psiminus}.
\end{proof}

\section{Basic energy $\gamma$-decay condition implies pointwise decay}


In this section, we prove Theorem \ref{thm:2} which is to obtain pointwise asymptotics for spin $\pm 1$ components and the middle component in a subextremal Kerr spacetime $(\mathcal{M},g_{M,a})$ under the assumption that the basic energy $\gamma$-decay condition with $\gamma\geq 1$, a suitably large $\reg$ and $D_{\pm 1}=D_{\pm 1}(M,a,\reg,j)$ holds for spin $\pm 1$ components. This basic energy $\gamma$-decay condition is  assumed throughout this section.


\subsection{Decay for spin $\pm 1$  components in a subextremal Kerr spacetime}
\label{sect:vrdecay:psiplus}

We are ready to prove weak pointwise decay for spin $\pm 1$ components.

\begin{prop}
\label{prop:weakdecay:spin+1-1:v3}
There exists a constant $\reg'(j)$ such that the following estimates hold for spin $+1$ and $-1$ components assuming the basic energy $\gamma$-decay condition for these two components respectively:
\begin{subequations}
\label{eq:weakdecay:spin+1-1:v3}
\begin{align}
\label{eq:weakdecay:spin+1:v3}
\absCDeri{\Lxi^j\psiplus}{\reg-\reg'}
\lesssim{}&(D_{+1})^{\half}v^{-1}\tb^{-(\gamma-1)/2-j},\\
\label{eq:weakdecay:spin-1:v3}
\absCDeri{\Lxi^j\psiminus}{\reg-\reg'}
\lesssim{}&(D_{-1})^{\half}v^{-1}\tb^{-(\gamma+1)/2-j}\max\{r^{-1},\tb^{-1}\}.
\end{align}
\end{subequations}
In particular, the peeling properties are shown, and we obtain $\tb^{-(\gamma-1)/2}$ and $\tb^{-(\gamma+3)/2}$ pointwise decay for the radiation fields $r \psiplus$ and $r\psiminus$, respectively.
\end{prop}

\begin{proof}

By the basic energy $\gamma$-decay condition,  the estimate \eqref{eq:Sobolev:1} gives
\begin{subequations}
\begin{align}
\absCDeri{\Lxi^j\Psiplus}{\reg-\regl}
\lesssim{}(D_{+1})^{\half}\tb^{-(\gamma-1)/2-j},
\end{align}
the estimate \eqref{eq:Sobolev:2} with $\alpha=1/2$ gives
\begin{align}
\absCDeri{\Lxi^j(r^{-1/2}\Psiplus)}{\reg-\reg'}
\lesssim{}(D_{+1})^{\half}\tb^{-\gamma/2-j},
\end{align}
and the estimate \eqref{eq:Sobolev:3}  implies
\begin{align}
\absCDeri{\Lxi^j\psiplus}{\reg-\reg'}
\lesssim{}(D_{+1})^{\half}\tb^{-(\gamma+1)/2-j}.
\end{align}
\end{subequations}
Combining the above three estimates together proves the estimate \eqref{eq:weakdecay:spin+1:v3}.

The same argument as treating spin $+1$ component applies to spin $-1$ component and gives  and any $i\in \{0,1\}$,
\begin{subequations}
\begin{align}
\absCDeri{\Lxi^j\PsiminusHigh{i}}{\reg-\reg'}
\lesssim{}&(D_{-1})^{\half}\tb^{-(\gamma+1)/2-j},\\
\absCDeri{\Lxi^j(r^{-1/2}\PsiminusHigh{i})}{\reg-\reg'}
\lesssim{}&(D_{-1})^{\half}\tb^{-\gamma/2-1-j},\\
\absCDeri{\Lxi^j(r^{-1}\PsiminusHigh{i})}
{\reg-\reg'}
\lesssim{}&(D_{-1})^{\half}\tb^{-(\gamma+3)/2-j}.
\end{align}
\end{subequations}
One obtains from the above three estimates together that
\begin{align}
\label{eq:psiminusi:weakweak}
\sum_{i=0,1}\absCDeri{\Lxi^j\psiminusHigh{i}}{\reg-\reg'}
\lesssim{}(D_{-1})^{\half}v^{-1}\tb^{-(\gamma+1)/2-j}.
\end{align}
Furthermore, the wave equation \eqref{eq:Phi-10} of $\Phiminus{0}$ can be rewritten as
\begin{align}
(2\edthR'\edthR+\tfrac{2ar}{\R}\Leta)\Psiminus
={}&Y\Phiminus{1}
 -(
 2a\Leta
+a^2\sin^2\theta\Lxi
+2ia\cos\theta) \Lxi\Psiminus.
\end{align}
In the region close to horizon, the estimate \eqref{eq:weakdecay:spin-1:v3} follows from \eqref{eq:psiminusi:weakweak}.
Consider the region away from horizon. On the RHS, one can expand $Y$ in terms of $r^{-1} rV$, $\Lxi$ and $r^{-2}\Leta$ with $O(1)$ coefficients, and the remaining terms all have a $\Lxi$ derivative, yielding the RHS decays like $v^{-1}\tb^{-(\gamma+1)}\max\{\tb^{-1},r^{-1}\}$. The LHS is an elliptic operator on $\mathbb{S}^2$, hence these together prove \eqref{eq:weakdecay:spin-1:v3}.
\end{proof}

\subsection{Improved decay for spin $\pm 1$ components in slowly rotating Kerr spacetimes}


The pointwise behaviours in Proposition \ref{prop:weakdecay:spin+1-1:v3}, however, do not  enjoy fast time decay in the interior region $\{\rb\leq \tb\}$, in particular not in finite radius region. We can improve these estimates in the interior region and obtain better decay results for these components in slowly rotating Kerr spacetimes.

\subsubsection{Improved decay of spin $-1$ component}


We start with proving an elliptic-type estimate.
\begin{prop}
There exists a constant $\veps_0>0$ such that for all $|a|/M\leq \veps_0$,
\begin{align}\label{eq:degellip:phiminus:lowerreg}
\hspace{4ex}&\hspace{-4ex}\int_{\Sigmatb}\Big(
\mu^2 r^2\abs{\partial_{\rb}^2 \phiminus{1}}^2
+r^{-2}\abs{\edthR'\edthR \phiminus{1}}^2
+\mu \abs{\partial_{\rb}\edthR\phiminus{1}}^2
+\abs{\partial_{\rb}\phiminus{1}}^2\notag\\
\hspace{4ex}&\hspace{-4ex}\qquad
+\mu \big(r^4\abs{\partial_{\rb}^2 \phiminus{0}}^2
+\abs{\edthR'\edthR \psiminusHigh{0}}^2\big)
+\mu^2r^2\abs{\partial_{\rb}\edthR\phiminus{0}}^2
\Big)\di^3\mu\notag\\
\lesssim{}&\int_{\Sigmatb}
\bigg(\sum_{i=0,1}\Big(
\abs{\Lxi\prb\psiminusHigh{i}}^2
+\abs{\Lxi\psiminusHigh{i}}^2
+\abs{\Lxi^2\psiminusHigh{i}}^2
+\abs{\Lxi\Leta\psiminusHigh{i}}^2
\Big)\notag\\
&\quad \quad
+r^2\abs{\Lxi\prb\phiminus{1}}^2
+
\sum_{i=0,1}a^2 r^{-2}\abs{\Leta\prb\psiminusHigh{i}}^2
\bigg)\di^3\mu\notag\\
\lesssim{}& \sum\limits_{i=0,1}
\Big(\norm{\Lxi\psiminusHigh{i}}^2_{W_{0}^{1}(\Sigmatb)}
+a^2\norm{\Leta\prb\psiminusHigh{i}}^2_{W_{-2}^{0}(\Sigmatb)}\Big),
\end{align}
and for any $\reg\geq 1$, any $r'>r_+$ and any $\abs{a}/M\leq \veps_0$,
\begin{align}\label{eq:degellip:phiminus}
\hspace{4ex}&\hspace{-4ex}\int_{\Sigmatb^{r'}}\Big(
\mu^2 r^2\absCDeri{\partial_{\rb}^2 \phiminus{1}}{\reg-1}^2
+r^{-2}\absCDeri{\edthR'\edthR \phiminus{1}}{\reg-1}^2
+\mu \absCDeri{\partial_{\rb}\edthR\phiminus{1}}{\reg-1}^2
+\absCDeri{\partial_{\rb}\phiminus{1}}{\reg-1}^2\notag\\
\hspace{4ex}&\hspace{-4ex}\qquad
+\mu \big(r^4\absCDeri{\partial_{\rb}^2 \phiminus{0}}{\reg-1}^2
+\absCDeri{\edthR'\edthR \psiminusHigh{0}}{\reg-1}^2\big)
+\mu^2r^2\absCDeri{\partial_{\rb}\edthR\phiminus{0}}{\reg-1}^2
\Big)\di^3\mu\notag\\
\hspace{2ex}&\hspace{-2ex}\lesssim_{\reg,r'}{} \sum\limits_{i=0,1}
\Big(\norm{\Lxi\psiminusHigh{i}}^2_{W_{0}^{\reg}(\Sigmatb)}
+a^2\norm{\Leta\prb\psiminusHigh{i}}^2_{W_{-2}^{\reg-1}(\Sigmatb)}\Big).
\end{align}
\end{prop}

\begin{proof}
Let
\begin{subequations}
\label{def:vectorsKandVR}
\begin{align}
K={}\Lxi +a(\R)^{-1}\Leta.
\end{align}
Let $H=2\mu^{-1}+\partial_r h(r)$ with $h(r)$ being the function introduced in Section \ref{sect:foliation}, then
 one can express  $Y$ and $\VR$ as
\begin{align}
\label{def:vectorVRintermsofprb}
Y={}-\prb+(2\mu^{-1}-H)\Lxi,\quad
\VR={}\partial_{\rb}+H K
+ a(2\mu^{-1}-H)(\R)^{-1}\Leta.
\end{align}
\end{subequations}
By the choice of the hyperboloidal coordinates, there exist positive constants $c_0$ and $c_1$ such that
\begin{align}
\label{eq:propertyofHfunction}
\lim_{r\to \infty}r^2 H=c_0, \quad \text{and} \quad \abs{ H-2\mu^{-1}-c_1}\lesssim \mu  \quad
\text{as } r\to r_+.
\end{align}
Defining further $\tilde{H}=2\mu^{-1}-H=-\partial_rh(r)$, one can expand $-(\R)YV\varphi$ as
\begin{align}
\hspace{4ex}&\hspace{-4ex}-(\R)YV\varphi\notag\\
={}&-(\R)(-\prb +\tilde{H}\Lxi)\Big(\mu \Big(\prb +HK+\tilde{H}\frac{a}{\R}\Leta\Big)\varphi\Big)\notag\\
={}&(\R)\prb(\mu\prb\varphi)
-\mu (\R)\tilde{H}\Lxi\prb\varphi
+(\R)\prb(\mu H)\Lxi\varphi\notag\\
&
+(\R)\prb\Big(\frac{a\mu H}{\R}\Big)\Leta\varphi
+\mu H(\R)\Lxi\prb\varphi
+a\mu H\Leta\prb\varphi\notag\\
&-\mu (\R)H\tilde{H}\Lxi K\varphi
+a\mu\tilde{H}\Leta\prb\varphi
+(\R)\prb\Big(\frac{a\mu\tilde{H}}{\R}\Big)\Leta\varphi\notag\\
&-a\mu \tilde{H}^2\Lxi\Leta\varphi\notag\\
={}&(\R)\prb(\mu\prb\varphi)-\mu (\R)\tilde{H}\Lxi\prb\varphi+(\R)\prb(\mu H)\Lxi\varphi\notag\\
&
+\mu H(\R)\Lxi\prb\varphi-\mu (\R)H\tilde{H}\Lxi K\varphi
-a\mu \tilde{H}^2\Lxi\Leta\varphi\notag\\
&+(\R)\prb(2a(\R)^{-1})\Leta\varphi
+2a\Leta\prb\varphi\notag\\
={}&(\R)\prb(\mu\prb\varphi)
+2(\R)(\mu H-1)\Lxi\prb\varphi
+(\R)H(\mu H-2)\Lxi^2\varphi\notag\\
&+2a\mu^{-1}(\mu H-2)\Lxi\Leta\varphi
++(\R)\prb(\mu H)\Lxi\varphi\notag\\
&-\frac{4ar}{\R}\Leta\varphi
+2a\Leta\prb\varphi.
\end{align}
The field equations \eqref{eq:Phi-10} and \eqref{eq:Phi-11} for $\Phiminus{i}$ $(i=0,1)$ become
\begin{subequations}
\label{eq:Phiminus01:hyperboloidal}
\begin{align}
\label{eq:Phiminus01:hyperboloidal:0}
\hspace{4ex}&\hspace{-4ex}
2\edthR'\edthR\Phiminus{0}
+(\R)\partial_{\rb}(\mu\partial_{\rb})\Phiminus{0}
+\tfrac{a^2\Delta}{(\R)^2}\Phiminus{0}\notag\\
\hspace{4ex}&\hspace{-4ex}
-\tfrac{4ar}{\R}\Leta\Phiminus{0}
+\tfrac{2(\PR)}{(\R)^2}\Phiminus{1}
\notag\\
={}&H^{(0)}(\Phiminus{0}),\\
\label{eq:Phiminus01:hyperboloidal:1}
\hspace{4ex}&\hspace{-4ex}
2\edthR'\edthR\Phiminus{1}
+(\R)\partial_{\rb}(\mu\partial_{\rb})\Phiminus{1}
+\tfrac{a^2\Delta}{(\R)^2}\Phiminus{1}\notag\\
\hspace{4ex}&\hspace{-4ex}
-\tfrac{2a^2(\PR)}{(\R)^2}\Phiminus{0}
+\tfrac{2a(r^2-a^2)}{\R}\Leta\Phiminus{0}\notag\\
={}&H^{(1)}(\Phiminus{1}),
\end{align}
\end{subequations}
where for $i\in \{0,1\}$,
\begin{align}
\label{def:HiPhiminus}
H^{(i)}(\Phiminus{i})={}&
-2a(1+\mu^{-1}(\mu H-2))\Lxi\Leta\Phiminus{i}
-2(\R)(\mu H-1)\Lxi\prb\Phiminus{i}
\notag\\
&
-((\R)H(\mu H-2)+a^2\sin^2\theta )\Lxi^2\Phiminus{i}
\notag\\
&
-(2ia\cos\theta+(\R)\partial_r(\mu H))\Lxi \Phiminus{i}\notag\\
&-2a(\R)^{\half}\Leta((\R)^{-\half}\prb\Phiminus{i}).
\end{align}
Equation \eqref{eq:Phiminus01:hyperboloidal:0} can further be written as
\begin{align}
\label{eq:Phiminus01:hyperboloidal:0:rescale}
2\edthR'\edthR\Phiminus{0}
+\mu\partial_{\rb}((\R)\partial_{\rb})\Phiminus{0}
+\tfrac{a^2\mu}{\R}\Phiminus{0}
={}&\tilde{H}^{(0)}(\Phiminus{0}),
\end{align}
where
\begin{align}
\label{def:HitildePhiminus}
\tilde{H}^{(0)}(\Phiminus{0})={}&
-2a(1+\mu^{-1}(\mu H-2))\Lxi\Leta\Phiminus{0}
-2(\R)(\mu H-1)\Lxi\prb\Phiminus{0}
\notag\\
&
-((\R)H(\mu H-2)+a^2\sin^2\theta )\Lxi^2\Phiminus{0}
\notag\\
&
-(2ia\cos\theta
+\mu\prb((\R)H))\Lxi \Phiminus{0}\notag\\
&-2a\mu(\R)^{\half}\Leta\prb\big((\R)^{-\half}
\mu^{-1}\Phiminus{0}\big).
\end{align}
Multiplying  \eqref{eq:Phiminus01:hyperboloidal:0:rescale} and \eqref{eq:Phiminus01:hyperboloidal:1} by $-\mu^{-1}(\R)^{-1}\overline{\Phiminus{0}}$ and $-(\R)^{-2}\overline{\Phiminus{1}}$ respectively, taking the real part, integrating over $\Sigmatb$ and using integration by parts, we arrive at
\begin{subequations}
\begin{align}
\label{eq:degellip:Phiminus:0:1}
\hspace{4ex}&\hspace{-4ex}\int_{\Sigmatb}\Big(
2\mu^{-1}(\R)^{-1} \abs{\edthR\Phiminus{0}}^2
-\partial_{\rb}\big( \partial_{\rb}\Phiminus{0}\overline{\Phiminus{0}}\big)\Big.\notag\\
\hspace{4ex}&\hspace{-4ex}\Big.
\qquad
-\tfrac{2r}{\R}
\Re\big(\partial_{\rb}\Phiminus{0}\overline{\Phiminus{0}}\big)
+\abs{\partial_{\rb}\Phiminus{0}}^2
-\tfrac{a^2}{(\R)^2}\abs{\Phiminus{0}}^2
\Big)
\di^3\mu\notag\\
={}& \int_{\Sigmatb} -\mu^{-1}(\R)^{-1}\Re\big(\overline{\Phiminus{0}}
\tilde{H}^{(0)}(\Phiminus{0})\big)\di^3\mu,\\
\label{eq:degellip:Phiminus:0:2}
\hspace{4ex}&\hspace{-4ex}\int_{\Sigmatb}\Big(2(\R)^{-2} \abs{\edthR\Phiminus{1}}^2
-\partial_{\rb}\big(\mu (\R)^{-1}\partial_{\rb}\Phiminus{1}\overline{\Phiminus{1}}\big)\Big.\notag\\
\hspace{4ex}&\hspace{-4ex}\Big.
\qquad
-\tfrac{2r\mu}{(\R)^2}
\Re\big(\partial_{\rb}\Phiminus{1}\overline{\Phiminus{1}}\big)
+\mu (\R)^{-1}\abs{\partial_{\rb}\Phiminus{1}}^2
-\tfrac{a^2\Delta}{(\R)^4}\abs{\Phiminus{1}}^2\Big.\notag\\
\hspace{4ex}&\hspace{-4ex}\Big. \qquad
+\tfrac{2a^2(\PR)}{(\R)^4}
\Re\big(\overline{\Phiminus{1}}\Phiminus{0}\big)
-\tfrac{2a(r^2-a^2)}{(\R)^3}\overline{\Phiminus{1}}\Leta\Phiminus{0}
\Big)\di^3\mu\notag\\
={}& \int_{\Sigmatb} -(\R)^{-2}\Re\big(\overline{\Phiminus{1}}
H^{(1)}(\Phiminus{1})\big)\di^3\mu.
\end{align}
\end{subequations}
The integral of the total $\prb$-derivative terms vanish from Proposition \ref{prop:weakdecay:spin+1-1:v3}, and from Proposition \ref{prop:basicesti}, the LHS of \eqref{eq:degellip:Phiminus:0:1} is thus larger than or equal to
\begin{align}
\hspace{4ex}&\hspace{-4ex}\frac{1}{4}\int_{\Sigmatb}\Big(
\mu^{-1}(\R)^{-1} \abs{\edthR\Phiminus{0}}^2
+\abs{\partial_{\rb}\Phiminus{0}}^2
\Big)
\di^3\mu,
\end{align}
which then implies
\begin{align}
\label{eq:degellip:Phiminus:0:8}
\hspace{4ex}&\hspace{-4ex}\int_{\Sigmatb}\Big(
2\mu^{-1}(\R)^{-1} \abs{\edthR\Phiminus{0}}^2
+2\abs{\partial_{\rb}\Phiminus{0}}^2
\Big)
\di^3\mu\notag\\
\leq{}& \int_{\Sigmatb} -8\mu^{-1}(\R)^{-1}\Re\big(\overline{\Phiminus{0}}
\tilde{H}^{(0)}(\Phiminus{0})\big)\di^3\mu.
\end{align}
By adding  the estimate \eqref{eq:degellip:Phiminus:0:8} to \eqref{eq:degellip:Phiminus:0:2} and applying Cauchy-Schwarz to the third line of \eqref{eq:degellip:Phiminus:0:2},
 we arrive at
\begin{align}
\label{eq:degellip:Phiminus:0:9}
\hspace{4ex}&\hspace{-4ex}\int_{\Sigmatb}\Big(
\mu^{-1}(\R)^{-1} \abs{\edthR\Phiminus{0}}^2
+\abs{\partial_{\rb}\Phiminus{0}}^2
+\tfrac{1}{8}(\R)^{-2} \abs{\edthR\Phiminus{1}}^2
+\tfrac{1}{8}\mu(\R)^{-1}\abs{\partial_{\rb}\Phiminus{1}}^2
\Big)
\di^3\mu\notag\\
\leq{}& \int_{\Sigmatb} \Re\Big(-8\mu^{-1}(\R)^{-1}\overline{\Phiminus{0}}
\tilde{H}^{(0)}(\Phiminus{0})
-(\R)^{-2}\overline{\Phiminus{1}}
{H}^{(1)}(\Phiminus{1})\Big)\di^3\mu\notag\\
\triangleq {}&\mathbb{G}_1.
\end{align}

Moving the second line of \eqref{eq:Phiminus01:hyperboloidal:1} to the RHS, taking a square of each side of equations \eqref{eq:Phiminus01:hyperboloidal:0:rescale} and \eqref{eq:Phiminus01:hyperboloidal:1}, multiplying by $\mu^{-1}(\R)^{-1}$ and $(\R)^{-2}$ respectively,
taking the real part, integrating over $\Sigmatb$ and applying Cauchy-Schwarz inequality to the integral terms of $\tilde{H}^{(0)}(\Phiminus{0})$ and $H^{(1)}(\Phiminus{1})$, one arrives at
\begin{subequations}\label{eq:degellip:Phiminus:1}
\begin{align}
\label{eq:degellip:Phiminus:1:0}
\hspace{4ex}&\hspace{-4ex}
\int_{\Sigmatb}\Big(4 \mu^{-1}(\R)^{-1}\abs{\edthR'\edthR \Phiminus{0}}^2
+\mu(\R)^{-1}\abs{\partial_{\rb}((\R)\partial_{\rb})\Phiminus{0}}^2\Big.\notag\\
\hspace{4ex}&\hspace{-4ex}\Big. \qquad
+4(\R)^{-1}\Re\big(\partial_{\rb}((\R)\partial_{\rb}
\overline{\Phiminus{0}})
\edthR'\edthR \Phiminus{0}\big)\Big)\di^3\mu\notag\\
\lesssim{}&\int_{\Sigmatb} \mu^{-1}r^{-2}\abs{\tilde H^{(0)}(\Phiminus{0})}^2
\di^3\mu\notag\\
\triangleq {}&\mathbb{G}_2,\\
\label{eq:degellip:Phiminus:1:1}
\hspace{4ex}&\hspace{-4ex}
\int_{\Sigmatb}\Big(4 (\R)^{-2}\abs{\edthR'\edthR \Phiminus{1}}^2
+\abs{\partial_{\rb}(\mu\partial_{\rb})\Phiminus{1}}^2\Big.\notag\\
\hspace{4ex}&\hspace{-4ex}\Big. \qquad
+4(\R)^{-1}\Re\big(\partial_{\rb}(\mu\partial_{\rb}
\overline{\Phiminus{1}})
\edthR'\edthR \Phiminus{1}\big)\Big)\di^3\mu\notag\\
\lesssim{}&\int_{\Sigmatb} \Big(r^{-4}\abs{H^{(1)}(\Phiminus{1})}^2
+a^2r^{-6}\abs{\Phiminus{0}}^2
+a^2r^{-4}\abs{\Leta\Phiminus{0}}^2\Big)\di^3\mu\notag\\
\triangleq {}& \mathbb{G}_3.
\end{align}
\end{subequations}
For the third term in each subequation, we integrate by parts and find
\begin{subequations}
\label{eq:degellip:Phiminus:2:all}
\begin{align}\label{eq:degellip:Phiminus:2}
\hspace{4ex}&\hspace{-4ex}
\int_{\Sigmatb}4(\R)^{-1}\Re\Big(\partial_{\rb}\Big(
(\R)\partial_{\rb}\overline{\Phiminus{0}}\Big)
\edthR'\edthR \Phiminus{0}\Big)\di^3\mu\notag\\
={}&\int_{\Sigmatb}
\Re\Big(-\prb\Big(4 \prb \overline{\edthR \Phiminus{0}}\edthR \Phiminus{0}\Big)
-8r(\R)^{-1} \prb\overline{ \Phiminus{0}}\edthR'\edthR \Phiminus{0}
+4\abs{\edthR\prb \Phiminus{0}}^2\Big)\di^3\mu,\\
\label{eq:degellip:Phiminus:2:1}
\hspace{4ex}&\hspace{-4ex}
\int_{\Sigmatb}4(\R)^{-1}\Re\Big(\partial_{\rb}\big(\mu\partial_{\rb}
\overline{\Phiminus{1}}\big)
\edthR'\edthR \Phiminus{1}\Big)\di^3\mu\notag\\
={}&\int_{\Sigmatb}
\Re\Big(-\prb\Big(4(\R)^{-1}\mu \prb \overline{\edthR \Phiminus{1}}\edthR \Phiminus{1}\Big)
-8r(\R)^{-2}\mu \prb\overline{ \Phiminus{1}}\edthR'\edthR \Phiminus{1}\Big.\notag\\
&\qquad \Big.
+4(\R)^{-1}\mu \abs{\edthR\prb \Phiminus{1}}^2\Big)\di^3\mu.
\end{align}
\end{subequations}
From point \ref{prop:basicesti:pt2} of Proposition \ref{prop:basicesti}, $\int_{\Sigmatb} 4\abs{\edthR\prb \Phiminus{0}}^2\di^3\mu\geq \int_{\Sigmatb}4 \abs{\prb \Phiminus{0}}^2\di^3\mu$, hence the estimates \eqref{eq:degellip:Phiminus:1} and \eqref{eq:degellip:Phiminus:2:all} together imply
\begin{subequations}
\label{eq:degellip:Phiminus:3:all}
\begin{align}
\label{eq:degellip:Phiminus:3}
&\int_{\Sigmatb}r^{-2}\Big(4 (\mu^{-1}-1)\abs{\edthR'\edthR \Phiminus{0}}^2
+\mu\abs{\partial_{\rb}((\R)\partial_{\rb}\Phiminus{0})}^2\Big)
\di^3\mu
\lesssim{}\mathbb{G}_2,\\
\label{eq:degellip:Phiminus:3:1}
&\int_{\Sigmatb}\Big(4 (\R)^{-2}(1-\mu)\abs{\edthR'\edthR \Phiminus{1}}^2
+\abs{\partial_{\rb}(\mu\partial_{\rb}\Phiminus{1})}^2\Big)
\di^3\mu
\lesssim{}\mathbb{G}_3.
\end{align}
\end{subequations}
Applying Cauchy-Schwarz to the second line of each subequation of \eqref{eq:degellip:Phiminus:1} and taking into account of the estimates \eqref{eq:degellip:Phiminus:3:all}, these yield
\begin{subequations}
\label{eq:degellip:Phiminus:4:all}
\begin{align}
\label{eq:degellip:Phiminus:4}
&\int_{\Sigmatb}\Big(\mu^{-1}r^{-2}\abs{\edthR'\edthR \Phiminus{0}}^2
+\mu r^{-2}\abs{\partial_{\rb}(r^2\partial_{\rb}\Phiminus{0})}^2
+\abs{\edthR\prb \Phiminus{0}}^2
\Big)\di^3\mu
\lesssim{}\mathbb{G}_2,\\
\label{eq:degellip:Phiminus:4:1}
&\int_{\Sigmatb}\Big(r^{-4}\abs{\edthR'\edthR \Phiminus{1}}^2
+\abs{\partial_{\rb}(\mu\partial_{\rb}\Phiminus{1})}^2
+r^{-2}\mu \abs{\edthR\prb \Phiminus{1}}^2
\Big)\di^3\mu
\lesssim{}\mathbb{G}_3.
\end{align}
\end{subequations}
Expand the second term on the LHS of \eqref{eq:degellip:Phiminus:4:1}
\begin{align}
\abs{\partial_{\rb}(\mu\partial_{\rb})\Phiminus{1}}^2
={}&\mu^2\abs{\prb^2\Phiminus{1}}^2
+2\mu \prb\mu \Re\Big(\overline{\prb\Phiminus{1}}\prb^2\Phiminus{1}\Big)
+\abs{\prb\mu}^2\abs{\prb\Phiminus{1}}^2\notag\\
={}&\mu^2\abs{\prb^2\Phiminus{1}}^2
+\abs{\prb\mu}^2\abs{\prb\Phiminus{1}}^2
-\prb(\mu\prb\mu)\abs{\prb\Phiminus{1}}^2\notag\\
&+\prb(\mu\prb\mu \abs{\prb\Phiminus{1}}^2).
\end{align}
The integral for the last term vanishes, hence,
\begin{align}
\int_{\Sigmatb}\mu^2\abs{\prb^2\Phiminus{1}}^2\di^3\mu={}&
\int_{\Sigmatb}\Big(\abs{\prb(\mu\prb\Phiminus{1})}^2
+\mu\prb^2\mu \abs{\prb\Phiminus{1}}^2\Big)\di^3\mu.
\end{align}
Since $\prb^2\mu<0$,  we then have from \eqref{eq:degellip:Phiminus:4:1} that
\begin{align}
\label{eq:degellip:Phiminus:6}
\int_{\Sigmatb}\Big(\mu^2\abs{\prb^2\Phiminus{1}}^2
+r^{-4}\abs{\edthR'\edthR \Phiminus{1}}^2
+r^{-2}\mu \abs{\edthR\prb \Phiminus{1}}^2
+r^{-2}\abs{\prb\Phiminus{1}}^2
\Big)\di^3\mu
\lesssim{}&\mathbb{G}_3.
\end{align}
The estimate \eqref{eq:degellip:Phiminus:4} also leads to
\begin{align}
\label{eq:degellip:Phiminus:5}
\int_{\Sigmatb}\Big(
\mu r^2\abs{\prb^2\Phiminus{0}}^2
+\mu^{-1}r^{-2}\abs{\edthR'\edthR \Phiminus{0}}^2
+ \abs{\edthR\prb \Phiminus{0}}^2
\Big)\di^3\mu
\lesssim{}&\mathbb{G}_2.
\end{align}
The integrals of $a^2(r^{-4}\abs{\Leta\Phiminus{0}}^2
+r^{-6}\abs{\Phiminus{0}}^2)$ in $\mathbb{G}_3$ can be bounded by the LHS of the estimate \eqref{eq:degellip:Phiminus:0:9}, thus we get
\begin{align}\label{eq:degellip:Phiminus:7}
\hspace{4ex}&\hspace{-4ex}\int_{\Sigmatb}\Big(
\mu^2 \abs{\partial_{\rb}^2 \Phiminus{1}}^2
+\abs{\partial_{\rb}(\mu\partial_{\rb}\Phiminus{1})}^2
+r^{-4}\abs{\edthR'\edthR \Phiminus{1}}^2
+\mu r^{-2}\abs{\partial_{\rb}\edthR\Phiminus{1}}^2
\Big.\notag\\
\hspace{4ex}&\hspace{-4ex}
\Big. \qquad+\mu (\R)\abs{\prb^2\Phiminus{0}}^2
+\mu^{-1}r^{-2}\abs{\edthR'\edthR \Phiminus{0}}^2
+ \abs{\edthR\prb \Phiminus{0}}^2
\Big)\di^3\mu\notag\\
\lesssim{}&\int_{\Sigmatb}
\Big(\mu^{-1}r^{-2}\abs{\tilde H^{(0)}(\Phiminus{0})}^2
+
r^{-4}\abs{H^{(1)}(\Phiminus{1})}^2\Big)
\di^3\mu,
\end{align}
where Cauchy-Schwarz inequality is applied to bound the term $\mathbb{G}_1$.
One can estimate the integrals of $\tilde H^{(0)}(\Phiminus{0})$ and $H^{(1)}(\Phiminus{1})$ using the expressions \eqref{def:HitildePhiminus} and \eqref{def:HiPhiminus}, and it holds from \eqref{def:vectorVRintermsofprb} that
\begin{align}
\label{eq:LxiprbintermsofLxiandLeta}
r^2\abs{\Lxi\prb\Phiminus{0}}^2\lesssim{}&
r^{-2}\Big(\abs{\Lxi\Phiminus{1}}^2
+\abs{\Lxi^2\PsiminusHigh{0}}^2
+a^2\abs{\Lxi\Leta\PsiminusHigh{0}}^2\Big).
\end{align}
Near horizon, it is manifest from \eqref{eq:LxiprbintermsofLxiandLeta} and \eqref{eq:propertyofHfunction} that
\begin{align}
\hspace{4ex}&\hspace{-4ex}\mu^{-1}\abs{2(\R)(\mu H-1)\Lxi\prb\Phiminus{0}
+
\mu\prb((\R)H)\Lxi \Phiminus{0}}^2\notag\\
={}&
\mu^{-1}\abs{\mu\Lxi\prb((\R)H\Phiminus{0})
+(\R)(\mu H-2)\Lxi\prb\Phiminus{0}}^2\notag\\
\lesssim{}&\mu \Big(\abs{\Lxi\prb\PsiminusHigh{0}}^2
+ \abs{\Lxi\PsiminusHigh{0}}^2
+\abs{\Lxi\Phiminus{1}}^2
+\abs{\Lxi^2\PsiminusHigh{0}}^2
+a^2\abs{\Lxi\Leta\PsiminusHigh{0}}^2\Big);
\end{align}
hence
the RHS of \eqref{eq:degellip:Phiminus:7} is bounded by
\begin{align}
&C\int_{\Sigmatb}\Big(
\abs{\Lxi\prb\Phiminus{1}}^2
+r^{-4}\big(\abs{\Lxi^2 \Phiminus{1}}^2
+a^2\abs{\Lxi\Leta \Phiminus{1}}^2\big)
+r^{-2}\abs{\Lxi \Phiminus{1}}^2
\Big.\notag\\
&\quad\qquad\Big.
+\mu r^{-2}\big(\abs{\Lxi\prb\PsiminusHigh{0}}^2
+\abs{\Lxi^2 \PsiminusHigh{0}}^2
+a^2\abs{\Lxi\Leta \PsiminusHigh{0}}^2
+\abs{\Lxi \PsiminusHigh{0}}^2\big)\Big.\notag\\
&\quad \qquad\Big.
+
a^2r^{-4}\abs{\Leta\prb\Phiminus{1}}^2
+
a^2\mu\abs{\Leta\prb\psiminusHigh{0}}^2
\Big)\di^3\mu.
\end{align}
We thus conclude
\begin{align}\label{eq:degellip:Phiminus}
\hspace{4ex}&\hspace{-4ex}\int_{\Sigmatb}\Big(
\mu^2 \abs{\partial_{\rb}^2 \Phiminus{1}}^2
+r^{-4}\abs{\edthR'\edthR \Phiminus{1}}^2
+\mu r^{-2}\abs{\partial_{\rb}\edthR\Phiminus{1}}^2
+r^{-2}\abs{\partial_{\rb}\Phiminus{1}}^2\Big.\notag\\
\hspace{4ex}&\hspace{-4ex}
\Big. \qquad+\mu r^2\abs{\prb^2\Phiminus{0}}^2
+\mu^{-1}r^{-2}\abs{\edthR'\edthR \Phiminus{0}}^2
+ \abs{\edthR\prb \Phiminus{0}}^2
\Big)\di^3\mu\notag\\
\lesssim{}&\int_{\Sigmatb}\Big(
\abs{\Lxi\prb\Phiminus{1}}^2
+r^{-2}\big(\abs{\Lxi^2 \Phiminus{1}}^2
+\abs{\Lxi\Leta \Phiminus{1}}^2
+\abs{\Lxi \Phiminus{1}}^2\big)
\Big.\notag\\
&\qquad\Big.
+\mu r^{-2}\big(\abs{\Lxi\prb\PsiminusHigh{0}}^2
+\abs{\Lxi^2 \PsiminusHigh{0}}^2
+\abs{\Lxi\Leta \PsiminusHigh{0}}^2
+\abs{\Lxi \PsiminusHigh{0}}^2\big)\Big.\notag\\
&\qquad \Big.
+
a^2r^{-4}\abs{\Leta\prb\Phiminus{1}}^2
+
a^2\mu \abs{\Leta\prb\psiminusHigh{0}}^2)
\Big)\di^3\mu.
\end{align}
The RHS is bounded using the Hardy's inequality \eqref{eq:Hardy:trivial:1} by
\begin{align}
&\int_{\Sigmatb}
\Big(\sum_{i=0,1}\Big(
\abs{\Lxi\prb\psiminusHigh{i}}^2
+\abs{\Lxi\psiminusHigh{i}}^2
+\abs{\Lxi^2\psiminusHigh{i}}^2
+\abs{\Lxi\Leta\psiminusHigh{i}}^2\Big)
+r^2\abs{\Lxi\prb\phiminus{1}}^2
\notag\\
&\quad
+a^2r^{-2}\abs{\Leta\prb\phiminus{1}}^2
+a^2r^{-4}\abs{\Leta\phiminus{1}}^2
+
a^2\mu \abs{\Leta\prb\psiminusHigh{0}}^2)\Big)\di^3\mu.
\end{align}
The first term in the second line can be bounded by
\begin{align}
\int_{\Sigmatb}\frac{a^2}{r^2}\abs{\Leta\prb\phiminus{1}}^2\di^3\mu
\lesssim{}&
\int_{\Sigmatb}\frac{a^2}{r^2}(\abs{\Leta\prb\psiminusHigh{1}}^2
+\abs{\Leta\prb\psiminusHigh{0}}^2
+\abs{\Leta\psiminusHigh{0}}^2)\di^3\mu\notag\\
\lesssim{}&\int_{\Sigmatb}\frac{a^2}{r^2}\bigg(
\sum_{i=0,1}\abs{\Leta\prb\psiminusHigh{i}}^2
+\mu\abs{\edthR'\edthR\psiminusHigh{0}}^2\bigg)\di^3\mu,
\end{align}
and by taking $|a|/M\leq \veps_0$ sufficiently small, the integral of $a^2r^{-2}\mu \abs{\edthR'\edthR\psiminusHigh{0}}^2+a^2r^{-4}\abs{\Leta\phiminus{1}}^2$ over $\Sigmatb$ and the integral of $a^2\mu \abs{\Leta\prb\psiminusHigh{0}}^2$ away from horizon are absorbed, hence
proving the estimate \eqref{eq:degellip:phiminus:lowerreg}.

It is manifest that the estimate \eqref{eq:degellip:phiminus} holds true if we replace $\CDeri$ by $X=\{\Lxi,\Leta\}$. Further, we commute $\edthR$ and $\edthR'$ with equations \eqref{eq:Phiminus01:hyperboloidal}, put the resulting commutators into the RHS and obtain
\begin{align}
\tilde{H}^{(0)}(\edthR\Phiminus{0})
={}\edthR(\tilde{H}^{(0)}(\Phiminus{0}))
+O(1)\edthR\Phiminus{0}
+O(1)\edthR'\Phiminus{0}
+O(1)\Lxi\Phiminus{0}
+O(1)\Lxi^2\Phiminus{0}.
\end{align}
The terms $\tilde{H}^{(0)}(\edthR'\Phiminus{0})$, ${H}^{(1)}(\edthR\Phiminus{0})$ and ${H}^{(1)}(\edthR'\Phiminus{0})$ have a similar expression. By going through the above discussions, we obtain similar estimates as \eqref{eq:degellip:Phiminus:7}, and the integral on the RHS of these commutators are bounded by the LHS of \eqref{eq:degellip:phiminus:lowerreg}. Hence the estimate \eqref{eq:degellip:phiminus} with $\CDeri$ replaced by $\mathbf{X}=\{\Lxi,\Leta,\edthR,\edthR'\}$ holds. In the end, we commute $r\prb$ with equations \eqref{eq:Phiminus01:hyperboloidal} and use the proven  estimates to obtain a bound over the integrals away from horizon of more $\prb$ derivatives acting on $\phiminus{1}$ and $\phiminus{0}$.
\end{proof}

The estimate \eqref{eq:degellip:phiminus} is a degenerate elliptic estimate because of the presence of the factor $\mu$ which degenerates at $\Horizon$. This degeneracy can be removed by an application of the red-shift estimates of spin $-1$ component. To state the red-shift estimates near horizon, let us define a wave operator
\begin{align}\label{eq:RewrittenFormofSWRWE}
\Sigma\widetilde{\Box}_g\varphi
\triangleq{}&\left\{\partial_r(\Delta\partial_r)
-\tfrac{\left((r^2+a^2)\partial_t+a\partial_{\phi}\right)^2}{\Delta}\right.\notag\\
&\left.\quad+
\tfrac{1}{\sin \theta}\tfrac{d}{d\theta}\left(\sin \theta \tfrac{d}{d\theta}\right)+\left(\tfrac{\partial_{\phi}+is\cos \theta }{\sin \theta}+a\sin \theta\partial_{t}\right)^2\right\}\varphi,
\end{align}
such that $\Sigma\widetilde{\Box}_g$ is the same as the rescaled scalar wave operator $\Sigma \Box_g$, except for $(\tfrac{\partial_{\phi}}{\sin \theta}+is\cot \theta +a\sin \theta\partial_{t})^2$ in place of the operator $(\tfrac{\partial_{\phi}}{\sin \theta}+a\sin \theta\partial_{t})^2$ in the expansion of $\Sigma \Box_g$.
We shall also define a different set of operators $\RDeri=\{\curlY, \Lxi,\Leta,\edthR,\edthR'\}$.

\begin{lemma}\label{pro:redshift:general:smalla}
Let $s\neq 0$, and let $\varphi$ be a spin weight $s$ scalar satisfying
\begin{align}
\Sigma\widetilde{\Box}_g\varphi={}&\vartheta.
\end{align}
Let $1\leq \reg\in \mathbb{N}$.
Then there exists an $\veps_0>0$, two constants $r_+<r_0(\veps_0,M)<r_1(\veps_0,M)$ with $\abs{r_1 -r_+}$ suitably small, and constants $D_i \gg 1$ for $i=1,\ldots, \reg-1$ such that for all $\abs{a}/M\leq \veps_0$ and any $C_i$ with $C_1=1$ and $C_i/C_{i+1}\geq D_i$,
\begin{align}
\label{eq:redshift:general:smalla}
\hspace{4ex}&\hspace{-4ex}
c\Big(\norm{\varphi}^2_{W_{0}^{\reg}(\Sigmatwo^{\leq r_0})}
+\norm{\varphi}^2_{W_{0}^{\reg}(\Donetwo^{\leq r_0})}\Big)\notag\\
\leq{}&
\norm{\varphi}^2_{W_{0}^{\reg}(\Sigmaone^{\leq r_1})}
+\norm{\varphi}^2_{W_{0}^{\reg}(\Donetwo^{r_0,r_1})}
-\sum_{\reg_1=1}^{\reg}C_{\reg_1}\sum_{\abs{\mathbf{a}}=\reg_1-1}
\int_{\Donetwo^{\leq r_1}}
\Re(\overline{\RDeri^{\mathbf{a}}\vartheta}
N\RDeri^{\mathbf{a}}\varphi)\di^4\mu.
\end{align}
Here, $N=f_1(r)Y +f_2(r)\Lxi$ is a timelike vector field, with $f_1(r)$ and $f_2(r)$ being  two real smooth functions and satisfying $f_1, f_2\to 1$ as $r\to r_+$ and $f_1=0$, $f_2=1$ for $r\geq r_1$.
\end{lemma}

\begin{proof}
As is shown in \cite[Lemma 3.3]{Ma2017Maxwell}, there exists $\veps_0>0$ and such a vector field $N$ such that it can be used to achieve the following red-shift estimate for all $\abs{a}/M\leq \veps_0$,
\begin{align}
\hspace{4ex}&\hspace{-4ex}c\Big(\norm{\varphi}^2_{W_{0}^{1}
(\Sigmatwo^{\leq r_0})}
+\norm{\varphi}^2_{W_{0}^{1}(\Donetwo^{\leq r_0})}\Big)\notag\\
\leq{}&
-\int_{\Donetwo^{\leq r_1}}
\Re(\overline{\vartheta}
N\varphi)\di^4\mu+
C\Big(\norm{\varphi}^2_{W_{0}^{1}(\Sigmaone^{\leq r_1})}
+\norm{\varphi}^2_{W_{0}^{1}(\Donetwo^{r_0,r_1})}\Big).
\end{align}
This proves $\reg=1$ case. For general $\reg\geq 1$ cases, we use a few commutation relations. The Killing vectors $\Lxi$ and $\Leta$ commute with the wave equation. It is easy to prove by induction that for any $n\geq 0$, there exists $\kappa_n\geq 0$ and smooth functions $c_{\mathbf{a},\mathbf{b}}(\theta, r)$ where $\mathbf{b}=(b_1,b_2,b_3)$ such that
\begin{align}
\label{eq:BoxtildecommutewithY}
[\Sigma\widetilde{\Box}_g, Y^n]\varphi={}&\kappa_{n}Y^{n+1}\varphi
+\sum_{\abs{\mathbf{a}}+\abs{\mathbf{b}}\leq n+1, b_3\neq n+1}c_{\mathbf{a},\mathbf{b}}
\SDeri^{\mathbf{a}}\Lxi^{b_1}\Leta^{b_2}Y^{b_3}\varphi,
\end{align}
and for any $\abs{\mathbf{d}}=n$, there exists smooth functions ${c}_{\mathbf{a},\mathbf{x}}(\theta,r)$ where $\mathbf{x}=(x_1,x_2)$ such that
\begin{align}
\label{eq:Boxtiledcommtewithangular}
[\Sigma\widetilde{\Box}_g, \SDeri^{\mathbf{d}}]\varphi={}&\sum_{\abs{\mathbf{a}}+\abs{\mathbf{x}}\leq n+1,\abs{\mathbf{a}}\neq n+1}{c}_{\mathbf{a},\mathbf{x}}\SDeri^{\mathbf{a}}\Leta^{x_1}
\Lxi^{x_2}\varphi.
\end{align}
Assume the estimate \eqref{eq:redshift:general:smalla} holds for $\reg=\reg_1$,  we show it is valid for $\reg=\reg_1+1$. First, $\Lxi$ commutes with the wave equation, hence the estimate \eqref{eq:redshift:general:smalla} holds true by replacing $\varphi$ by $\Lxi\varphi$. Second, by commuting the wave equation with $\{\edthR, \edthR'\}$, the commutator in \eqref{eq:Boxtiledcommtewithangular} can be controlled by the above estimates. This requires to add a large multiple of the estimate with $\reg=\reg_1$ to bound the error terms arising from the commutator, and this is why we require $C_{i}/C_{i+1}$ to be suitably large. Finally, we commute with $Y$ and find the last term in the commutation relation \eqref{eq:BoxtildecommutewithY} is already bounded. Therefore, this proves the estimate \eqref{eq:redshift:general:smalla} for $\reg=\reg_1+1$ and completes the proof by induction.
\end{proof}

We can now apply this estimate and obtain a red-shift estimate of $\prb\psiminusHigh{i}$ near horizon.

\begin{prop}
There exists an $\veps_0>0$, two constants $r_+<r_0(\veps,M)<r_1(\veps,M)$ with $\abs{r_1 -r_+}$ suitably small such that for all $\abs{a}/M\leq \veps_0$ and any $\reg\geq 1$,
\begin{align}
\label{eq:redshift:NPsiminus:smalla}
\hspace{4ex}&\hspace{-4ex}\sum_{i=0,1}\Big(
\norm{\prb\psiminusHigh{i}}^2_{W_{0}^{\reg}(\Sigmatwo^{\leq r_0})}
+\norm{\Lxi\psiminusHigh{i}}^2_{W_{0}^{\reg}(\Sigmatwo^{\leq r_0})}
+\norm{\prb\psiminusHigh{i}}^2_{W_{0}^{\reg}(\Donetwo^{\leq r_0})}\Big)\notag\\
\lesssim{}&
\sum_{i=0,1}\Big(
\norm{\prb\psiminusHigh{i}}^2_{W_{0}^{\reg}(\Sigmaone^{\leq r_1})}
+\norm{\prb\psiminusHigh{i}}^2_{W_{0}^{\reg}(\Donetwo^{r_0,r_1})}
+
\norm{\psiminusHigh{i}}^2_{W_{0}^{\reg}(\Donetwo^{r_0,r_1})}
\notag\\
&
+\norm{\Lxi\psiminusHigh{i}}^2_{W_{0}^{\reg}(\Sigmaone^{\leq r_1})}
+a^2\norm{\Leta\psiminusHigh{i}}^2_{W_{0}^{\reg}(\Donetwo^{r_0,r_1})}
+\norm{\Lxi\psiminusHigh{i}}^2_{W_{0}^{\reg}(\Donetwo^{\leq r_1})}\Big).
\end{align}
\end{prop}

\begin{proof}
The equation of $\psiminusHigh{0}$ is (see \cite[Equation (4.26)]{Ma2017Maxwell})
\begin{align}
\label{eq:RSpsiminus0:1}
\Sigma\widetilde{\Box}_g\psiminusHigh{0}
={}&\vartheta(\psiminusHigh{0})\notag\\
={}&\Big(\tfrac{3}{2}-\tfrac{5a^2M}{2r(\R)}\Big)\psiminusHigh{0}+
\left(\tfrac{4(r-M)r-5\Delta}{2r}Y
+r\Lxi\right)\psiminusHigh{0}\notag\\
&-4ia\cos\theta\Lxi\psiminusHigh{0}
+\tfrac{5}{r}\left(a^2\Lxi+a\Leta\right)
\psiminusHigh{0}
-\tfrac{5}{2r}\phiminus{1}.
\end{align}
Since $\Sigma\widetilde{\Box}_g=\mathbf{L}_s+4ias\cos\theta\Lxi +\tfrac{r^4-2Mr^3+6a^2 Mr -a^4}{(\R)^2}$ with $\mathbf{L}_s$ as defined in \cite[Equation (1.32)]{Ma2017Maxwell}, one has from Remark \ref{rem:squareShatandLs} that
\begin{align}
\label{eq:BoxHatandLs}
\squareShat_{s} (\sqrt{\R}\varphi)={}&
\sqrt{\R}(\Sigma\widetilde{\Box}_g-4ias\cos\theta\Lxi-s)\varphi.
\end{align}
Together with \eqref{eq:Phi-11}, this gives an equation of $\phiminus{1}$
\begin{align}
\label{eq:RSphiminus1:1}
\Sigma\widetilde{\Box}_g\phiminus{1} ={}&\vartheta(\phiminus{1})\notag\\
={}&\tfrac{2(r^4-Mr^3+a^2r^2+3a^2Mr)-(\R)^2}{(\R)^2}\phiminus{1}
-4ia\cos\theta\Lxi\phiminus{1}\notag\\
&+\tfrac{2\mu a^2(\PR)}{(\R)^2}\psiminusHigh{0}
-\tfrac{2\mu a(r^2-a^2)}{\R}\partial_{\phi}\psiminusHigh{0}
.
\end{align}
In particular, we observe that $\tfrac{4(r-M)r-5\Delta}{2r}Y
+r\Lxi$ is close to a positive multiple of $N$ near horizon. We commute $\Lxi$ with  equations \eqref{eq:RSpsiminus0:1} and \eqref{eq:RSphiminus1:1}, apply the estimate \eqref{eq:redshift:general:smalla} to the gained equations and find
\begin{subequations}
\label{eq:errorRS:Lxiphiminus01}
\begin{align}
\label{eq:errorRS:Lxipsiminus0}
\hspace{4ex}&\hspace{-4ex}
-\sum_{\abs{\mathbf{a}}=\reg_1-1}\int_{\Donetwo^{\leq r_1}}
\Re\Big(\overline{\RDeri^{\mathbf{a}}
\vartheta(\Lxi\psiminusHigh{0})}
N\RDeri^{\mathbf{a}}\Lxi\psiminusHigh{0}\Big)\di^4\mu\notag\\
\lesssim{}& a^2\norm{\Lxi\psiminusHigh{0}}^2_{W_{0}^{\reg_1}(\Donetwo^{\leq r_1})}
+\norm{\Lxi\psiminusHigh{0}}^2_{W_{0}^{\reg_1-1}(\Donetwo^{\leq r_1})}
+\norm{\Lxi\phiminus{1}}^2_{W_{0}^{\reg_1-1}(\Donetwo^{\leq r_1})},\\
\label{eq:errorRS:Lxiphiminus1}
\hspace{4ex}&\hspace{-4ex}
-\sum_{\abs{\mathbf{a}}=\reg_1-1}\int_{\Donetwo^{\leq r_1}}
\Re\Big(\overline{\RDeri^{\mathbf{a}}\vartheta(\Lxi\phiminus{1} )}
N\RDeri^{\mathbf{a}}\Lxi\phiminus{1} \Big)\di^4\mu\notag\\
\lesssim{}& (\veps+a^2)\norm{\Lxi\phiminus{1} }^2_{W_{0}^{\reg_1}(\Donetwo^{\leq r_1})}
+\veps^{-1}\Big(
a^2\norm{\mu\Lxi\Leta\psiminusHigh{0}}^2_{W_{0}^{\reg_1-1}
(\Donetwo^{\leq r_1})}\notag\\
&
+a^2\norm{\mu\Lxi\psiminusHigh{0}}^2_{W_{0}^{\reg_1-1}
(\Donetwo^{\leq r_1})}
+\norm{\Lxi\phiminus{1} }^2_{W_{0}^{\reg_1-1}(\Donetwo^{\leq r_1})}
\Big).
\end{align}
\end{subequations}
The terms $(\veps+a^2)\norm{\Lxi\phiminus{1} }^2_{W_{0}^{\reg_1}(\Donetwo^{\leq r_0})}$ and $a^2\norm{\Lxi\psiminusHigh{0}}^2_{W_{0}^{\reg_1}(\Donetwo^{\leq r_0})}$ can be absorbed by the LHS of the estimate \eqref{eq:redshift:general:smalla} of $\Lxi\phiminus{1}$ and of  $\Lxi\psiminusHigh{0}$ respectively by taking $\veps$ and $\abs{a}/M\leq \veps_0$ suitably small.
Add a large multiple of the estimate \eqref{eq:redshift:general:smalla} of $\Lxi\phiminus{1}$ to the one of $\Lxi\psiminusHigh{0}$ and take $C_{i}/C_{i+1}$ sufficiently large, then for sufficiently small $\veps_0$,
the last two terms of \eqref{eq:errorRS:Lxipsiminus0} and the last three terms of \eqref{eq:errorRS:Lxiphiminus1} but with integral region $\Donetwo^{\leq r_0}$ can be absorbed. That is, after making such choice of the parameters, all the terms on the RHS of \eqref{eq:errorRS:Lxiphiminus01} with integral region $\Donetwo^{\leq r_0}$ are absorbed.
In conclusion,
\begin{align}
\label{eq:RS:Lxiphiminus1psiminus0}
\hspace{4ex}&\hspace{-4ex}
\norm{\Lxi\psiminusHigh{0}}^2_{W_{0}^{\reg}(\Sigmatwo^{\leq r_0})}
+\norm{\Lxi\phiminus{1}}^2_{W_{0}^{\reg}(\Sigmatwo^{\leq r_0})}
+\norm{\Lxi\psiminusHigh{0}}^2_{W_{0}^{\reg}
(\Donetwo^{\leq r_0})}
+\norm{\Lxi\phiminus{1}}^2_{W_{0}^{\reg}
(\Donetwo^{\leq r_0})}\notag\\
\lesssim{}&
\norm{\Lxi\psiminusHigh{0}}^2_{W_{0}^{\reg}(\Sigmaone^{\leq r_1})}
+\norm{\Lxi\phiminus{1}}^2_{W_{0}^{\reg}(\Sigmaone^{\leq r_1})}
+\norm{\Lxi\psiminusHigh{0}}^2_{W_{0}^{\reg}(\Donetwo^{r_0,r_1})}
+\norm{\Lxi\phiminus{1}}^2_{W_{0}^{\reg}(\Donetwo^{r_0,r_1})}.
\end{align}

From the commutation relation \cite[Proposition A.1]{Ma2017Maxwell}, the commutator of $\curlY$ and $\Sigma\widetilde{\Box}_g$ is
\begin{align}
[\Sigma\widetilde{\Box}_g, \curlY]\varphi={}&
\curlY\left((\R)\partial_r\left(\tfrac{\Delta}
{(\R)^2}\right)\curlY\varphi\right)
+\tfrac{4ar}{\R}\Leta\curlY \varphi
-\tfrac{2a^3}{\R}\Leta\varphi\notag\\
&
-\Big((\R) (\partial_r \mu) +a^2(\R)\partial_r\left(\tfrac{3r^2-10Mr +3a^2}{(\R)^2}\right)\Big)\varphi,
\end{align}
which yields
\begin{align}
\label{eq:generalcommutator:BoxtildeandcurlY}
\Sigma\widetilde{\Box}_g(\curlY\varphi)
={}&\vartheta(\curlY\varphi)
=\curlY\vartheta(\varphi)
+[\Sigma\widetilde{\Box}_g, \curlY]\varphi\notag\\
={}&\curlY\vartheta(\varphi)
+(\R)^2\partial_r\left(\tfrac{\Delta}
{(\R)^2}\right)Y\curlY\varphi\notag\\
&
-\sqrt{\R}\partial_r\left((\R)^{\frac{3}{2}}\partial_r\left(\tfrac{\Delta}
{(\R)^2}\right)\right)\curlY\varphi
+\tfrac{4ar}{\R}\Leta\curlY \varphi
\notag\\
&
-\tfrac{2a^3}{\R}\Leta\varphi-(\R) \Big( \partial_r \mu +a^2\partial_r\left(\tfrac{3r^2-10Mr +3a^2}{(\R)^2}\right)\Big)\varphi.
\end{align}
One can then derive the explicit forms of the wave equations of $\curlY \psiminusHigh{0}$ and $\curlY\phiminus{1}$ and apply the estimate \eqref{eq:redshift:general:smalla}. Note that the coefficient $(\R)^2\partial_r\big(\tfrac{\Delta}
{(\R)^2}\big)$ of the term $Y\curlY\varphi$ is negative near horizon and we rewrite $Y\curlY\varphi$ in terms of $N\curlY\varphi$ and $\Lxi\curlY\varphi$, hence,
\begin{align}
\label{eq:errorRS:curlYpsiminus0}
\hspace{4ex}&\hspace{-4ex}
-\sum_{\abs{\mathbf{a}}=\reg_1-1}\int_{\Donetwo^{\leq r_1}}
\Re\Big(\overline{\RDeri^{\mathbf{a}}\vartheta(\curlY\psiminusHigh{0})}
N\RDeri^{\mathbf{a}}\curlY\psiminusHigh{0}\Big)\di^4\mu\notag\\
\lesssim{}& \veps_{\reg_1-1}\norm{\curlY\psiminusHigh{0}}^2_{W_{0}^{\reg_1}
(\Donetwo^{\leq r_1})}
+\veps_{\reg_1-1}^{-1}\Big(
\norm{\curlY\psiminusHigh{0}}^2_{W_{0}^{\reg_1-1}(\Donetwo^{\leq r_1})}
+a^2 \norm{\Leta\curlY\psiminusHigh{0}}^2_{W_{0}^{\reg_1-1}
(\Donetwo^{\leq r_1})}\notag\\
&
+a^2\norm{\Leta\psiminusHigh{0}}^2_{W_{0}^{\reg_1-1}
(\Donetwo^{\leq r_1})}
+\norm{\psiminusHigh{0}}^2_{W_{0}^{\reg_1-1}(\Donetwo^{\leq r_1})}
+\norm{\Lxi\psiminusHigh{0}}^2_{W_{0}^{\reg_1}(\Donetwo^{\leq r_1})}\notag\\
&+\norm{\curlY\phiminus{1}}^2_{W_{0}^{\reg_1-1}(\Donetwo^{\leq r_1})}
+\norm{\phiminus{1}}^2_{W_{0}^{\reg_1-1}(\Donetwo^{\leq r_1})}
\Big),\\
\label{eq:errorRS:curlYphiminus1}
\hspace{4ex}&\hspace{-4ex}
-\sum_{\abs{\mathbf{a}}=\reg_1-1}\int_{\Donetwo^{\leq r_1}}
\Re\Big(\overline{\RDeri^{\mathbf{a}}\vartheta(\curlY\phiminus{1} )}
N\RDeri^{\mathbf{a}}\curlY\phiminus{1} \Big)\di^4\mu\notag\\
\lesssim{}& \veps_{\reg_1-1}\norm{\curlY\phiminus{1}}^2_{W_{0}^{\reg_1}(\Donetwo^{\leq r_1})}
+\veps_{\reg_1-1}^{-1}\Big(
\norm{\curlY\phiminus{1}}^2_{W_{0}^{\reg_1-1}(\Donetwo^{\leq r_1})}
+a^2 \norm{\Leta\curlY\phiminus{1}}^2_{W_{0}^{\reg_1-1}(\Donetwo^{\leq r_1})}\notag\\
&
+a^2\norm{\Leta\phiminus{1}}^2_{W_{0}^{\reg_1-1}(\Donetwo^{\leq r_1})}
+\norm{\phiminus{1}}^2_{W_{0}^{\reg_1-1}(\Donetwo^{\leq r_1})}
+\norm{\Lxi\phiminus{1}}^2_{W_{0}^{\reg_1}(\Donetwo^{\leq r_1})}\notag\\
&+a^2\sum_{i=0,1}\Big(\norm{\mu\curlY\Leta^i\psiminusHigh{0}}^2_{W_{0}^{\reg_1-1}(\Donetwo^{\leq r_1})}
+\norm{\Leta^i\psiminusHigh{0}}^2_{W_{0}^{\reg_1-1}(\Donetwo^{\leq r_1})}\Big)
\Big).
\end{align}
We choose $\veps_{\reg_1-1}$ suitably small such that the terms with $\veps_{\reg_1-1}$ coefficient on the RHS of \eqref{eq:errorRS:curlYpsiminus0} and \eqref{eq:errorRS:curlYphiminus1} but with integral region $\Donetwo^{\leq r_0}$ are absorbed, and add a large multiple of the estimate \eqref{eq:redshift:general:smalla} of $\varphi=\curlY\phiminus{1}$ to the corresponding estimate \eqref{eq:redshift:general:smalla} of $\varphi=\curlY \psiminusHigh{0}$ and take $\abs{a}/M\leq \veps_0$ suitably small such that the first term in the second last line and the terms in the last line of both \eqref{eq:errorRS:curlYpsiminus0} and \eqref{eq:errorRS:curlYphiminus1} with integrals over $\Donetwo^{r_+,r_0}$ are absorbed. Here, we used again the Hardy's inequality \eqref{eq:Hardy:trivial}. Furthermore, the terms $\veps_{\reg_1-1}^{-1}a^2 \norm{\Leta\curlY\psiminusHigh{0}}^2_{W_{0}^{\reg_1-1}
(\Donetwo^{\leq r_0})}$ and $\veps_{\reg_1-1}^{-1}a^2 \norm{\Leta\curlY\phiminus{1}}^2_{W_{0}^{\reg_1-1}
(\Donetwo^{\leq r_0})}$  can also be absorbed by taking $|a|/M\leq\veps_0$ sufficiently small. The remaining terms on the RHS of the estimates \eqref{eq:errorRS:curlYpsiminus0} and \eqref{eq:errorRS:curlYphiminus1} with integrals over $\Donetwo^{\leq r_0}$ can then be absorbed by taking $C_i/C_{i+1}$ sufficiently large. In conclusion, one arrives at
\begin{align}
\label{eq:redshift:curlYphiminus1psiminus0}
\hspace{4ex}&\hspace{-4ex}
\sum_{\varphi\in\{\psiminusHigh{0},\phiminus{1}\}}
\Big(\norm{\curlY\varphi}^2_{W_{0}^{\reg}(\Sigmatwo^{\leq r_0})}
+\norm{\curlY\varphi}^2_{W_{0}^{\reg}(\Donetwo^{\leq r_0})}
\Big)\notag\\
\lesssim{}&
\sum_{\varphi\in\{\psiminusHigh{0},\phiminus{1}\}}
\Big(\norm{\curlY\varphi}^2_{W_{0}^{\reg}(\Sigmaone^{\leq r_1})}
+\norm{\curlY\varphi}^2_{W_{0}^{\reg}(\Donetwo^{r_0,r_1})}
\notag\\
&+
\norm{\varphi}^2_{W_{0}^{\reg+1}(\Donetwo^{r_0,r_1})}
+a^2\norm{\Leta\varphi}^2_{W_{0}^{\reg}(\Donetwo^{r_0,r_1})}
+\norm{\Lxi\varphi}^2_{W_{0}^{\reg}(\Donetwo^{\leq r_1})}
\Big).
\end{align}
Since $\prb$ is a linear combination of $Y$ and $\Lxi$ with $O(1)$ coefficients, the Hardy's inequality \eqref{eq:Hardy:trivial} implies that for any $\tb$ and any $r'>r_+$,
\begin{align}
\label{eq:curlYLxiequivalenttoNLxi}
\norm{\curlY\varphi}^2_{W_{0}^{\reg}(\Sigmatb^{\leq r'})}
+\norm{\Lxi\varphi}^2_{W_{0}^{\reg}(\Sigmatb^{\leq r'})}
\sim {}&\norm{Y\varphi}^2_{W_{0}^{\reg}(\Sigmatb^{\leq r'})}
+\norm{\Lxi\varphi}^2_{W_{0}^{\reg}(\Sigmatb^{\leq r'})}\notag\\
\sim {}&\norm{\prb\varphi}^2_{W_{0}^{\reg}(\Sigmatb^{\leq r'})}
+\norm{\Lxi\varphi}^2_{W_{0}^{\reg}(\Sigmatb^{\leq r'})}.
\end{align}
The estimate \eqref{eq:redshift:NPsiminus:smalla} then follows from adding the two estimates \eqref{eq:RS:Lxiphiminus1psiminus0} and \eqref{eq:redshift:curlYphiminus1psiminus0} together and making use of \eqref{eq:curlYLxiequivalenttoNLxi}.
\end{proof}

\begin{prop}\label{prop:improveEnerDecay:phiminus}
There exists $\veps_0>0$ and $\regl>0$ such that for any $\abs{a}/M\leq \veps_0$ and any $\tb\geq \tb_0$,
\begin{align}
\label{eq:improveEnerDecay:phiminus}
\sum_{i=0,1}
\bigg(\norm{\prb\Lxi^j\psiminusHigh{i}}^2_{W_{2-2i}^{\reg-\regl}(\Sigmatb)}
+\norm{\Lxi^j\psiminusHigh{i}}^2_{W_{-2i}^{\reg-\regl}(\Sigmatb)}\bigg)
\lesssim {}&D_{-1}\tb^{-4-\gamma-2j}.
\end{align}
\end{prop}

\begin{proof}
We consider only $j=0$ case, the $j>0$ cases being analogous.
We integrate the estimate \eqref{eq:degellip:phiminus} on $\tb\in [\tb_1,\tb_2]$ and add $C_0$ multiple of the gained inequality to the estimate \eqref{eq:redshift:NPsiminus:smalla}, then by taking $C_0$ large enough and for sufficiently small $\abs{a}\leq a_0$, the terms $a^2 C_0 \norm{\Leta\prb\psiminusHigh{i}}^2_{W_{-2}^{\reg-1}(\Donetwo)}$
and $\norm{\prb\psiminusHigh{i}}^2_{W_{0}^{\reg}(\Donetwo^{r_0,r_1})}
+\norm{\psiminusHigh{i}}^2_{W_{0}^{\reg}(\Donetwo^{r_0,r_1})}
+a^2\norm{\Leta\psiminusHigh{i}}^2_{W_{0}^{\reg}(\Donetwo^{r_0,r_1})}$
are absorbed. Adding furthermore the estimate \eqref{eq:degellip:phiminus} on $\Sigmatwo$ in yields
\begin{align}
\hspace{4ex}&\hspace{-4ex}\sum_{i=0,1}
\bigg(\norm{\prb\psiminusHigh{i}}^2_{W_{0}^{\reg}(\Sigmatwo)}
+\norm{\prb\psiminusHigh{i}}^2_{W_{0}^{\reg}(\Donetwo)}\bigg)\notag\\
\lesssim{}&
\sum_{i=0,1}\bigg(\norm{\prb\psiminusHigh{i}}^2_{W_{0}^{\reg}(\Sigmaone)}
+\norm{\Lxi\psiminusHigh{i}}^2_{W_{0}^{\reg}(\Sigmatwo)}
+\norm{\Lxi\psiminusHigh{i}}^2_{W_{0}^{\reg}(\Sigmaone)}
+\norm{\Lxi\psiminusHigh{i}}^2_{W_{0}^{\reg}(\Donetwo)}
\bigg).
\end{align}
Note from \eqref{eq:BEDC:Phiminus:1:v1} that
\begin{align}\label{eq:Lxienergybound:psiminus}
\sum_{i=0,1}\norm{\Lxi\psiminusHigh{i}}^2_{W_{0}^{\reg}(\Sigmatwo)}\lesssim
\sum_{i=0,1}\norm{\Lxi\psiminusHigh{i}}^2_{W_{0}^{\reg}(\Sigmaone)}.
\end{align}
This implies for $f_{\tb}^{\reg}(\psiminus)=\sum\limits_{i=0,1}(\norm{\prb\psiminusHigh{i}}^2_{W_{0}^{\reg}(\Sigmatb)}
+\norm{\Lxi\psiminusHigh{i}}^2_{W_{0}^{\reg}(\Sigmatb)})$,
\begin{align}
f_{\tb_2}^{\reg}(\psiminus)+\int_{\tb_1}^{\tb_2}f_{\tb}^{\reg}(\psiminus)\di \tb\lesssim{}& f_{\tb_1}^{\reg}(\psiminus)
+\sum_{i=0,1}\norm{\Lxi\psiminusHigh{i}}^2_{W_{0}^{\reg}(\Donetwo)}.
\end{align}
For the last term, it is bounded by $C(\tb_2-\tb_1)\sum\limits_{i=0,1}\norm{\Lxi\psiminusHigh{i}}^2_{W_{0}^{\reg}(\Sigmaone)}$
 from the estimate \eqref{eq:Lxienergybound:psiminus} and this is further bounded by both $CD_{-1}(\tb_2-\tb_1)f_{\tb_0}$ and $CD_{-1}(\tb_2-\tb_1)\tb_1^{-4-\gamma}$, hence an application of \cite[Lemma 7.4]{} gives
\begin{align}
f_{\tb}^{\reg-\regl}(\psiminus)\lesssim{}&D_{-1}\tb^{-4-\gamma}.
\end{align}
This estimate combined with \eqref{eq:degellip:phiminus} then yields the estimate \eqref{eq:improveEnerDecay:phiminus}.
\end{proof}

In total, we have  from Propositions \ref{prop:weakdecay:spin+1-1:v3} and \ref{prop:improveEnerDecay:phiminus} that for $i=0,1$,
\begin{align}
\label{eq:improvePTWDecay:phiminus}
\absCDeri{\Lxi^j\psiminusHigh{i}}{\reg-\regl}
\lesssim{}&(D_{-1})^{\half}v^{-1}r^i\tb^{-(2+\gamma)/2-j}\max\{\tb^{-1/2},r^{-1/2}\},
\end{align}
and in particular in the exterior region $\{\rb\geq \tb\}$,
\begin{align}
\absCDeri{\Lxi^j\psiminus}{\reg-\regl}
\lesssim{}&(D_{-1})^{\half}v^{-1}\tb^{-(3+\gamma)/2-j}.
\end{align}
In the interior region $\{\rb\leq \tb\}$, however, this needs to be improved.

The estimate \eqref{eq:thm:2:spinpm1:-1} follows from the following statement.

\begin{prop}
\label{prop:evenimproveEnerDecay:phiminus}
Let $\delta>0$ and $j\in \mathbb{N}$. Then there exist  constants $\veps_0>0$ and $\regl$ such that for all $|a|/M\leq \veps_0$ and any $\tb\geq \tb_0$, we have spacetime integral decay
\begin{align}
\label{eq:evenimproveEnerDecay:phiminus}
\norm{\Lxi^j\Psiminus
}_{W_{-3+2\delta}^{{\reg-\regl}}(\DOC_{\tb,\infty})}^2
\lesssim_{\delta}{}D_{-1}\tb^{-4+2\delta-\gamma-2j},
\end{align}
and pointwise decay
\begin{align}
\label{eq:evenimprovePTWDecay:phiminus}
\absCDeri{\Lxi^j\psiminus}{\reg-\regl}
\lesssim_{\delta} {}&(D_{-1})^{\half}v^{-1}\tb^{-\frac{3-2\delta+\gamma+2j}{2}}\max\{r^{-\delta}
,\tb^{-\delta}\}.
\end{align}
\end{prop}

\begin{proof}
Rewrite the wave equation of $\Phiminus{0}$ as
\begin{align}
\label{eq:waveofPsiminus:intermsofPhi1}
-Y\Phiminus{1}+2\edthR'\edthR\Psiminus
={}&
 \Big(
 -2a\Lxi\Leta
-a^2\sin^2\theta\Lxi^2
-2ia\cos\theta\Lxi
+\tfrac{2ar}{\R}\Leta\Big) \Psiminus.
\end{align}
With equation \eqref{def:vectorsKandVR}, this becomes
\begin{align}
\label{eq:waveofPsiminus:intermsofPhi1:v11}
\prb\Phiminus{1}+\mathbf{S}\Psiminus
={}&G(\Psiminus),
\end{align}
with $\mathbf{S}=2\edthR'\edthR$ and
\begin{align}
G(\Psiminus)={}&
\Big(-2a\Lxi\Leta
-a^2\sin^2\theta\Lxi^2
-2ia\cos\theta\Lxi
\Big) \Psiminus
+(2\mu^{-1}-H)\Lxi\Phiminus{1}\notag\\
&
+\tfrac{2ar}{\R}\Leta\Psiminus.
\end{align}
Taking a square on both sides, multiplying by $f^2(\R)^{-3/2}$ with real smooth function $f=f(r)$, and integrating over $\Sigmatb$, one finds
\begin{align}
\hspace{4ex}&\hspace{-4ex}
\int_{\Sigmatb}f^2 (\R)^{-\frac{3}{2}}\abs{G(\Psiminus)}^2\notag\\
={}&\int_{\Sigmatb}f^2 (\R)^{-\frac{3}{2}}
\left(\abs{\prb\Phiminus{1}}^2
+\abs{\mathbf{S}\Psiminus}^2
+2\Re\big(\prb\overline{\Phiminus{1}}\mathbf{S}\Psiminus\big)\right).
\end{align}
Note that we compress the volume element $\di^3\mu$ in the integral over $\Sigmatb$ throughout this proof for simplicity.
Define $\mathbf{S}^{\half}=\sqrt{2}\edthR$.
The third term on the RHS is
\begin{align}
&\int_{\Sigmatb}
\Re\Big(\prb(2f^2(\R)^{-\frac{3}{2}}\overline{\Phiminus{1}}\mathbf{S}
\Psiminus)
-\prb(2f^2(\R)^{-\frac{3}{2}}\mu^{-1})\overline{\Phiminus{1}}
\mathbf{S}\Phiminus{0}\notag\\
&\qquad
-2f^2(\R)^{-\frac{3}{2}}\mu^{-1}\overline{\Phiminus{1}}\prb\mathbf{S}
\Phiminus{0}\Big),
\end{align}
and this gives
\begin{align}
\label{eq:betterdecay:generalexp:phiminus:1}
\hspace{4ex}&\hspace{-4ex}
\int_{\Sigmatb}f^2 (\R)^{-\frac{3}{2}}\abs{G(\Psiminus)}^2\notag\\
={}&\int_{\Sigmatb}f^2 (\R)^{-\frac{3}{2}}
\left(\abs{\prb\Phiminus{1}}^2
+\abs{\mathbf{S}\Psiminus}^2
+2(\R)^{-1}\mu^{-1}
\abs{\mathbf{S}^{\half}\Phiminus{1}}^2\right)\notag\\
&-\prb(2f^2(\R)^{-\frac{3}{2}}\mu^{-1})\Re\big(\overline{\Phiminus{1}}
\mathbf{S}\Phiminus{0}\big)\notag\\
&
+\prb\Big(2f^2(\R)^{-\frac{3}{2}}\Re\big(\overline{\Phiminus{1}}\mathbf{S}
\Psiminus\big)\Big)\notag\\
&-2f^2(\R)^{-\frac{3}{2}}\mu^{-1}\Re\Big(\overline{\mathbf{S}^{\half}\Phiminus{1}}
(\mu H\Lxi\mathbf{S}^{\half}\Psiminus
+2a(\R)^{-1}\Leta\mathbf{S}^{\half}\Psiminus)\Big).
\end{align}
The second line on the RHS is
\begin{align}
\int_{\Sigmatb}\Big(
6f^2\mu^{-1}r(\R)^{-\frac{5}{2}}
-2\prb(\mu^{-1}f^2)(\R)^{-\frac{3}{2}}\Big)
\Re\Big(\mathbf{S}\Phiminus{0}\overline{\Phiminus{1}}\Big),
\end{align}
and the first term on the RHS equals
\begin{align}
\label{eq:betterdecay:generalexp:phiminus:1:1}
\hspace{4ex}&\hspace{-4ex}
\int_{\Sigmatb} f^2(\R)^{-\frac{3}{2}}
\Big|(\R)\prb\Big((\R)^{-1}\Phiminus{1}\Big)
+2r(\R)^{-1}\Phiminus{1}\Big|^2\notag\\
={}&\int_{\Sigmatb} f^2(\R)^{\frac{1}{2}}
\abs{\prb((\R)^{-1}\Phiminus{1})}^2
+4f^2r^2(\R)^{-\frac{7}{2}}\abs{\Phiminus{1}}^2\notag\\
&\quad
-\prb(2f^2 r (\R)^{-\half})(\R)^{-2}\abs{\Phiminus{1}}^2
+\prb\Big(2f^2r (\R)^{-\frac{5}{2}}\abs{\Phiminus{1}}^2\Big).
\end{align}
Therefore, the first two lines on the RHS of \eqref{eq:betterdecay:generalexp:phiminus:1} are equal to
\begin{align}
\label{eq:betterdecay:generalexp:phiminus:2}
&\int_{\Sigmatb}\tfrac{f^2}{ (\R)^{\frac{3}{2}}}
\Big(\abs{\mathbf{S}\Psiminus}^2
+\tfrac{2}{\mu(\R)}
\abs{\mathbf{S}^{\half}\Phiminus{1}}^2
+\tfrac{4r^2}{(\R)^{2}}\abs{\Phiminus{1}}^2\Big)
\notag\\
&\qquad +\Big(6f^2\mu^{-1}r(\R)^{-\frac{5}{2}}
-2\prb(\mu^{-1}f^2)(\R)^{-\frac{3}{2}}
\Big)
\Re\Big(\mathbf{S}\Phiminus{0}\overline{\Phiminus{1}}\Big)\notag\\
&\qquad
-\prb\Big(\tfrac{2rf^2}{\sqrt{\R}}\Big)
\tfrac{\abs{\Phiminus{1}}^2}{(\R)^2}
+f^2(\R)^{\frac{1}{2}}
\Big|\prb\Big(\tfrac{\Phiminus{1}}{\R}\Big)\Big|^2
+\prb\Big(\tfrac{2f^2r \abs{\Phiminus{1}}^2}{(\R)^{5/2}}\Big).
\end{align}
In addition,
\begin{align}
\label{eq:betterdecay:generalexp:phiminus:2:1}
\hspace{4ex}&\hspace{-4ex}
\int_{\Sigmatb}\tfrac{f^2}{\mu (\R)^{1/2}}\abs{\prb\mathbf{S}^{\half}\Phiminus{0}}^2\notag\\
={}&\int_{\Sigmatb}\tfrac{f^2}{\mu (\R)^{1/2}}
\abs{\prb(r\mu)r^{-1}\mathbf{S}^{\half}\Psiminus
+r\mu\prb(r^{-1}\mathbf{S}^{\half}\Psiminus)}^2\notag\\
={}&\int_{\Sigmatb}\tfrac{f^2}{\mu (\R)^{1/2}}(\prb(r\mu))^2r^{-2}\abs{\mathbf{S}^{\half}\Psiminus}^2
+\tfrac{f^2}{\mu (\R)^{1/2}}r^2\mu^2\abs{\prb(r^{-1}\mathbf{S}^{\half}\Psiminus)}^2\notag\\
&-\prb\Big(\tfrac{f^2 r \prb(r\mu)}{(\R)^{1/2}}\Big)r^{-2}
\abs{\mathbf{S}^{\half}\Psiminus}^2
+\prb\Big(\tfrac{rf^2\prb(r\mu)}{(\R)^{1/2}}
r^{-2}\abs{\mathbf{S}^{\half}\Psiminus}^2\Big).
\end{align}
Together with the expression
\begin{align}
\Phiminus{1}=(\R)\prb\Phiminus{0}
+(\R)(H\Lxi+2a\mu^{-1}(\R)^{-1}\Leta)\Phiminus{0},
\end{align}
 we obtain
\begin{align}
\label{eq:betterdecay:generalexp:phiminus:3}
\hspace{4ex}&\hspace{-4ex}
\int_{\Sigmatb}\tfrac{h^2}{\mu (\R)^{5/2}}\abs{\mathbf{S}^{\half}\Phiminus{1}}^2\notag\\
={}&\int_{\Sigmatb}\tfrac{h^2}{\mu (\R)^{1/2}}\abs{\prb\mathbf{S}^{\half}\Phiminus{0}}^2
-\tfrac{h^2}{\mu (\R)^{1/2}}\abs{(\mu H\Lxi +2a(\R)^{-1}\Leta)\mathbf{S}^{\half}\Psiminus}^2\notag\\
&\qquad
+\tfrac{2h^2}{\mu (\R)^{3/2}}\Re\big(\overline{\mathbf{S}^{\half}\Phiminus{1} } (\mu H\Lxi +2a(\R)^{-1}\Leta)\mathbf{S}^{\half}\Psiminus\big).
\end{align}
Choose $f^2 = \mu (\R)^{2\delta}$ and note that
\begin{align}
\label{eq:betterdecay:generalexp:phiminus:2:2}
\hspace{4ex}&\hspace{-4ex}
\int_{\Sigmatb} (\R)^{-\half+2\delta}\abs{\prb\mathbf{S}^{\half}\Phiminus{0}}^2\notag\\
={}&\int_{\Sigmatb}(\R)^{-\half+2\delta}
\abs{\prb(r^{1-2\delta}\mu)r^{-1+2\delta}\mathbf{S}^{\half}\Psiminus
+r^{1-2\delta}\mu\prb(r^{-1+2\delta}\mathbf{S}^{\half}\Psiminus)}^2\notag\\
={}&\int_{\Sigmatb} (\R)^{-\half+2\delta}(\prb(r^{1-2\delta}\mu))^2r^{-2+4\delta}\abs{\mathbf{S}^{\half}\Psiminus}^2
\notag\\
&\qquad+\tfrac{\mu^2}{ (\R)^{1/2-2\delta}}r^{2-4\delta}\abs{\prb(r^{-1+2\delta}\mathbf{S}^{\half}\Psiminus)}^2\notag\\
&\qquad-\prb\Big(\tfrac{\mu r^{1-2\delta} \prb(\mu r^{1-2\delta})}{(\R)^{1/2-2\delta}}\Big)r^{-2+4\delta}
\abs{\mathbf{S}^{\half}\Psiminus}^2
+\prb\Big(\tfrac{\mu r^{-1+2\delta}\prb(r^{1-2\delta}\mu)}{(\R)^{1/2-2\delta}}
\abs{\mathbf{S}^{\half}\Psiminus}^2\Big).
\end{align}
Since the third line on the RHS of \eqref{eq:betterdecay:generalexp:phiminus:1} is identically zero,
equations \eqref{eq:betterdecay:generalexp:phiminus:2}, \eqref{eq:betterdecay:generalexp:phiminus:3} and \eqref{eq:betterdecay:generalexp:phiminus:2:2}
yield the RHS of \eqref{eq:betterdecay:generalexp:phiminus:1} equals the sum of
\begin{align}
\label{eq:betterdecay:generalexp:phiminus:4:1}
&\int_{\Sigmatb}\mu(\R)^{\frac{1}{2}+2\delta}
\Big|\prb\Big(\tfrac{\Phiminus{1}}{\R}\Big)\Big|^2
+\tfrac{1}{ (\R)^{1/2-2\delta}}r^{2-4\delta}\mu^2
\abs{\prb(r^{-1+2\delta}\mathbf{S}^{\half}\Psiminus)}^2\notag\\
&\qquad +\prb\Big(\tfrac{2\mu r \abs{\Phiminus{1}}^2}{(\R)^{5/2-2\delta}}\Big)
+\prb\Big(\tfrac{r^{-1+2\delta}\mu\prb(r^{1-2\delta}\mu)}
{(\R)^{1/2-2\delta}}
\abs{\mathbf{S}^{\half}\Psiminus}^2\Big)\notag\\
&\qquad
-\tfrac{1}{ (\R)^{1/2-2\delta}}\abs{( \mu H\Lxi +2a(\R)^{-1}\Leta)\mathbf{S}^{\half}\Psiminus}^2
\end{align}
and
\begin{align}
\label{eq:betterdecay:generalexp:phiminus:5:1}
\hspace{4ex}&\hspace{-4ex}\int_{\Sigmatb}\tfrac{\mu}{ (\R)^{\frac{3}{2}-2\delta}}
\abs{\mathbf{S}\Psiminus
+\tfrac{2r}{\R}\Phiminus{1}}^2
-\prb\Big(\tfrac{2r\mu}{(\R)^{1/2-2\delta}}\Big)
\tfrac{\abs{\Phiminus{1}}^2}{(\R)^2}\notag\\
\hspace{4ex}&\hspace{-4ex}\qquad
+\tfrac{1}{ (\R)^{5/2-2\delta}}\abs{\mathbf{S}^{\half}\Phiminus{1}}^2
+\tfrac{(2-8\delta)\mu r}{(\R)^{5/2-2\delta}}
\Re\big(\overline{\mathbf{S}\Psiminus}\Phiminus{1}\big)\notag\\
\hspace{4ex}&\hspace{-4ex}\qquad +\Big((\R)^{-\half+2\delta}(\prb(r^{1-2\delta}\mu))^2
-\prb\Big(\tfrac{\mu r^{1-2\delta} \prb(\mu r^{1-2\delta})}{(\R)^{1/2-2\delta}}\Big)\Big)r^{-2+4\delta}\abs{\mathbf{S}^{\half}\Psiminus}^2\notag\\
\geq{}&\int_{\Sigmatb}\tfrac{\mu}{ (\R)^{\frac{3}{2}-2\delta}}
\abs{\mathbf{S}\Psiminus
+\tfrac{2r}{\R}\Phiminus{1}}^2
+\tfrac{(\R)^2-(1-4\delta)r^2\Delta}{ (\R)^{9/2-2\delta}}(\abs{\mathbf{S}^{\half}\Phiminus{1}}^2-
2\abs{\Phiminus{1}}^2)\notag\\
&\qquad
+\tfrac{\mu(1-4\delta)}{ (\R)^{7/2-2\delta}}\abs{r\mathbf{S}^{\half}\Phiminus{1}
-(\R)\mathbf{S}^{\half}\Psiminus}^2
\notag\\
&\qquad
+\tfrac{8a^2Mr}{(\R)^{9/2-2\delta}}\abs{\Phiminus{1}}^2
+4\delta \mu (1-\mu +\mu\delta)(\R)^{-\frac{3}{2}+2\delta}\abs{\mathbf{S}^{\half}\Psiminus}^2\notag\\
&\qquad
-\tfrac{Ca^2}{(\R)^{5/2-2\delta}}\abs{\mathbf{S}^{\half}\Psiminus}^2.
\end{align}
An application of Hardy's inequality \eqref{eq:Hardy:trivial} implies that the second last term in \eqref{eq:betterdecay:generalexp:phiminus:5:1} and the second term in \eqref{eq:betterdecay:generalexp:phiminus:4:1} together bound over $c\delta\int_{\Sigmatb}r^{-3+4\delta}\abs{\mathbf{S}^{\half}\Psiminus}^2$, hence for sufficiently small $\abs{a}/M\leq \veps_0$, the last term of \eqref{eq:betterdecay:generalexp:phiminus:5:1} is absorbed. The total derivative terms vanish identically in view of the asymptotics, therefore, the RHS of \eqref{eq:betterdecay:generalexp:phiminus:1} is larger than
\begin{align}
& \int_{\Sigmatb}c\delta\Big(\mu r^{-5+4\delta}
(r^2\abs{\mathbf{S}\Psiminus}^2 +
r^2\abs{\prb\Phiminus{1}}^2
+\abs{\mathbf{S}^{\half}\Phiminus{1}}^2)
+\mu^2 r^{-1+4\delta}\abs{\prb\mathbf{S}^{\half}\Psiminus}^2
\notag\\
& \qquad\quad
+r^{-5+4\delta}(\abs{\mathbf{S}^{\half}\Phiminus{1}}^2-
2\abs{\Phiminus{1}}^2
+r^2\abs{\mathbf{S}^{\half}\Psiminus}^2)\Big)
\notag\\
&\qquad
+\tfrac{8a^2Mr}{(\R)^{9/2-2\delta}}\abs{\Phiminus{1}}^2
-\tfrac{1}{ (\R)^{1/2-2\delta}}\abs{( \mu H\Lxi +2a(\R)^{-1}\Leta)\mathbf{S}^{\half}\Psiminus}^2.
\end{align}
By taking $\abs{a}/M\leq \veps_0$ sufficiently small, the $\tfrac{2ar}{\R}\Leta\Psiminus$ part in the integral of  $G(\Psiminus)$ of the LHS of \eqref{eq:betterdecay:generalexp:phiminus:1} can be absorbed.
In conclusion, we arrive at an estimate
\begin{align}
\label{eq:integralonSigmatb:generalexp:phiminus:1:v1}
\hspace{4ex}&\hspace{-4ex}
\int_{\Sigmatb}\Big(\mu (\R)^{-\frac{3}{2}+2\delta}\abs{G(\Psiminus)-\tfrac{2ar}{\R}\Leta\Psiminus}^2
+\tfrac{1}{ (\R)^{1/2-2\delta}}\abs{( \mu H\Lxi +2a(\R)^{-1}\Leta)\mathbf{S}^{\half}\Psiminus}^2\Big)\di^3\mu\notag\\
={}&\mathbf{E}_{\Sigmatb}(\Psiminus)\notag\\
\gtrsim_{\delta} {}&\int_{\Sigmatb}
r^{-3+4\delta}\Big(\mu^2 \abs{r\prb\mathbf{S}^{\half}\Psiminus}^2
+\mu (\abs{\edthR\mathbf{S}^{\half}\Psiminus}^2
+\abs{\edthR\mathbf{S}^{\half}\Psiminus}^2)
\Big).
\end{align}
Moreover, the LHS is further bounded by
\begin{align}
\label{eq:integralonSigmatb:generalexp:phiminus:1:v1:11}
\text{LHS of } \eqref{eq:integralonSigmatb:generalexp:phiminus:1:v1}\lesssim
a^2\norm{\Leta\mathbf{S}^{\half}
\Psiminus}^2_{W_{-5+4\delta}^{0}(\Sigmatb)}
+\norm{\Lxi\Psiminus}^2_{W_{-3+4\delta}^{1}(\Sigmatb)}
+\norm{\Lxi\Phiminus{1}}^2_{W_{-3+4\delta}^{0}(\Sigmatb)}
\end{align}
From the estimate \eqref{eq:improveEnerDecay:phiminus}, the integral of $\abs{\Leta\mathbf{S}^\half \Psiminus}^2$ has $\tb^{-4-\gamma}$ decay while all the other terms on the LHS of \eqref{eq:integralonSigmatb:generalexp:phiminus:1:v1} have decay $\tb^{-5+4\delta-\gamma}$. To get around this problem, we shall use a red-shift estimate of $\psiminus$ to bound it. The governing equation \eqref{eq:TME} of $\psiminus$ is
\begin{align}
\label{eq:psiminus:forredshift}
\hspace{4ex}&\hspace{-4ex}
\Big(\partial_r(\Delta\partial_r)
-\tfrac{((\R)\Lxi+a\Leta)^2}{\Delta}
+a^2\sin^2\theta\Lxi^2
+2a\Lxi\Leta
+2\edthR'\edthR
\Big)\psiminus\notag\\
={}&
2(r-M)Y\psiminus - (2r+2ia\cos\theta)\Lxi\psiminus-2\psiminus.
\end{align}
Commuting with $\edthR$ gives
\begin{align}
\hspace{4ex}&\hspace{-4ex}
\Big(\partial_r(\Delta\partial_r)
-\tfrac{((\R)\Lxi+a\Leta)^2}{\Delta}
+a^2\sin^2\theta\Lxi^2
+2a\Lxi\Leta
+2\edthR'\edthR
\Big)\edthR\psiminus\notag\\
={}&
(2(r-M)Y - (2r+2ia\cos\theta)\Lxi
+(4s+2))\edthR\psiminus\notag\\
&
+2^{-\half}
(\partial_{\theta}(a^2\sin^2\theta)\Lxi^2 -2ia\partial_\theta (\cos\theta)\Lxi)\psiminus.
\end{align}
Note that $\edthR\psiminus$ has spin weight $0$, and the operator on the LHS is the same as the expansion of $\Sigma\Box$ when acting on spin $0$ scalars. One can thus follow \cite[Lemma 3.3]{Ma2017Maxwell}, apply the same multiplier, and conclude that there exist constants $\veps_0>0$, $r_+ <r_0<r_1$ and a timelike vector field $N=f_1(r)Y+f_2(r)\Lxi$ with $f_1, f_2\to 1$ as $r\to r_+$  such that for all $\abs{a}/M\leq \veps_0$, any $\reg'\in \mathbb{Z}^+$ and $\tb_2>\tb_1\geq \tb_0$,
\begin{align}
\label{eq:integralonSigmatb:generalexp:phiminus:1:v2}
\hspace{4ex}&\hspace{-4ex}
\norm{\edthR\psiminus
}^2_{W_{0}^{\reg'}(\Sigmatwo^{\leq r_0})}
+\norm{\edthR\psiminus
}^2_{W_{0}^{\reg'}(\Donetwo^{\leq r_0})}\notag\\
\lesssim{}&\norm{\edthR\psiminus
}^2_{W_{0}^{\reg'}(\Sigmatwo^{\leq r_1})}
+\norm{\edthR\psiminus
}^2_{W_{0}^{\reg'}(\Donetwo^{r_0, r_1})}
+\norm{\Lxi\psiminus
}^2_{W_{0}^{\reg'}(\Donetwo^{\leq r_1})}
+\norm{\edthR\psiminus
}^2_{W_{0}^{\reg'-1}(\Donetwo^{\leq r_1})}.
\end{align}
We add to this inequality with $\reg'=1$ the estimate \eqref{eq:integralonSigmatb:generalexp:phiminus:1:v1} with both $\tb=\tb_1$ and $\tb=\tb_2$  and a large multiple of the integral of the estimate \eqref{eq:integralonSigmatb:generalexp:phiminus:1:v1} for $\tb\in [\tb_1,\tb_2]$ and utilize the bound \eqref{eq:integralonSigmatb:generalexp:phiminus:1:v1:11}, then for sufficiently small $\abs{a}/M\leq \veps_0$, the terms $\norm{\edthR\psiminus
}^2_{W_{0}^{1}(\Donetwo^{r_0, r_1})}$ and $\norm{\edthR\psiminus
}^2_{W_{0}^{0}(\Donetwo^{\leq r_1})}$ in \eqref{eq:integralonSigmatb:generalexp:phiminus:1:v2} and the integral of
$a^2\norm{\Leta\mathbf{S}^{\half}
\Psiminus}^2_{W_{-5+4\delta}^{0}(\Sigmatb)}$ are thus all absorbed, leading to
\begin{align}
\label{eq:integralonSigmatb:generalexp:phiminus:1:v2:2}
h_{\tb_2}+
\int_{\tb_1}^{\tb_2}h_{\tb}\di\tb
\lesssim_{\delta} {}&h_{\tb_1}
+\Big(\norm{\Lxi\Psiminus}^2_{W_{-3+4\delta}^{1}(\Sigmaone)}
+\norm{\Lxi\Phiminus{1}}^2_{W_{-3+4\delta}^{0}(\Sigmaone)}\Big)\notag\\
&+\Big(\norm{\Lxi\Psiminus}^2_{W_{-3+4\delta}^{1}(\Donetwo)}
+\norm{\Lxi\Phiminus{1}}^2_{W_{-3+4\delta}^{0}(\Donetwo)}\Big),
\end{align}
where
\begin{align}
\label{eq:htb}
h_{\tb}
={}&\mathbf{E}_{\Sigmatb}(\Psiminus)
+\norm{\Lxi\mathbf{S}^{\half}\Psiminus}^2_{W_{-3+4\delta}^{0}(\Sigmatb)}
+\norm{\edthR\Psiminus
}^2_{W_{-3+4\delta}^{1}(\Sigmatb)}\notag\\
\gtrsim{}&\norm{\edthR\Psiminus
}^2_{W_{-3+4\delta}^{1}(\Sigmatb)}
+\norm{\mu \prb\Phiminus{1}
}^2_{W_{-5+4\delta}^{1}(\Sigmatb)}
+\norm{\mu \mathbb{S}^{\half}\Phiminus{1}
}^2_{W_{-5+4\delta}^{1}(\Sigmatb)}
.
\end{align}
From equation \eqref{eq:waveofPsiminus:intermsofPhi1:v11}, one can freely add $\norm{\prb\Phiminus{1}
}^2_{W_{-5+4\delta}^{1}(\Sigmatb)}+\norm{ \Lxi\Phiminus{1}
}^2_{W_{-3+4\delta}^{1}(\Sigmatb)}$ to the expression of $h_{\tb}$ such that the estimate \eqref{eq:integralonSigmatb:generalexp:phiminus:1:v2:2} remains valid.
Moreover, one can extract out the following red-shift estimate for $\psiminus$
\begin{align}
\label{eq:integralonSigmatb:generalexp:phiminus:1:v2:22}
\hspace{4ex}&\hspace{-4ex}
\norm{\psiminus
}^2_{W_{0}^{\reg'}(\Sigmatwo^{\leq r_0})}
+\norm{\psiminus
}^2_{W_{0}^{\reg'}(\Donetwo^{\leq r_0})}\notag\\
\lesssim{}&\norm{\psiminus
}^2_{W_{0}^{\reg'}(\Sigmatwo^{\leq r_1})}
+\norm{\psiminus
}^2_{W_{0}^{\reg'}(\Donetwo^{r_0, r_1})}
+\norm{\Leta\psiminus
}^2_{W_{0}^{\reg'-1}(\Donetwo^{\leq r_1})}\notag\\
&+\norm{\Lxi\psiminus
}^2_{W_{0}^{\reg'-1}(\Donetwo^{\leq r_1})}
+\norm{\psiminus
}^2_{W_{0}^{\reg'-1}(\Donetwo^{\leq r_1})}
+\norm{\Phiminus{1}
}^2_{W_{0}^{\reg'-1}(\Donetwo^{\leq r_1})}.
\end{align}
By adding this estimate with $\reg'=2$ to a large multiple of the inequality \eqref{eq:integralonSigmatb:generalexp:phiminus:1:v2:2} with $h_{\tb}$ being the sum of the RHS of \eqref{eq:htb} and $\norm{\prb\Phiminus{1}
}^2_{W_{-5+4\delta}^{1}(\Sigmatb)}+\norm{ \Lxi\Phiminus{1}
}^2_{W_{-3+4\delta}^{1}(\Sigmatb)}$, the last five terms in \eqref{eq:integralonSigmatb:generalexp:phiminus:1:v2:22} are absorbed, hence we eventually obtain the estimate \eqref{eq:integralonSigmatb:generalexp:phiminus:1:v2:2} with instead $h_{\tb}=\norm{\Psiminus
}^2_{W_{-3+4\delta}^{2}(\Sigmatb)}$. It is manifest that one can also commute with $\Lxi$, yielding
\begin{align}
\label{eq:integralonSigmatb:generalexp:phiminus:1:v2:33}
\hspace{4ex}&\hspace{-4ex}
\norm{\Lxi^j\Psiminus
}^2_{W_{-3+4\delta}^{2}(\Sigmatwo)}
+\norm{\Lxi^j\Psiminus
}^2_{W_{-3+4\delta}^{2}(\Donetwo)}\notag\\
\lesssim_{\delta}{}& \norm{\Lxi^j\Psiminus
}^2_{W_{-3+4\delta}^{2}(\Sigmaone)}
+\norm{\Lxi^{j+1}\Psiminus}^2_{W_{-1+4\delta}^{1}(\Sigmaone)}
+\norm{\Lxi^{j+1}\Psiminus}^2_{W_{-1+4\delta}^{1}(\Donetwo)}
\end{align}
The last two terms on the RHS have $\tb_1^{-4-\gamma-2j+4\delta}$ decay from the energy decay estimate \eqref{eq:improveEnerDecay:phiminus}, hence a simple application of the mean-value principle gives $\tb_1^{-1}$ decay for the LHS. One can iterate this procedure and obtain eventually
\begin{align}
\label{eq:integralonSigmatb:generalexp:phiminus:1:v2:44}
\norm{\Lxi^j\Psiminus
}^2_{W_{-3+4\delta}^{2}(\Sigmatb)}
+\norm{\Lxi^j\Psiminus
}^2_{W_{-3+4\delta}^{2}(\DOC_{\tb,\infty})}
\lesssim_{\delta}{}D_{-1}\tb^{-4-\gamma-2j+4\delta} .
\end{align}

We next show that the estimate \eqref{eq:integralonSigmatb:generalexp:phiminus:1:v2:44} holds with the regularity parameter $2$ replaced by general $\reg\geq 2$. Commuting equation \eqref{eq:waveofPsiminus:intermsofPhi1:v11} with any $\mathbf{X}\in \CDeri$ gives
\begin{align}
\label{eq:waveofPsiminus:intermsofPhi1:v11:11}
\prb\Big((\R)\VR\Big(\tfrac{\Delta}{\R}\mathbf{X}\Lxi^j\Psiminus\Big)\Big)
+\mathbf{S}(\mathbf{X}\Lxi^j\Psiminus)
={}&G(\mathbf{X}\Lxi^j\Psiminus)+\sum_{\abs{\mathbf{a}}\leq 2}f_{\mathbf{a}}\CDeri^{\mathbf{a}}\Lxi^j\Psiminus
\end{align}
with $f_{\mathbf{a}}$ being $O(1)$ functions. By running the argument above again, one achieves the same estimate as \eqref{eq:integralonSigmatb:generalexp:phiminus:1:v2:33} but replacing $\Psiminus$ by $\mathbf{X}\Psiminus$ and adding both $W_{-3+4\delta}^{0}(\Sigmaone)$ and $W_{-3+4\delta}^{0}(\Donetwo)$ norm squares of $\sum\limits_{\abs{\mathbf{a}}\leq 2}f_{\mathbf{a}}\CDeri^{\mathbf{a}}\Lxi^j\Psiminus$ to the RHS. There norm squares are bounded by $\tb_1^{-4-\gamma-2j+4\delta}$ from inequality \eqref{eq:integralonSigmatb:generalexp:phiminus:1:v2:44}. A same argument as proving \eqref{eq:integralonSigmatb:generalexp:phiminus:1:v2:44} then implies the estimate \eqref{eq:integralonSigmatb:generalexp:phiminus:1:v2:44} holds with regularity parameter $2$ replaced by $3$. One can commute further \eqref{eq:waveofPsiminus:intermsofPhi1:v11:11} with $\CDeri$ and conclude the estimate for general $\reg\geq 2$, and this proves \eqref{eq:evenimproveEnerDecay:phiminus}. The pointwise decay estimate \eqref{eq:evenimprovePTWDecay:phiminus} are manifest from \eqref{eq:improvePTWDecay:phiminus}, the just proved estimate \eqref{eq:evenimproveEnerDecay:phiminus} and inequality \eqref{eq:Sobolev:3}.
\end{proof}

\subsection{Improved decay of spin $+1$ component}
\label{sect:vdecay:psiplus}

The estimate \eqref{eq:thm:2:spinpm1:+1} follows from the following proposition.
\begin{prop}\label{prop:ImproDecaySpin+1}
There exists an $\veps_0>0$ and a $\regl>0$ such that for $\abs{a}/M\leq \veps_0$ and any $\delta\in (0,1/2)$,
\begin{align}
\label{eq:imprDec:psiplus:1}
\absCDeri{\Lxi^j(r^{-2}\psiplus)}{\reg-\regl}\lesssim_{\delta}{}
(D_{+1}+D_{-1})^{\half}v^{-3}
\tb^{-(\gamma-1)/2-j+\delta}\max\{r^{-\delta}
,\tb^{-\delta}\}.
\end{align}
\end{prop}

\begin{proof}
To obtain improved decay for spin $+1$ component, we utilize one of the Teukolsky-Starobinsky identities (TSI) \cite{aksteiner2019new}:
\begin{align}
\label{eq:TSI:simpleform}
(\sqrt{2}\edthR'-ia\sin\theta\Lxi)^2(\Delta^{-1}\psiplus)
={}\VR^2(\Delta\psiminus).
\end{align}
This is the physical space version of the ones first appearing in \cite[Equations (3.9) and (3.10)]{TeuPress1974III}. Multiplying on both sides by $\frac{\Delta}{\R}$ and expanding in terms of $V$, one obtains
\begin{align}
\label{eq:TSI:original}
\hspace{4ex}&\hspace{-4ex}
(2(\edthR')^2-2\sqrt{2}ia\sin\theta \Lxi\edthR'
-a^2\sin^2\theta\Lxi^2)((\R)^{-1}{\psiplus})
\notag\\
={}&V^2((\R)\psiminus)+V\bigg(\frac{2M(r^2-a^2)}{\R}\psiminus\bigg).
\end{align}
For convenience, define $\psiplusc=(\R)^{-1}{\psiplus}$ and define
\begin{subequations}
\begin{align}
L_{+1}(\psiplusc)={}&
(2\sqrt{2}ia\sin\theta \edthR'
+a^2\sin^2\theta\Lxi)\psiplusc,\\
L_{-1}(\psiminus)={}&V^2((\R)\psiminus)+V\bigg(\frac{2M(r^2-a^2)}{\R}\psiminus\bigg).
\end{align}
\end{subequations}
The TSI \eqref{eq:TSI:original} then simplifies to
\begin{align}
2(\edthR')^2\psiplusc={}\Lxi L_{+1}(\psiplusc)+L_{-1}(\psiminus).
\end{align}
Applying $j$ $(j \in \mathbb{Z}^+)$ times $\Lxi$ gives
\begin{align}
\label{eq:TSI:simplify}
2(\edthR')^2\Lxi^j\psiplusc={}\Lxi^{j+1} L_{+1}(\psiplusc)+\Lxi^j L_{-1}(\psiminus).
\end{align}
In view of the pointwise estimates for $\Lxi^{j+1}\psiplus$ and $(rV)^i\Lxi^j\psiminus$ $(i=0,1,2)$ and the fact that $(\edthR')^2$ has a non-trivial kernel when acting on spin $+1$ scalars, we achieve by integrating over spheres that for any $j\in \mathbb{Z}^+$,
\begin{align}
\int_{\mathbb{S}^2}\absSDeri{\Lxi^j\psiplusc}{2}\di^2 \mu\lesssim{}v^{-1}\tb^{-(\gamma+1)/2-j+\delta}\max\{r^{-\delta}
,\tb^{-\delta}\}.
\end{align}
Substituting this back to \eqref{eq:TSI:simplify} enables us to improve the decay estimates to
\begin{align}
\int_{\mathbb{S}^2}\absSDeri{\Lxi^j\psiplusc}{2}\di^2 \mu\lesssim{}v^{-1}
\tb^{-(\gamma+3)/2-j+\delta}\max\{r^{-\delta}
,\tb^{-\delta}\}.
\end{align}
Using a Sobolev imbedding on spheres and combined with the decay estimate \eqref{eq:weakdecay:spin+1:v3}, we conclude
\begin{align}
\label{eq:imprDec:psiplus:1:1}
\abs{\Lxi^j(r^{-2}\psiplus)}\lesssim{}v^{-3}
\tb^{-(\gamma-1)/2-j+\delta}\max\{r^{-\delta}
,\tb^{-\delta}\}.
\end{align}
One can iterate by commuting with $\CDeri$ to close the proof.
\end{proof}

\subsection{Decay for the middle component}
\label{sect:vtdecay:psimiddle}

The following lemma is a standard statement and we take it from \cite[Proposition 2]{larsblue15Maxwellkerr}. This is to decompose a Maxwell field into a stationary part and a radiative part.
\begin{lemma}
\label{lem:decomp:Maxwellfield}
For a Maxwell field in a subextremal Kerr spacetime, it can be decomposed into
\begin{align}
\mathbf{F}=\mathbf{F}_{\text{sta}}+
       \mathbf{F}_{\text{rad}},
\end{align}
where the part $\mathbf{F}_{\text{rad}}$ is the radiative or non-charged part of the Maxwell field and the other part $\mathbf{F}_{\text{sta}}$ is the charged stationary Coulomb part,
such that
\begin{enumerate}
\item $\mathbf{F}_{\text{sta}}$ and $\mathbf{F}_{\text{rad}}$ are both solutions to the Maxwell equations, and the N-P components of them satisfy
\begin{subequations}
\label{eq:FradandFsta:property}
\begin{align}
\NPRplus(\mathbf{F}_{\text{sta}})
={}&\NPRminus(\mathbf{F}_{\text{sta}})=0,&
\NPRzero(\mathbf{F}_{\text{sta}})
={}&\kappa^{-2}
\left(q_{\mathbf{E}} + iq_{\mathbf{B}}\right),\\
\NPRplus(\mathbf{F}_{\text{rad}})={}&\NPRplus(\mathbf{F}),&
\NPRminus(\mathbf{F}_{\text{rad}})={}&\NPRminus(\mathbf{F});
\end{align}
\end{subequations}
\item The charges $q_{\mathbf{E}}$ and $q_{\mathbf{B}}$ are constants at all spheres $\mathbb{S}^2(\tb,\rb)$ for any $\tb\in \mathbb{R}$ and $\rb\geq r_+$, and can be calculated from the initial data;
   \item For any closed $2$-surface, say $\mathbb{S}^2$, $\int_{\mathbb{S}^2} \mathbf{F}_{\text{rad}}=  \int_{\mathbb{S}^2} \leftidx{^{\star}}{\mathbf{F}}_{\text{rad}}=0$;
   \item $\Lxi\mathbf{F}_{\text{sta}}=0$.
 \end{enumerate}
 \end{lemma}

The point \ref{point2:thm:2} of Theorem \ref{thm:2} then follows from the proven estimates above and the proposition below, and the last statement about the asymptotics in the exterior region $\{\rb\geq \tb\}$ in Theorem \ref{thm:2} follows from Proposition \ref{prop:weakdecay:spin+1-1:v3} and the  following proposition.
\begin{prop}
\label{prop:estiofMiddlecomp:generalassupofextremecomps}
Let $\gamma\geq 1$, $0\leq \delta_0\leq 1/2$, $j\in\mathbb{N}$  and $D_{\reg,j} \geq 0$ with $\reg$ suitably large (depending on $j$). Assume
\begin{subequations}
\label{eq:psiplusminus:summary:1}
\begin{align}
&\absCDeri{\Lxi^j \psiminus}{\reg-\reg'}
\leq{}D_{\reg,j} v^{-1}\tb^{-(3+\gamma)/2+\delta_0-j}\max\{r^{-\delta_0}
,\tb^{-\delta_0}\},\\
&\absCDeri{\Lxi^j (rV(r\psiminus))}{\reg-\reg'}\leq{}
D_{\reg,j} v^{-1}\tb^{-(1+\gamma)/2+\delta_0-j}\max\{r^{-\delta_0}
,\tb^{-\delta_0}\},\\
\label{eq:psiplusminus:summary:1:c}
&\absCDeri{\Lxi^j (r^{-2}\psiplus)}{\reg-\reg'}\leq{}
D_{\reg,j} v^{-3}\tb^{-(\gamma-1)/2+\delta_0-j}\max\{r^{-\delta_0}
,\tb^{-\delta_0}\}.
\end{align}
\end{subequations}
Then there exists a stationary function $\NPRzero^{\text{sta}}$ defined at every point $(\tb,\rb)$ by
$\NPRzero^{\text{sta}}
={\kappa^{-2}}(q_{\mathbf{E}} + iq_{\mathbf{B}})$ and a constant $\regl>0$
such that
\begin{align}
\absCDeri{\Lxi^j(\NPRzero
-\NPRzero^{\text{sta}}))}
{\reg-\regl}
\leq {}&D_{\reg,j}
v^{-2}\tb^{-\frac{\gamma+1}{2}+\delta_0-j}\max\{r^{-\delta_0}, \tb^{-\delta_0}\}.
\end{align}
\end{prop}

\begin{proof}
The charges of the radiative part $\mathbf{F}_{\text{rad}}$ vanish, hence
\begin{align}
0={}&q_{\mathbf{E}}(\mathbf{F}_{\text{rad}}) + iq_{\mathbf{B}}(\mathbf{F}_{\text{rad}})={}\int_{\mathbb{S}^2(\tb,\rb)}
(\leftidx{^{\star}}{\mathbf{F}_{\text{rad}}}+i\mathbf{F}_{\text{rad}})(\partial_{\theta},\partial_{\pb})
\di\theta\di\pb.
\end{align}
By expanding out the above expression in terms of the N-P components, one finds
\begin{align}
0={}&\int_{\mathbb{S}^2(\tb,\rb)}
\bigg(\frac{\R}{\kappa^2}\psizero^{\text{rad}}
+\frac{ia\sin\theta }{\kappa}(\sqrt{2}\Sigma\NPRplus^{\text{rad}} -\Delta\psiminus^{\text{rad}})\bigg)\di^2\mu.
\end{align}
Here, the superscript $\text{rad}$ means the quantities are the corresponding components of the radiative part $\mathbf{F}_{\text{rad}}$.
Let $\hatpsizero=\frac{\R}{\kappa^2}\psizero^{\text{rad}}$, and decompose it into a sum of the spherically symmetric part $\hatpsizeroS$ and the non-spherically symmetric part $\hatpsizeroN$, i.e., $\hatpsizero=\hatpsizeroS+\hatpsizeroN$.
Therefore,
\begin{align}\label{eq:Psizero:middle:rad:0}
\hatpsizeroS={}&\int_{\mathbb{S}^2(\tb,\rb)}
\frac{-ia\sin\theta }{4\pi\kappa}(\sqrt{2}\Sigma\NPRplus^{\text{rad}} -\Delta\psiminus^{\text{rad}})\di^2\mu.
\end{align}
For the non-spherically symmetric part $\hatpsizeroN$, one can apply the standard elliptic estimate
\begin{align}\label{eq:Psizero:middle:rad:1}
\int_{\mathbb{S}^2(\tb,\rb)}2\abs{\hatpsizeroN}^2\di^2\mu\leq {}&
\int_{\mathbb{S}^2(\tb,\rb)}2\abs{\edthR\hatpsizeroN}^2
\di^2\mu
=\int_{\mathbb{S}^2(\tb,\rb)}2
\abs{\edthR\hatpsizero}^2
\di^2\mu.
\end{align}
Plugging in the relation between $\hatpsizero$ and $\psizero^{\text{rad}}$, one finds
\begin{align}
\text{RHS of } \eqref{eq:Psizero:middle:rad:1}={}\int_{\mathbb{S}^2(\tb,\rb)}\bigg(2\frac{\R}{\kappa^2}
\edthR\psizero^{\text{rad}} + \partial_{\theta} (\kappa^{-2}(\R))\psizero^{\text{rad}}\bigg)^2 \di^2\mu.
\end{align}
Furthermore, we have
\begin{align}\label{eq:Psizero:middle:rad:2}
\hspace{4ex}&\hspace{-4ex}
\int_{\mathbb{S}^2(\tb,\rb)}2\abs{\hatpsizeroN}^2-
\abs{\partial_{\theta} (\kappa^{-2}(\R))\psizero^{\text{rad}}}^2\di^2\mu\notag\\
={}&\int_{\mathbb{S}^2(\tb,\rb)}\bigg(
\frac{2(\R)^2}{\Sigma^2}\abs{\psizero^{\text{rad}}}^2
-\abs{\partial_{\theta} (\kappa^{-2}(\R))\psizero^{\text{rad}}}^2
-2\abs{\hatpsizeroS}^2\bigg)\di^2\mu\notag\\
={}&\int_{\mathbb{S}^2(\tb,\rb)}\bigg(
\frac{2(\R)^2}{\Sigma^3}
({\Sigma-2a^2\sin^2\theta})\abs{\psizero^{\text{rad}}}^2
-2\abs{\hatpsizeroS}^2\bigg)\di^2\mu
\end{align}
For $r\geq 2M\geq 2\abs{a}$, $\Sigma-2a^2\sin^2\theta\geq \half r^2$, hence it follows from \eqref{eq:Psizero:middle:rad:1}--\eqref{eq:Psizero:middle:rad:2} that there exist two positive constants $C_1$ and  $C_2$ such that
\begin{align}\label{eq:Psizero:middle:rad:3}
\int_{\mathbb{S}^2(\tb,\rb)}\abs{\psizero^{\text{rad}}}^2\di^2\mu\leq {}&
C_1\int_{\mathbb{S}^2(\tb,\rb)}\abs{\edthR\psizero^{\text{rad}}}^2
\di^2\mu
-C_2\int_{\mathbb{S}^2(\tb,\rb)}\abs{\hatpsizeroS}^2 \di^2\mu
.
\end{align}
Given the estimates \eqref{eq:psiplusminus:summary:1}, and in view of the fact that $Y$ is a linear combination of $r^{-1} rV$, $\Lxi$ and $r^{-2}\Leta$ with $O(1)$ coefficients when away from horizon and the decay estimate \eqref{eq:psiplusminus:summary:1:c}, it holds true that
\begin{align}
\absCDeri{\Lxi^j Y(r^{-2}\psiplus)}{\reg-\regl}\leq{}
D_{\reg,j}v^{-3}\tb^{-(\gamma-1)/2+\delta_0-j}\max\{r^{-1-\delta_0}
,\tb^{-1-\delta_0}\},
\end{align}
We thus obtain from the estimate \eqref{eq:Psizero:middle:rad:0} that
\begin{align}\label{eq:Psizero:middle:rad:4}
r^{-2}\absCDeri{\Lxi^j\hatpsizeroS}{\reg-\regl}\lesssim{}&
D_{\reg,j}v^{-2}\tb^{-(\gamma+1)/2+\delta_0-j}\max\{r^{-\delta_0}
,\tb^{-\delta_0}\},
\end{align}
from the first two subequations of the Maxwell system \eqref{eq:TSIsSpin1Kerr} that
\begin{align}
r^{-2}(\absCDeri{\Lxi^j(\bar{\kappa}m^{\mu}\partial_{\mu}\psizero^{\text{rad}})}{\reg-\regl}
+\absCDeri{\Lxi^j(\kappa \overline{m}^{\mu}\partial_{\mu}\psizero^{\text{rad}})}{\reg-\regl})\lesssim{}& D_{\reg,j}v^{-2}\tb^{-(\gamma+1)/2+\delta_0-j}\max\{r^{-\delta_0}
,\tb^{-\delta_0}\},
\end{align}
and from the last two subequations of the Maxwell system \eqref{eq:TSIsSpin1Kerr} that
\begin{align}
\label{eq:YVpsizerorad}
r^{-2}(\absCDeri{Y\Lxi^j\psizero^{\text{rad}}}{\reg-\regl}
+\absCDeri{V\Lxi^j\psizero^{\text{rad}}}{\reg-\regl})
\lesssim{}& D_{\reg,j}v^{-2}\tb^{-(\gamma+1)/2+\delta_0}\max\{r^{-\delta_0}
,\tb^{-\delta_0}\}.
\end{align}
For $K$ defined as in \eqref{def:vectorsKandVR}, there is a relation
\begin{align}
\partial_{\theta} +\frac{\R-a^2\sin^3\theta}{\R}\frac{i\partial_{\phi}}{\sin\theta}
={}&
\bar{\kappa}m^{\mu}\partial_{\mu}-ia\sin\theta K\notag\\
={}&\bar{\kappa}m^{\mu}\partial_{\mu}-\half ia\sin\theta\bigg(\frac{\Delta}{\R}Y+V\bigg)
\end{align}
and its complex conjugate. Therefore, the above decay estimates together imply
\begin{align}\label{eq:Psizero:middle:rad:5}
\absCDeri{\edthR (r^{-2}\Lxi^j\psizero^{\text{rad}})}{\reg-\regl}\lesssim{}&
D_{\reg,j} v^{-2}\tb^{-(\gamma+1)/2+\delta_0-j}\max\{r^{-\delta_0}
,\tb^{-\delta_0}\},
\end{align}
and this together with the estimates \eqref{eq:Psizero:middle:rad:3} and \eqref{eq:Psizero:middle:rad:4} yields that for $r\geq 2M$,
\begin{align}
r^{-4}\int_{\mathbb{S}^2(\tb,\rb)}\absCDeri{\Lxi^j\psizero^{\text{rad}}}{\reg-\regl}^2\di^2\mu
\lesssim {}&D_{\reg,j}v^{-4}\tb^{-(\gamma+1)+2\delta_0-2j}\max\{r^{-2\delta_0}
,\tb^{-2\delta_0}\}.
\end{align}

One can commute with $\edthR'$ and $\edthR$ to obtain the same type of estimates
\begin{align}
\sum_{\abs{\mathbf{a}}\leq m}
r^{-4}\int_{\mathbb{S}^2(\tb,\rb)}
\absCDeri{\SDeri^{\mathbf{a}}\Lxi^j\psizero^{\text{rad}}}{\reg-\regl}^2\di^2\mu
\lesssim {}&D_{\reg,j}^2v^{-4}\tb^{-(\gamma+1)+2\delta_0-2j}\max\{r^{-2\delta_0}
,\tb^{-2\delta_0}\}.
\end{align}
The vector field $\Lxi$ commutes with the Maxwell system \eqref{eq:TSIsSpin1Kerr}, and the decay estimates \eqref{eq:psiplusminus:summary:1} are valid if replacing $\psiplus$ and $\psiminus$ by $\Lxi^j\psiplus$ and $\Lxi^j\psiminus$ and adding extra $\tb^{-j}$ decay on the RHS of each equation; hence for any $m,j\in\mathbb{N}$ and $r\geq 2M$,
\begin{align}
\sum_{\abs{\mathbf{a}}\leq m}
r^{-4}\int_{\mathbb{S}^2(\tb,\rb)}
\absCDeri{\Lxi^j\SDeri^{\mathbf{a}}\psizero^{\text{rad}}}{\reg-\regl}^2\di^2\mu\lesssim {}&D_{\reg,j}^2v^{-4}\tb^{-(\gamma+1)-2j+2\delta_0}\max\{r^{-2\delta_0}
,\tb^{-2\delta_0}\}.
\end{align}
A standard Sobolev imbedding on sphere then gives for any $j\in\mathbb{N}$ and $r\geq 2M$
\begin{align}
\label{eq:Psizero:middle:rad:7}
\absCDeri{\Lxi^j(r^{-2}\psizero^{\text{rad}})}{\reg-\regl}
\lesssim {}&D_{\reg,j}v^{-2}\tb^{-(\gamma+1)/2-j+\delta_0-j}\max\{r^{-\delta_0}
,\tb^{-\delta_0}\}.
\end{align}
For $r=\rb'\in [r_+,2M)$, we integrate from $(\tb, 2M,\theta,\pb)$
\begin{align}
\absCDeri{\Lxi^j\psizero^{\text{rad}}}{\reg-\regl}
\lesssim {}&D_{\reg,j}v^{-2}\tb^{-(\gamma+1)/2-j+\delta_0}\max\{r^{-\delta_0}
,\tb^{-\delta_0}\} +
\int_{\rb'}^{2M}
\absCDeri{\Lxi^j\prb\psizero^{\text{rad}}}{\reg-\regl}\di\rb\notag\\
\lesssim {}&D_{\reg,j}v^{-2}\tb^{-(\gamma+1)/2-j+\delta_0}\max\{r^{-\delta_0}
,\tb^{-\delta_0}\},
\end{align}
where in the last step we have used an analogous estimate as \eqref{eq:YVpsizerorad}:
\begin{align}
\label{eq:YVpsizeroradHigh}
r^{-2}(\absCDeri{Y\Lxi^j\psizero^{\text{rad}}}{\reg-\regl}
+\absCDeri{V\Lxi^j\psizero^{\text{rad}}}{\reg-\regl})
\lesssim{}& D_{\reg,j} v^{-2}\tb^{-(\gamma+1)/2-j+\delta_0}\max\{r^{-\delta_0}
,\tb^{-\delta_0}\}.
\end{align}
In summary, the estimate \eqref{eq:Psizero:middle:rad:7} holds for any $j\in\mathbb{N}$ and $r\geq r_+$.
From Lemma \ref{lem:decomp:Maxwellfield}, we are then led to
\begin{align}
\absCDeri{\Lxi^j(r^{-2}(\psizero-(q_{\mathbf{E}} + iq_{\mathbf{B}})))}{\reg-\regl}
\leq {}&D_{\reg,j}v^{-2}\tb^{-(\gamma+1)/2-j+\delta_0}\max\{r^{-\delta_0}
,\tb^{-\delta_0}\},
\end{align}
and this closes the proof.
\end{proof}


\section{Almost Price's law for Maxwell field on Schwarzschild}
\label{sect:Maxwell:Schw:lgeneral}

In this section, we prove Theorem \ref{thm:Schw} in which almost Price's law for Maxwell field on a Schwarzschild background is achieved.
Following the discussions in Section \ref{sect:decompIntoModes}, we decompose spin $\pm 1$ components into modes $\Psiminus=\sum\limits_{\ell_0=1}^{\infty}\Psiminus^{\ell=\ell_0} $ and $\Psiplus=\sum\limits_{\ell_0=1}^{\infty}\Psiplus^{\ell=\ell_0} $, where $\Psiminus^{\ell=\ell_0}$ and $\Psiplus^{\ell=\ell_0}$ are supported on $\ell =\ell_0$. Since the Schwarzschild spacetime is spherically symmetric, one finds each mode satisfies the Maxwell equations.

On Schwarzschild, the BEAM estimates, basic energy $2$-decay condition and the claimed pointwise asymptotics in Theorems \ref{thm:BEAM}, \ref{thm:1}, \ref{thm:2} and \ref{thm:3} are all satisfied. To obtain better decay estimates for a fixed mode of Maxwell field on a Schwarzschild background, the $r^p$ estimate in Proposition \ref{prop:wave:rp:highmodes} needs to be utilized such that the basic energy $\gamma$-decay condition holds for larger $\gamma$ if the field is supported in larger $\ell$ modes. Besides, given basic energy $\gamma$-decay condition for larger $\gamma$, the estimate in Proposition \ref{prop:evenimproveEnerDecay:phiminus:Schw} below is useful in removing the $\delta$ loss in time decay in Theorem \ref{thm:2}.

\begin{prop}
\label{prop:evenimproveEnerDecay:phiminus:Schw}
In a Schwarzschild spacetime, the spin $-1$ component  satisfies the following estimate
\begin{align}\label{eq:W-3energynorm:spin-1:Schw}
\hspace{4ex}&\hspace{-4ex}
\int_{\Sigmatb}\big(
\mu r^{-5}\abs{2\Phiminus{1}+r \mathbf{S}\Psiminus}^2
+\mu r^{-5}\abs{\mathbf{S}^{\half}\Phiminus{1}
-\mathbf{S}^{\half}\Psiminus}^2
+\mu r
\abs{\prb(r^{-2}{\Phiminus{1}})}^2\notag\\
\hspace{4ex}&\hspace{-4ex}\qquad
+2M r^{-6}(\abs{\mathbf{S}^{\half}\Phiminus{1}}^2
-2\abs{\Phiminus{1}}^2)
+\mu^2 r\abs{\prb(r^{-1}\mathbf{S}^{\half}\Psiminus)}^2\big)\di^3 \mu\notag\\
={}&\int_{\Sigmatb}\big(\mu r^{-3}\abs{G(\Psiminus)}^2
+r^{-1}\abs{ \mu H\Lxi \mathbf{S}^{\half}\Psiminus}^2\big)\di^3\mu,
\end{align}
where  $G(\Psiminus)=(2\mu^{-1}-H)\Lxi\Phiminus{1}$ and $H= 2\mu^{-1}+\partial_r h(r)$ with $h(r)$ being the function introduced in Section \ref{sect:foliation}.
\end{prop}

\begin{proof}
Taking $a=0$ and $f=\mu$ (i.e., taking $\delta=0$ in \eqref{eq:betterdecay:generalexp:phiminus:4:1} and \eqref{eq:betterdecay:generalexp:phiminus:5:1}),
equation \eqref{eq:betterdecay:generalexp:phiminus:1} becomes
\begin{align}
\label{eq:betterdecay:generalexp:phiminus:4:1:Schw}
\hspace{4ex}&\hspace{-4ex}
\int_{\Sigmatb}\mu r^{-3}\abs{G(\Psiminus)}^2\di^3\mu\notag\\
={}
&\int_{\Sigmatb}
\big(\mu r
\abs{\prb(r^{-2}{\Phiminus{1}})}^2
+\mu^2 r
\abs{\prb(r^{-1}\mathbf{S}^{\half}\Psiminus)}^2
-r^{-1}\abs{ \mu H\Lxi \mathbf{S}^{\half}\Psiminus}^2\notag\\
&\qquad
+\mu r^{-5}
\abs{r\mathbf{S}\Psiminus
+2\Phiminus{1}}^2
-\prb(2\mu)r^{-4}
{\abs{\Phiminus{1}}^2}\notag\\
&\qquad
+r^{-5}\abs{\mathbf{S}^{\half}\Phiminus{1}}^2
+2\mu r^{-4}
\Re\big(\overline{\mathbf{S}\Psiminus}\Phiminus{1}\big)\notag\\
&\qquad +(r^{-1}(\prb(r\mu))^2
-\prb(\mu\prb(\mu r)))r^{-2}
\abs{\mathbf{S}^{\half}\Psiminus}^2\big)\di^3\mu.
\end{align}
Equation \eqref{eq:W-3energynorm:spin-1:Schw} follows after simple calculations.
\end{proof}

\begin{remark}
\label{rem:evenimproveEnerDecay:phiminus:Schw}
If the basic energy $\gamma$-decay condition holds true for the spin $-1$ component, then from Proposition \ref{prop:improveEnerDecay:phiminus}, the RHS, and hence the LHS of \eqref{eq:W-3energynorm:spin-1:Schw}, decays like $\tb^{-5-\gamma}$.
\end{remark}
We discuss about the Newman--Penrose constants in Section \ref{sect:NPconsts} and consider in Sections \ref{sect:spin+1:l=1:Schw}--\ref{sect:Maxwell:Schw:l=l0:minus} a fixed $\ell=\ell_0$ mode of spin $+1$ or $-1$ component of Maxwell field on Schwarzschild, and eventually give a proof of Theorem \ref{thm:Schw} in Section \ref{sect:close:thm:Schw}. Unless otherwise stated, we will simple drop the superscript $\ell=\ell_0$ in the scalars defined by these spin $\pm 1$ components since it is clear from the title of each subsection which mode we are treating.

\subsection{Newman--Penrose constants}
\label{sect:NPconsts}

\begin{definition}
\label{def:tildePhiplusandminusHigh}
For any $i\in \mathbb{N}$, let $f_{i,1}=(i+1)(i+2)$, $f_{i,2}=-2(i+2)$, $g_i=6 \sum\limits_{j=0}^i f_{j,1}= 2i(i+1)(i+2)$, $x_{i+1,i}=\frac{g_{i+1}}{f_{i+1,1} -f_{i,1}}=(i+1)(i+3)$, and $x_{i+1,j}=-\frac{g_{i+1}x_{i,j}}{f_{i+1,1}-f_{j,1}}$ for $0\leq j\leq i-1$. Define
\begin{align}
\PhiplusHigh{i}={}&\curlVR^i\Phiplus, &\Phiminus{i+2}={}&\curlVR^i\Phiminus{2},
\end{align}
and
\begin{subequations}
\begin{align}
\tildePhiplusHigh{0}={}&\PhiplusHigh{0},& \tildePhiplusHigh{i+1}={}&\PhiplusHigh{i+1}
+\sum_{j=0}^ix_{i+1,j}M^{i+1-j}\tildePhiplusHigh{j},\\
\tildePhiminus{2}={}&\Phiminus{2},& \tildePhiminus{i+3}={}&\Phiminus{i+3}
+\sum_{j=0}^ix_{i+1,j}M^{i+1-j}\tildePhiminus{j+2}.
\end{align}
\end{subequations}
\end{definition}

It is convenient to introduce $(u,v,\theta,\phi)$ coordinates, where $u=(t-r^*)/2$ and $v=(t+r^*)/2$. Then $\pu=\mu Y$ and $\pv=V$.

\begin{prop}
Let $i\in \mathbb{N}$.
\begin{enumerate}
  \item The equation of $\Phiplus$ is
\begin{align}
\label{eq:Phi+1:Schw:generall}
-\pu\curlVR \Phiplus +(2\edthR\edthR'+2)\Phiplus-{4(r-3M)r^{-2}}\curlVR\Phiplus -{12M}{r^{-1}}\Phiplus={}&0,
\end{align}
 the equation of $\PhiplusHigh{i}$ is
\begin{align}
\label{eq:Phiplushighi:Schw:generall}
&-\pu \curlVR  \PhiplusHigh{i} +(2\edthR\edthR'+f_{i,1})\PhiplusHigh{i}
+{f_{i,2}(r-3M)r^{-2}}\curlVR\PhiplusHigh{i}
-6f_{i,1}Mr^{-1}\PhiplusHigh{i}
+g_i M\PhiplusHigh{i-1}={}0,
\end{align}
and the equation of $\tildePhiplusHigh{i}$ is
\begin{align}
\label{eq:Phiplushighi:Schw:generall:tildePhiplusi}
&-\pu \curlVR  \tildePhiplusHigh{i} +(2\edthR\edthR'+f_{i,1})\tildePhiplusHigh{i}
+{f_{i,2}(r-3M)r^{-2}}\curlVR\tildePhiplusHigh{i}
-6f_{i,1}Mr^{-1}\tildePhiplusHigh{i}
+\sum_{j=0}^{i}h_{i,j} \PhiplusHigh{j}={}0,
\end{align}
with $h_{i,j}=O(r^{-1})$ for all $j\in \{0,1,\ldots, i\}$.
  \item The equation of $\Phiminus{2}$ is
\begin{align}
\label{eq:Phiminus2:Schw:generall}
-\pu \curlVR\Phiminus{2}+(2\edthR'\edthR+2)\Phiminus{2}
-4(r-3M)r^{-2}
{\curlVR\Phiminus{2}}-12Mr^{-1}\Phiminus{2}
={}0,
\end{align}
and the equation of $\Phiminus{i+2}$  is
\begin{align}
\label{eq:Phiminusi:Schw:generall}
&-\pu \curlVR  \Phiminus{i+2} +(2\edthR'\edthR+f_{i,1})\Phiminus{i+2}
+{f_{i,2}(r-3M)r^{-2}}\curlVR\Phiminus{i+2}
-6f_{i,1}Mr^{-1}\Phiminus{i+2}
+g_i M\Phiminus{i+1}={}0,
\end{align}
and the equation of $\tildePhiminus{i+2}$ is
\begin{align}
\label{eq:tildePhiminusi:Schw:generall:tildePhiminusi}
&-\pu \curlVR  \tildePhiminus{i+2} +(2\edthR\edthR'+f_{i,1})\tildePhiminus{i+2}
+{f_{i,2}(r-3M)r^{-2}}\curlVR\tildePhiminus{i+2}
-6f_{i,1}Mr^{-1}\tildePhiminus{i+2}
+\sum_{j=0}^{i}h_{i,j} \Phiminus{j+2}={}0,
\end{align}
with $h_{i,j}$ being the same as the ones in \eqref{eq:Phiplushighi:Schw:generall:tildePhiplusi} and satisfying $h_{i,j}=O(r^{-1})$ for all $j\in \{0,1,\ldots, i\}$.
\end{enumerate}

\end{prop}

\begin{proof}
In Schwarzschild, the wave equation \eqref{eq:Phi+1} simplifies to
\begin{align}
\label{eq:Phi+1:Schw:generall:v1}
-r^2 YV \Phiplus +(2\edthR\edthR'+2)\Phiplus-{2\mu^{-1}(r-3M)}V\Phiplus -{12M}{r^{-1}}\Phiplus={}&0,
\end{align}
which is exactly
\eqref{eq:Phi+1:Schw:generall}. By a simple commutation relation
\begin{align}
[\curlVR, -r^2YV]\varphi={}&-\curlVR\Big(\tfrac{2(r-3M)}{r^{2}} \curlVR\varphi\Big)
=-\tfrac{2(r-3M)}{r^2}\curlVR^2\varphi
+(2-12Mr^{-1})\curlVR\varphi,
\end{align}
one can inductively obtain
\begin{align}
\label{eq:Phiplushighi:Schw:generall:v1}
&-r^2Y V  \PhiplusHigh{i} +(2\edthR\edthR'+f_{i,1})\PhiplusHigh{i}
+{(f_{i,2}+2)(r-3M)r^{-2}}\curlVR\PhiplusHigh{i}
-6f_{i,1}Mr^{-1}\PhiplusHigh{i}
+g_i M\PhiplusHigh{i-1}={}0,
\end{align}
which proves \eqref{eq:Phiplushighi:Schw:generall} for general $i\in \mathbb{N}$.

Equation \eqref{eq:Phiplushighi:Schw:generall:tildePhiplusi} is proved by induction. Assume it holds for $\tildePhiplusHigh{j}$ for all $1\leq j\leq i$, we show it holds also for $\tildePhiplusHigh{i+1}$. By adding $x_{i+1, j}M^{i+1-j}$ multiple of equation \eqref{eq:Phiplushighi:Schw:generall:tildePhiplusi} for $\tildePhiplusHigh{j}$ for all $j=0,1,\ldots,i$ to equation \eqref{eq:Phiplushighi:Schw:generall} of $\PhiplusHigh{i+1}$, we arrive at
\begin{align}
\label{eq:tildePhipsHigh:choices}
&-2\pu \curlVR  \tildePhiplusHigh{i+1} +(2\edthR\edthR'+f_{i+1,1})\tildePhiplusHigh{i+1}
+{f_{i+1,2}(r-3M)r^{-2}}\curlVR\tildePhiplusHigh{i+1}
-6f_{i+1,1}Mr^{-1}\tildePhiplusHigh{i+1}\notag\\
&-\sum_{j=0}^i x_{i+1,j} M^{i+1-j}(f_{i+1,1}- f_{j,1})\tildePhiplusHigh{j}
+g_{i+1}M\PhiplusHigh{i}\notag\\
&+6Mr^{-1}\sum_{j=0}^i x_{i+1,j} M^{i+1-j}(f_{i+1,1}- f_{j,1})\tildePhiplusHigh{j}\notag\\
&
-(r-3M) r^{-2}\sum_{j=0}^i (f_{i+1,2}-f_{j,2})x_{i+1,j}M^{i+1-j}  \curlVR\tildePhiplusHigh{j}
+\sum_{j=0}^i x_{i+1,j} M^{i+1-j}\sum_{j'=0}^{j}h_{j,j'} \PhiplusHigh{j'}={}0.
\end{align}
By replacing $\PhiplusHigh{i}=\PhiplusHigh{i}
-\sum\limits_{j=0}^{i-1}x_{i,j}M^{i-j}\tildePhiplusHigh{j}$ into the last term of the second line, one finds the second line equals $\sum\limits_{j=1}^i e_{i+1,j}M^{i+1-j}\tildePhiplusHigh{j}$, with
\begin{subequations}
\begin{align}
e_{i+1,i}={}&-x_{i+1,i}(f_{i+1,1}-f_{i,1}) +g_{i+1},\\
e_{i+1,j}={}&-x_{i+1,j}(f_{i+1,1}-f_{j,1}) -g_{i+1}x_{i,j}, \quad \text{for} \quad 0\leq j\leq i-1.
\end{align}
\end{subequations}
All of these $\{e_{i+1,j}\}_{j=0,1,\ldots,i}$ are equal to zero from the choices of $\{x_{i+1,j}\}_{j=0,1,\ldots, i}$ in Definition \ref{def:tildePhiplusandminusHigh}, which means that the entire second line of \eqref{eq:tildePhipsHigh:choices} vanishes.
One can rewrite $\curlVR\tildePhiplusHigh{j}$ using Definition \ref{def:tildePhiplusandminusHigh} as a weighted sum of $\{\tildePhiplusHigh{j'}\}|_{j'=0,1,\ldots, j+1}$ with all coefficients being $O(1)$. Thus, by denoting all the terms in the last two lines on the LHS of \eqref{eq:tildePhipsHigh:choices} as $\sum\limits_{j=0}^{i+1}h_{i+1,j} \PhiplusHigh{j}$, one finds  $h_{i+1,j}=O(r^{-1})$ for all $j\in \{0,1,\ldots, i+1\}$, which thus proves equation \eqref{eq:Phiplushighi:Schw:generall:tildePhiplusi} for $\tildePhiplusHigh{i+1}$.

For spin $-1$ component, equation \eqref{eq:Phiminus2:Schw:generall} comes directly from \eqref{eq:Phi-12}. Since equations \eqref{eq:Phiminus2:Schw:generall} and \eqref{eq:Phi+1:Schw:generall} have the same form,
equations \eqref{eq:Phiminusi:Schw:generall} and \eqref{eq:tildePhiminusi:Schw:generall:tildePhiminusi} follow in the same fashion.
\end{proof}

\begin{prop}
\label{prop:nullinfBeha:PhiplusiandPhiminusi} Let $i,\reg\in \mathbb{N}$.  Let $\regl>0$ be suitably large.
\begin{enumerate}
  \item[$(i)$] If $\norm{\Psiplus}^2_{W_{-2}^{\reg+\regl+2i+2}(\Sigmazero)}<\infty$
      and $\lim\limits_{r\to\infty}\sum\limits_{j=0}^i\absCDeri{
      \PhiplusHigh{j}}{\reg}\vert_{\Sigmazero}
      <\infty$, then there is a $u$-dependent constant $C_i(u)\in (0,\infty)$ such that for any $\tb\geq \tb_0$, $\lim\limits_{r\to\infty}\sum\limits_{j=0}^i
      \absCDeri{\PhiplusHigh{j}}{\reg}\vert_{\Sigmatb}<\infty$. The same statement holds if one replaces all $\PhiplusHigh{j}$ by $\tildePhiplusHigh{j}$;
  \item[$(ii)$]  If $\norm{\Psiplus}^2_{W_{-2}^{\reg+\regl+2i+4}(\Sigmazero)}<\infty$
      and $\lim\limits_{r\to\infty}\Big(\sum\limits_{j=0}^i\absCDeri{
      \edthR\edthR'\PhiplusHigh{j}}{\reg}\vert_{\Sigmazero}
      +r^{-\alpha}\absCDeri{\PhiplusHigh{i+1}}{\reg}\vert_{\Sigmazero}\Big)
      <\infty$ for some $\alpha\in [0,2]$, then there is a $u$-dependent constant $C_{i+1,\alpha}(u)\in (0,\infty)$ such that for any $\tb\geq \tb_0$, $\lim\limits_{r\to\infty}
      \Big(\sum\limits_{j=0}^i
      \absCDeri{\edthR\edthR'\PhiplusHigh{j}}{\reg}\vert_{\Sigmatb}
      +r^{-\alpha}\absCDeri{\PhiplusHigh{i+1}}{\reg}\vert_{\Sigmatb}\Big)<\infty$.
      The same statement holds if one replaces all $\PhiplusHigh{j}$ by $\tildePhiplusHigh{j}$ for all $j\in \{0,\ldots, i+1\}$;
  \item[$(iii)$] If $\sum\limits_{j=0}^2\norm{(r^2 V)^j\Psiminus}^2_{W_{-2}^{\reg+\regl+2i+2}(\Sigmazero)}<\infty$ and $\lim\limits_{r\to\infty}\sum\limits_{j=2}^{i+2}\absCDeri{
      \Phiminus{j}}{\reg}\vert_{\Sigmazero}
      <\infty$, then there is a $u$-dependent constant $C_i(u)\in (0,\infty)$ such that for any $\tb\geq \tb_0$, $\lim\limits_{r\to\infty}\sum\limits_{j=2}^{i+2}
      \absCDeri{\Phiminus{j}}{\reg}\vert_{\Sigmatb}<\infty$.
      The same statement holds if one replaces all $\Phiminus{j}$ by $\tildePhiminus{j}$;
  \item[$(iv)$] If $\sum\limits_{j=0}^2\norm{(r^2 V)^j\Psiminus}^2_{W_{-2}^{\reg+\regl+2i+2}(\Sigmazero)}<\infty$ and $\lim\limits_{r\to\infty}\Big(\sum\limits_{j=2}^{i+3}
      \absCDeri{
      \edthR'\edthR\Phiminus{j}}{\reg}\vert_{\Sigmazero}
      +r^{-\alpha}\absCDeri{
      \Phiminus{i+3}}{\reg}\vert_{\Sigmazero}\Big)
      <\infty$  for some $\alpha\in [0,2]$, then there is a $u$-dependent constant $C_{i+1,\alpha}(u)\in (0,\infty)$ such that for any $\tb\geq \tb_0$, $\lim\limits_{r\to\infty}\Big(\sum\limits_{j=2}^{i+2}
      \absCDeri{\edthR'\edthR\Phiminus{j}}{\reg}\vert_{\Sigmatb}
      +r^{-\alpha}\absCDeri{\Phiminus{i+3}}{\reg}\vert_{\Sigmatb}\Big)<\infty$.
      The same statement holds if one replaces all $\Phiminus{j}$ by $\tildePhiminus{j}$ for all $j\in \{2,\ldots,i+3\}$.
\end{enumerate}
\end{prop}

\begin{proof}
The proof is similar to the one of \cite[Propositions 3.4 and 3.5]{angelopoulos2018vector} and we omit it.
\end{proof}

\begin{definition}
\label{def:NPCs}
Let $i\in \mathbb{Z}^+$. Define the $i$-th N--P constants of spin $+1$ and $-1$ components to be $\NPCP{i}(\theta,\phi)=\lim\limits_{\rb\to\infty}\curlVR\tildePhiplusHigh{i-1}$ and $\NPCN{i}(\theta,\phi)=\lim\limits_{\rb\to\infty}\curlVR
\tildePhiminus{i+1}$, respectively.
\end{definition}

\begin{remark}
As is shown in Proposition \ref{prop:NPCsindepentonu} below, these N--P constants are independent of $\tb$ under very general conditions, hence they are only dependent on $\theta$ and $\phi$.
\end{remark}

\begin{lemma}
On Schwarzschild, it holds true that
$\NPCN{i}=2(\edthR')^2\NPCP{i}$ for $i\in \mathbb{Z}^+$. In particular, if $\NPCN{i}$ vanishes, then $\NPCP{i}$ vanishes, and vice versa.
\end{lemma}

\begin{proof}
Equation \eqref{eq:TSI:simpleform}, one of the Teukolsky--Starobinsky identities, reduces to a simple form
\begin{align}
2(\edthR')^2\Phiplus
={}\Phiminus{2}.
\end{align}
The conclusion then follows manifestly from the fact that the operator $(\edthR')^2$ has no non-trivial kernel when acting on spin weight $+1$ scalar.
\end{proof}

\begin{prop}
\label{prop:NPCsindepentonu}
\begin{enumerate}
  \item Let the spin $\pm 1$ components of Maxwell field be supported on $\ell=1$ mode. Let $\regl>0$ be suitably large.
      \begin{itemize}
      \item Assume $\norm{\Psiplus}^2_{W_{-2}^{\regl}(\Sigmazero)}<\infty$ and $\lim\limits_{r\to\infty}\absCDeri{
      \PhiplusHigh{0}}{1}\vert_{\Sigmazero}
      <\infty$, then the first N--P constant $\NPCP{1}$ is independent of $\tb$;
      \item Assume $\sum\limits_{j=0}^2\norm{(r^2 V)^j\Psiminus}^2_{W_{-2}^{\regl}(\Sigmazero)}<\infty$ and $\lim\limits_{r\to\infty}\absCDeri{
      \Phiminus{2}}{1}\vert_{\Sigmazero}
      <\infty$, then the first N--P constant $\NPCN{1}$ is independent of $\tb$.
      \end{itemize}
  \item Let  the spin $\pm 1$ components of Maxwell field be supported on $\ell=\ell_0$ $(\ell_0\geq 2)$ mode. Let $\regl(\ell_0)>0$ be suitably large.
      \begin{itemize}
      \item Assume $\norm{\Psiplus}^2_{W_{-2}^{\regl(\ell_0)}(\Sigmazero)}<\infty$ and $\lim\limits_{r\to\infty}\sum\limits_{j=0}^{\ell_0-1}
          \absCDeri{
      \PhiplusHigh{j}}{1}\vert_{\Sigmazero}
      <\infty$, then the $\ell_0$-th N--P constant $\NPCP{\ell_0}$ is independent of $\tb$;
       \item Assume $\sum\limits_{j=0}^2\norm{(r^2 V)^j\Psiminus}^2_{W_{-2}^{\regl(\ell_0)}(\Sigmazero)}<\infty$ and $\lim\limits_{r\to\infty}\sum\limits_{j=2}^{\ell_0+1}\absCDeri{
      \Phiminus{j}}{1}\vert_{\Sigmazero}
      <\infty$, then the $\ell_0$-th N--P constant $\NPCN{\ell_0}$ is independent of $\tb$.
      \end{itemize}
\end{enumerate}
\end{prop}

\begin{proof}
If the field is supported on $\ell=1$ mode, then from \eqref{eq:Phi+1:Schw:generall} and  \eqref{eq:Phiminus2:Schw:generall}, $\Psi=\Phiplus$ or $\Psi=\Phiminus{2}$ solves
\begin{align}
-\pu\curlVR \Psi -{4(r-3M)r^{-2}}\curlVR\Psi -{12M}{r^{-1}}\Psi={}&0.
\end{align}
The results in Proposition \ref{prop:nullinfBeha:PhiplusiandPhiminusi} implies $\lim\limits_{r\to \infty}\pu (\curlVR \Psi)\vert_{\Sigmatb}=0$ for any $\tb\geq \tb_0$. The conclusion follows  from the bounded convergence theorem.

Instead, if the field is supported on $\ell=\ell_0$ mode for some $\ell_0\geq 2$, equations \eqref{eq:Phiplushighi:Schw:generall:tildePhiplusi}
 for $\tildePhiplusHigh{\ell_0-1}$ and \eqref{eq:tildePhiminusi:Schw:generall:tildePhiminusi} for $\tildePhiminus{\ell_0+1}$ become
\begin{align}
\label{eq:268}
&-\pu \curlVR \tildePhiplusHigh{\ell_0-1}
-{2(\ell_0+1)(r-3M)r^{-2}}\curlVR\tildePhiplusHigh{\ell_0-1}
+\sum_{j=0}^{\ell_0-1}O(r^{-1})\tildePhiplusHigh{j}={}0,\\
&-\pu \curlVR \tildePhiminus{\ell_0+1}
-{2(\ell_0+1)(r-3M)r^{-2}}\curlVR\tildePhiminus{\ell_0+1}
+\sum_{j=2}^{\ell_0+1}O(r^{-1})\tildePhiminus{j}={}0.
\end{align}
One also obtains $\lim\limits_{r\to \infty}\pu (\curlVR \tildePhiplusHigh{\ell_0-1})\vert_{\Sigmatb}=\lim\limits_{r\to \infty}\pu (\curlVR\tildePhiminus{\ell_0+1})\vert_{\Sigmatb}=0$ from Proposition \ref{prop:nullinfBeha:PhiplusiandPhiminusi}, and by the same way of arguing, the statement follows.
\end{proof}

\begin{prop}
\label{prop:vanishingNPC:betterasymnearscri}
Let the spin $\pm 1$ components of Maxwell field be supported on an $\ell= \ell_0$ mode with $\ell_0\geq 1$. Let $\alpha\in [0,1]$ be arbitrary and let $\reg\in \mathbb{N}$. Assume the $\ell_0$-th N--P constant $\NPCN{\ell_0}$  vanishes.
\begin{itemize}
\item There exists a $\regl(\ell_0)>0$ such that if $\sum\limits_{j=0}^2\norm{(r^2 V)^j\Psiminus}^2_{W_{-2}^{\reg+\regl(\ell_0)}(\Sigmazero)}
    +\lim\limits_{r\to\infty}\sum\limits_{j=2}^{\ell_0+1}\absCDeri{
      \Phiminus{j}}{\reg}\vert_{\Sigmazero}
     +\lim\limits_{r\to\infty}\absCDeri{
      r^{1-\alpha}\curlVR\tildePhiminus{\ell_0+1}}{\reg}\vert_{\Sigmazero}
      <\infty$,
      then there is a  constant $C_{\ell_0}(\tb,\theta,\pb)\in (0,\infty)$ such that for any $\tb\geq \tb_0$, $\lim\limits_{r\to\infty}
\absCDeri{r^{1-\alpha}\curlVR\tildePhiminus{\ell_0+1}}{\reg}
\vert_{\Sigmatb}<C_{\ell_0}(u,\theta,\pb)$.
In particular, if $\alpha>0$, then $\lim\limits_{r\to\infty}\absCDeri{r^{1-\alpha}
\curlVR\tildePhiminus{\ell_0+1}}{\reg}
\vert_{\Sigmatb}$ is independent of $u$ or $\tb$;
\item There exists a $\regl(\ell_0)>0$ such that if $\norm{\Psiplus}^2_{W_{-2}^{\reg+\regl(\ell_0)}(\Sigmazero)}
    +\lim\limits_{r\to\infty}\sum\limits_{j=0}^{\ell_0-1}
          \absCDeri{
      \PhiplusHigh{j}}{\reg}\vert_{\Sigmazero}
      +\lim\limits_{r\to\infty}r^{1-\alpha}
          \absCDeri{
      \curlVR\tildePhiplusHigh{\ell_0-1}}{\reg}\vert_{\Sigmazero}
      <\infty$, then there is a constant $C_{\ell_0}(u,\theta,\pb)$ such that for any $\tb\geq \tb_0$, $\lim\limits_{r\to\infty}
\absCDeri{r^{1-\alpha}
\curlVR\tildePhiplusHigh{\ell_0-1}}{\reg}
\vert_{\Sigmatb}<C_{\ell_0}(u,\theta,\pb)$. In particular, if $\alpha>0$, then $\lim\limits_{r\to\infty}\absCDeri{r^{1-\alpha}
\curlVR\tildePhiplusHigh{\ell_0-1}}{\reg}
\vert_{\Sigmatb}$ is independent of $u$  or $\tb$.
\end{itemize}
\end{prop}

\begin{proof}
We show it only for spin $+1$ component, the proof of spin $-1$ component being the same. Consider first the $\ell=\ell_0$ mode $\Psiplus^{\ell=\ell_0}$. The scalar $\tildePhiplusHigh{\ell_0-1}$  satisfies equation \eqref{eq:268}, and hence performing a rescaling gives
\begin{align}
-\pu (r^{1-\alpha}\curlVR \tildePhiplusHigh{\ell_0-1})
={}O(r^{-\alpha})\curlVR\tildePhiplusHigh{\ell_0-1}
+\sum_{j=0}^{\ell_0-1}O(r^{-\alpha})\tildePhiplusHigh{j}.
\end{align}
By Proposition \ref{prop:nullinfBeha:PhiplusiandPhiminusi} and the assumption of vanishing $\ell_0$-th N--P constant, this yields $\lim\limits_{r\to \infty}\absCDeri{r^{1-\alpha}
\curlVR\tildePhiplusHigh{\ell_0-1}}{\reg}
\vert_{\Sigmatb}=0$, and one obtains
$\lim\limits_{r\to \infty}\pu (r^{1-\alpha}\curlVR \Psi)\vert_{\Sigmatb})=0$ for any $\tb\geq \tb_0$ in the case that $\alpha>0$. The conclusion for $\alpha>0$ follows  from the bounded convergence theorem. For $\alpha=0$, the RHS is bounded by a $u$-dependent constant, hence $\lim\limits_{r\to\infty}
\absCDeri{r
\curlVR\tildePhiplusHigh{\ell_0-1}}{\reg}
\vert_{\Sigmatb}<C(u)$.
\end{proof}


\subsection{$\ell=1$ mode of spin $+1$ component}
\label{sect:spin+1:l=1:Schw}

\begin{prop}
\label{prop:BED:Psiplus:l=1}
Let $j\in \mathbb{N}$ and let $\reg$ suitably large. Let $\Psiplus$ be supported on $\ell=1$ mode. Let $F(\reg,p,\tb,\Psiplus)$ for $p\in [-1,2]$ be defined as in Definition \ref{eq:def:Ffts:Phiplus:-1to2}. Define additionally for any $p\in (2,5)$ that
\begin{align}
\label{def:Ffts:Phiplus:2to3}
F(\reg,p,\tb,\Psiplus)={}&\norm{rV\Psiplus}^2_{W_{p-2}^{\reg-1}(\Sigmatb)}
+\norm{\Psiplus}^2_{W_{-2}^{\reg}(\Sigmatb)}.
\end{align}
Then,
\begin{enumerate}
\item if the first N-P constant $\NPCP{1}$ does not vanish,
there is a constant $\regl(j)$ such that for any small $\delta>0$, any $p\in [0,3-\delta]$ and any $\tb\geq\tb_0$,
\begin{align}
\label{eq:BEDC:Phiplus:l=1:0to3}
F(\reg,p,\tb,\Lxi^j\Psiplus)\lesssim_{\delta,j,\reg} {}&\tb^{-3+\delta-2j+p}F(\reg+\regl(j),3-\delta,\tb_0,\Psiplus),
\end{align}
i.e., the basic energy $\gamma$-decay condition with $\gamma=3-\delta$ holds for spin $+1$ component;
\item
if the first N-P constant $\NPCP{1}$ vanishes, there is a constant $\regl(j)$ such that for any small $\delta>0$, any $p\in [0,5-\delta]$ and any $\tb\geq\tb_0$,
\begin{align}
\label{eq:BEDC:Phiplus:l=1:0to5}
F(\reg,p,\tb,\Lxi^j\Psiplus)\lesssim_{\delta,j,\reg}  {}&\tb^{-5+\delta-2j+p}F(\reg+\regl(j),5-\delta,\tb_0,\Psiplus),
\end{align}
i.e., the basic energy $\gamma$-decay condition with $\gamma=5-\delta$ holds for spin $+1$ component.
\end{enumerate}
\end{prop}

\begin{proof}
The wave equation \eqref{eq:Phi+1:Schw:generall} reads
\begin{align}
\label{eq:Phi+1:Schw:l=1}
-r^2 YV \Phiplus -{2\mu^{-1}(r-3M)}V\Phiplus -{12M}{r^{-1}}\Phiplus={}&0.
\end{align}
By multiplying this equation by $-2r^{p-2}\chi^2 V\overline{\Phiplus}$, taking the real part, and integrating over $\Donetwo$ with a measure $\di^4 \mu$,
\begin{align}
\label{eq:multiplierintegral:generalp:l=1:Schw}
\int_{\Donetwo}\Big(&Y(2r^{p}\chi^2 \abs{V\Phiplus}^2)+24Mr^{p-3} \chi^2\Re(V\overline{\Phiplus} \Phiplus)\notag\\
&((p+r\partial_r  )\chi^2 r^{p-1}+4r^{p}\Delta^{-1} (r-3M))\chi^2 \abs{V\Phiplus}^2\Big)\di^4\mu
={}0.
\end{align}

For $2\leq p< 3$, the second line on the LHS is bounded from below by a bulk integral $\int_{\Donetwo^{R_0-M}}\chi^2 r^{p-1}\abs{V\Phiplus}^2\di^4\mu$, and one applies an integration by parts to the second term on the LHS to obtain both positive fluxes and a positive spacetime integral.  Adding this to the BEAM estimate gives for any $p\in [2,3)$ and $\reg\geq 1$
\begin{align}
\label{eq:rphierachyintermsofF:Psiminus}
F(\reg,p,\tb_2,\Psiplus)+\int_{\tb_1}^{\tb_2}F(\reg,p-1,\tb,\Psiplus)\di\tb
\lesssim_{p,\reg}{} F(\reg+\regl,p,\tb_1,\Psiplus),
\end{align}
where the $\reg\geq 2$ cases follow in the same way as in Proposition \ref{prop:wave:rp} by commuting with the operator set $\mathbb{D}_2$.  This gives an extended $r^p$ hierarchy for $p\in [0,3)$. The estimate \eqref{eq:BEDC:Phiplus:l=1:0to3} follows easily by repeating the discussions in Section \ref{sect:BED2:spin+1}.

For $3\leq p< 4$, one can use instead the Cauchy-Schwarz inequality to bound the second term on the LHS of \eqref{eq:multiplierintegral:generalp:l=1:Schw} by
\begin{align}
\hspace{4ex}&\hspace{-4ex}
\bigg|\int_{\Donetwo}24Mr^{p-3} \chi^2\Re(V\overline\Phiplus \Phiplus)\di^4\mu\bigg|\notag\\
\lesssim{}&\veps\int_{\Donetwo^{R_0-M}}r^{p-1} \chi^2 \abs{V\Phiplus}^2\di^4\mu
+\veps^{-1}\int_{\Donetwo^{R_0-M}}r^{p-5} \chi^2 \abs{\Phiplus}^2\di^4\mu,
\end{align}
and this last term is in turn bounded via the Hardy's inequality \eqref{eq:HardyIneqLHSRHS} by $\veps^{-1}\big(\int_{\Donetwo^{R_0-M}}r^{p-3} \abs{\prb  \Phiplus}^2\di^4\mu+\int_{\Donetwo^{R_0-M, R_0}}r^{p-5} \abs{ \Phiplus}^2\di^4\mu\big)$ since $\lim\limits_{r\to \infty} r^{p-4}\abs{\Phiplus}^2=0$. This thus yields an extended $r^p$ hierarchy for $p\in [0,4)$, i.e., the estimate \eqref{eq:rphierachyintermsofF:Psiminus} holds for $p\in [0,4)$. Therefore,
there is a constant $\regl(j)$ such that for any small $\delta>0$, any $p\in [0,4-\delta]$ and any $\tb\geq2\tb_0$,
\begin{align}
\label{eq:BEDC:Phiplus:l=1:0to4}
F(\reg,p,\tb,\Lxi^j\Psiplus)\lesssim_{\delta,j,\reg} {}&\tb^{-4+\delta-2j+p}F(\reg+\regl(j),4-\delta,\tb/2,\Psiplus).
\end{align}

Furthermore, we shall extend the hierarchy to $p\in [0,5)$.
For $4\leq p\leq 5-\delta$ where $\delta>0$ is small, we estimate the second term on the LHS  of \eqref{eq:multiplierintegral:generalp:l=1:Schw} by
\begin{align}
\hspace{4ex}&\hspace{-4ex}
\bigg|\int_{\Donetwo}24Mr^{p-3} \chi^2_R\Re(V\overline\Phiplus \Phiplus)\di^4\mu\bigg|\notag\\
\lesssim{}&\veps\int_{\Donetwo}r^{p} \tb^{-1-\delta}\chi^2_R \abs{V\Phiplus}^2 \di^4\mu +\veps^{-1}\int_{\Donetwo}r^{p-6}\tb^{1+\delta} \chi^2_R \abs{\Phiplus}^2 \di^4\mu,
\end{align}
The first term on the LHS can be absorbed by choosing $\veps$ small, and the second term is bounded using the estimate \eqref{eq:BEDC:Phiplus:l=1:0to4} by
$\int_{\tb_1}^{\tb_2}\tb^{1+\delta}F(1,p-4,\tb,\Psiplus)\di\tb
\lesssim_{\delta}  \tb_1^{-6+2\delta+p}F(\regl,4-\delta,\tb_0,\Psiplus)$. Therefore, one obtains for any $p\in [4,5-\delta]$ and $\tb_2>\tb_1\geq \tb_0$,
\begin{align}
\hspace{4ex}&\hspace{-4ex}
F(\reg,p,\tb_2,\Psiplus)
+\int_{\tb_1}^{\tb_2}F(\reg,p-1,\tb,\Psiplus)\di\tb
\notag\\
\lesssim_{\delta,\reg} {} &F(\reg+\reg',p,\tb_1,\Psiplus)
+\tb_1^{-6+2\delta+p}F(\reg+\regl,4-\delta,\tb_1,\Psiplus.
\end{align}
This gives for any $p\in [4,5-\delta]$ and $\tb\geq 4\tb_0$,
\begin{align}
F(\reg,p,\tb,\Psiplus)
\lesssim_{\delta,\reg} {}&\tb^{-5+\delta+p}F(\reg+\regl,5-\delta,\tb/2,\Psiplus).
\end{align}
Together with the discussions for $3\leq p<4$, this yields the estimate \eqref{eq:BEDC:Phiplus:l=1:0to5}.
\end{proof}

\subsection{$\ell=\ell_0$ mode of spin $+1$ component with $\ell_0\geq 2$}
\label{sect:Maxwell:Schw:l=l0:plus}
\begin{prop}
\label{prop:BED:Psiplus:l=l0}
Let $\Psiplus$ be supported on a fixed $\ell=\ell_0$ mode with $\ell_0\geq 2$. Let $j\in \mathbb{N}$ and let $\reg\geq \ell_0$.  Define for any $p>2$ that
\begin{align}
\label{def:Ffts:PhiplusHigh:2to5}
F^{(\ell_0-1)}(\reg,p,\tb,\Psiplus)={}&\norm{rV\Psiplus}^2_{W_{0}^{\reg-1}(\Sigmatb)}
+\norm{\Psiplus}^2_{W_{-2}^{\reg}(\Sigmatb)}\notag\\
&
+\sum_{m=1}^{\ell_0-2}\Big(\norm{rV\PhiplusHigh{m}}^2_{W_{0}^{\reg-1-m}(\Sigmatb^{3M})}
+\norm{\PhiplusHigh{m}}^2_{W_{-2}^{\reg-m}(\Sigmatb^{3M})}\Big)\notag\\
&+\Big(\norm{rV\tildePhiplusHigh{\ell_0-1}}^2_{W_{p-2}^{\reg-\ell_0}(\Sigmatb^{3M})}
+\norm{\PhiplusHigh{\ell_0-1}}^2_{W_{-2}^{\reg+1-\ell_0}(\Sigmatb^{3M})}\Big),
\end{align}
Then,
\begin{enumerate}
\item \label{pt1: prop:BED:Psiplus:l=l0} if the $\ell_0$-th N-P constant $\NPCP{\ell_0}$ does not vanish,
there is a constant $\regl(j)$ such that for any small $\delta>0$, any $p\in [0,2]$ and any $\tb\geq\tb_0$,
\begin{align}
\label{eq:BEDC:Phiplus:l=l0:0to3}
F(\reg,p,\tb,\Lxi^j\Psiplus)\lesssim_{\delta,j,\ell_0,\reg} {}&\tb^{-3+\delta-2(\ell_0-1)-2j+p}F^{(\ell_0-1)}(\reg+\regl(j,\ell_0),3-\delta,\tb_0,\Psiplus),
\end{align}
i.e., the basic energy $\gamma$-decay condition holds for spin $+1$ component for $\gamma=3-\delta+2(\ell_0-1)$;
\item
if the $\ell_0$-th N-P constant $\NPCP{\ell_0}$ vanishes, there is a constant $\regl(j)$ such that for any small $\delta>0$, any $p\in [0,2]$ and any $\tb\geq\tb_0$,
\begin{align}
\label{eq:BEDC:Phiplus:l=l0:0to5}
F(\reg,p,\tb,\Lxi^j\Psiplus)\lesssim_{\delta,j,\ell_0,\reg}  {}&\tb^{-5+\delta-2(\ell_0-1)-2j+p}F^{(\ell_0-1)}(\reg+\regl(j,\ell_0),5-\delta,\tb_0,\Psiplus),
\end{align}
i.e., the basic energy $\gamma$-decay condition holds for spin $+1$ component for $\gamma=5-\delta+2(\ell_0-1)$.
\end{enumerate}
\end{prop}

\begin{proof}
The wave equation \eqref{eq:Phi+1:Schw:generall} now takes the form of
\begin{align}
\label{eq:Phi+1:Schw:l0}
-r^2 YV \Phiplus -(\ell_0(\ell_0+1)-2)\Phiplus-2(r-3M)r^{-2}\curlVR\Phiplus -{12M}{r^{-1}}\Phiplus={}&0.
\end{align}
For any $i\in\mathbb{N}$, equation \eqref{eq:Phiplushighi:Schw:generall} simplifies to
\begin{align}
\label{eq:Phiplushighi:Schw:l0}
&-r^2 YV \PhiplusHigh{i} -(\ell_0(\ell_0+1)-(i+1)(i+2))\PhiplusHigh{i}\notag\\
&
-{2(i+1)\mu^{-1}(r-3M)}V\PhiplusHigh{i}
+O(r^{-1})\PhiplusHigh{i}+O(1)\PhiplusHigh{i-1}={}0.
\end{align}
This can be put into the form of \eqref{eq:wave:rp}, and one finds as long as $i\leq \ell_0-2$, the assumptions in Proposition \ref{prop:wave:rp:highmodes} are satisfied with $b_{0,0}(\PhiplusHigh{i})+\ell_0(\ell_0+1)=-(i+1)(i+2)+\ell_0(\ell_0+1)>0$,
while if $i=\ell_0-1$, the assumptions in Proposition \ref{prop:wave:rp:highmodes} are satisfied with $b_{0,0}(\PhiplusHigh{\ell_0-1})+\ell_0(\ell_0+1)=0$. The estimates in Proposition \ref{prop:wave:rp:highmodes} then applies, and for any $p\in (0,2)$, the error term arising from the last two terms on the LHS \eqref{eq:Phiplushighi:Schw:l0} can clearly be bounded by a small portion of the spacetime integral of the other terms on the LHS plus the corresponding estimate of $\PhiplusHigh{i-1}$. For any $1\leq i\leq \ell_0-1$ and $p\in [0,2]$, let
\begin{align}
\label{def:Ffts:PhiplusHigh:0to2}
F^{(i)}(\reg,p,\tb,\Psiplus)={}&\norm{rV\Psiplus}^2_{W_{0}^{\reg-1}(\Sigmatb)}
+\norm{\Psiplus}^2_{W_{-2}^{\reg}(\Sigmatb)}\notag\\
&
+\sum_{m=1}^i\Big(\norm{rV\PhiplusHigh{m}}^2_{W_{p-2}^{\reg-1-m}(\Sigmatb^{3M})}
+\norm{\PhiplusHigh{m}}^2_{W_{-2}^{\reg-m}(\Sigmatb^{3M})}\Big),
\end{align}
and for any $1\leq i\leq \ell_0-1$ and $p\in [-1,0)$, let $F^{(i)}(\reg,p,\tb,\Psiplus)=0$. Then
it holds for any $1\leq i\leq \ell_0-1$, $p\in [0,2)$ and $\tb_2>\tb_1\geq \tb_0$,
\begin{align}
\label{eq:EnerDecay:Impro:Psiplus:l0:0to2:v2:Integralform}
F^{(i)}(\reg,p,\tb_2,\Psiplus)
+\int_{\tb_1}^{\tb_2}F^{(i)}(\reg,p-1,\tb,\Psiplus)\lesssim_{p,\reg} {}&F^{(i)}(\reg+\regl,p,\tb_1,\Psiplus).
\end{align}
Together with the fact that for any $i\in \mathbb{N}$, $F^{(i)}(\reg,2,\tb,\Psiplus)\sim F^{(i+1)}(\reg,0,\tb,\Psiplus)$, we arrive at for any $p\in [0,2)$ and any $1\leq i\leq \ell_0-1$, there exists a constant $\regl(j,\ell_0-i)$ such that
\begin{align}
\label{eq:EnerDecay:Impro:Psiplus:l0:0to2:v2}
F^{(i)}(\reg,p,\tb_1,\Lxi^j\Psiplus)\lesssim_{\delta,j,\ell_0,\reg}{}&
\tb^{-2(\ell_0-1-i)-2j-2+p+C\delta}
F^{(\ell_0-1)}(\reg+\regl(j,\ell_0-i),2-\delta,\tb_0,\Psiplus).
\end{align}
Furthermore, since $F^{(1)}(\reg,0,\tb,\Lxi^j\Psiplus)\sim
F(\reg,2,\tb,\Lxi^j\Psiplus)$, we obtain from inequality \eqref{eq:BEDC:Phiplus} that for any $p\in [0,2]$,
\begin{align}
\label{eq:EnerDecay:Impro:Psiplus:l0:0to2:v1}
F(\reg,p,\tb,\Lxi^j\Psiplus)\lesssim_{\delta,j,\ell_0,\reg} {}&\tb^{-2(\ell_0-1)-2j-2+p+C\delta}
F^{(\ell_0-1)}(\reg+\regl(j,\ell_0),2-\delta,\tb_0,\Psiplus).
\end{align}
We apply  now the estimate \eqref{eq:rp:p=2:2} in Proposition \ref{prop:wave:rp:highmodes} with $p=2$ to equation \eqref{eq:Phiplushighi:Schw:l0} and find the error term arising from the last two terms on the LHS \eqref{eq:Phiplushighi:Schw:l0} is bounded by a small portion of the spacetime integral of  the LHS, which is thus absorbed, plus an integral
\begin{align}
\hspace{4ex}&\hspace{-4ex}\bigg|\sum_{\abs{\mathbf{a}}\leq \reg-1-i}\int_{\Donetwo^{R_0}}
V\overline{\mathbb{D}_2^{\mathbf{a}}\PhiplusHigh{i}}
\mathbb{D}_2^{\mathbf{a}}\PhiplusHigh{i-1} \di^4 \mu\bigg|\notag\\
\leq{}&\veps\int_{\tb_1}^{\tb_2}\frac{1}{\tb^{1+\delta}}
\Big(\norm{rV\PhiplusHigh{i}}^2_{W_{0}^{\reg-1-i}(\Sigmatb^{R_0})}
+\norm{\PhiplusHigh{i}}^2_{W_{-2}^{\reg-i}(\Sigmatb^{R_0})}\Big)\di\tb\notag\\
&
+\frac{C}{\veps}\int_{\tb_1}^{\tb_2}
\tb^{1+\delta}\norm{\PhiplusHigh{i-1}}^2_{W_{-2}^{\reg}(\Sigmatb^{R_0})}
\di\tb.
\end{align}
The first term is absorbed by choosing $\veps$ small and the second term is bounded from the estimates \eqref{eq:EnerDecay:Impro:Psiplus:l0:0to2:v2} and \eqref{eq:EnerDecay:Impro:Psiplus:l0:0to2:v1} by $C\tb_1^{-2(\ell_0-i)+C\delta}F^{(\ell_0-1)}
(\reg+\regl(\ell_0-i),2-\delta,\tb_0,\Psiplus)$.
Running again the above discussions, one eventually obtains
for any $p\in [0,2]$,
\begin{align}
\label{eq:EnerDecay:Impro:Psiplus:l0:0to2}
F(\reg,p,\tb,\Lxi^j\Psiplus)\lesssim_{j,\ell_0,\reg} {}&\tb^{-2(\ell_0-1)-2j-2+p}
F^{(\ell_0-1)}(\reg+\regl(j,\ell_0),2,\tb_0,\Psiplus).
\end{align}

Equation \eqref{eq:Phiplushighi:Schw:generall:tildePhiplusi} of $\tildePhiplusHigh{i}$ for $i=\ell_0-1$ reads
\begin{align}
\label{eq:Phiplushighi:Schw:l=l0:tildePhiplusi}
&-r^2 YV  \tildePhiplusHigh{\ell_0-1}
-2\ell_0(r-3M)r^{-2}\curlVR\tildePhiplusHigh{\ell_0-1}
+\sum_{j=0}^{\ell_0-1}h_{\ell_0-1,j} \PhiplusHigh{j}={}0,
\end{align}
with $h_{\ell_0-1,j}=O(r^{-1})$ for all $j\in \{0,1,\ldots, \ell_0-1\}$.
By multiplying this equation by $-2r^{p-2}\chi^2 V\overline{\tildePhiplusHigh{\ell_0-1}}$, taking the real part, and integrating over $\Donetwo$ with a measure $\di^4 \mu$,
\begin{align}
\label{eq:rpinfty:Phiplus:Schw:l0:generalp}
\int_{\Donetwo}\Big(&Y(2r^{p}\chi^2 \abs{V\tildePhiplusHigh{\ell_0-1}}^2)
+((p\chi^2+r\partial_r(\chi^2)  ) r^{p-1}+2\ell_0r^{p}\Delta^{-1} (r-3M))\chi^2 \abs{V\tildePhiplusHigh{\ell_0-1}}^2\Big)\di^4\mu\notag\\
\hspace{4ex}&\hspace{-4ex}={}-\int_{\Donetwo} 2\chi^2r^{p-3} \sum_{j=0}^{\ell_0-1}rh_{\ell_0-1,j}\Re(V\overline{\tildePhiplusHigh{\ell_0-1}} \PhiplusHigh{j})\di^4\mu.
\end{align}
For $p\in (2,4)$, the integral on the RHS is bounded by $\int_{\Donetwo}(\veps r^{p-1}\chi_{R}^2\abs{V\tildePhiplusHigh{\ell_0-1}}^2
+C\veps^{-1}\sum\limits_{j=0}^{\ell_0-1}
r^{p-5}\chi_{R}^2\abs{\PhiplusHigh{j}}^2)\di^4\mu$. The $\veps$ part is absorbed by the LHS,  and the second term is bounded via the Hardy's inequality \eqref{eq:HardyIneqLHSRHS} by
\begin{align}
\hspace{4ex}&\hspace{-4ex}
C\veps^{-1}\sum\limits_{j=0}^{\ell_0-1}\Big(\int_{\Donetwo^{R_0-M}} r^{p-3}\abs{\prb \PhiplusHigh{j}}^2 \di^4\mu+\int_{\Donetwo^{R_0-M, R_0} } r^{p-3}\abs{\PhiplusHigh{j}}^2\di^4\mu\Big)\notag\\
\lesssim{}& \veps^{-1}\sum\limits_{j=0}^{\ell_0-1}\int_{\Donetwo^{R_0-M}} r^{p-7}(\abs{\PhiplusHigh{j}}^2+\abs{Y\PhiplusHigh{j}}^2 ) \di^4\mu
+\int_{\Donetwo^{R_0-M}} r^{p-3}\abs{V\PhiplusHigh{\ell_0-1}}^2 \di^4\mu,
\end{align}
where we have used the definition of $\PhiplusHigh{j+1}$ to rewrite $\prb \PhiplusHigh{j}$. One can bound these terms by $F^{(\ell_0-1)}(\reg,2,\tb_1,\Psiplus)$ using the estimate \eqref{eq:EnerDecay:Impro:Psiplus:l0:0to2:v2:Integralform} if $p\in [2,4)$, thus, for any $p\in [2,4)$,
\begin{align}
\label{eq:rp:Psiplus:Schw:l0:2to4}
F^{(\ell_0-1)}(\reg,p,\tb_2,\Psiplus)
+\int_{\tb_1}^{\tb_2}F^{(\ell_0-1)}(\reg-1,p-1,\tb,\Psiplus)\di\tb
\lesssim_{p,j,\reg}{}F^{(\ell_0-1)}(\reg+\regl,p,\tb_1,\Psiplus).
\end{align}

In the case that the $\ell_0$-th N-P constant $\NPCP{\ell_0}$ does not vanish,
one can use the above inequality for $p\in [2,3)$ and the estimate \eqref{eq:EnerDecay:Impro:Psiplus:l0:0to2} and obtain for any $\delta>0$ and $p\in [0,2]$,
\begin{align}
\label{eq:EnerDecay:Impro:Psiplus:l0:0to3}
F(\reg,p,\tb,\Lxi^j\Psiplus)\lesssim{}&\tb^{-2(\ell_0-1)-2j-3+\delta+p}
F^{(\ell_0-1)}(\reg+\regl(j,\ell_0),3-\delta,\tb_0,\Psiplus).
\end{align}
This proves inequality \eqref{eq:BEDC:Phiplus:l=l0:0to3}.

In the second case where the $\ell_0$-th N-P constant $\NPCP{\ell_0}$ vanishes,
we utilize the estimate \eqref{eq:rp:Psiplus:Schw:l0:2to4} and the estimate \eqref{eq:EnerDecay:Impro:Psiplus:l0:0to2} to achieve for any $\delta>0$, $p\in [0,2]$ and $1\leq i\leq \ell_0-1$,
\begin{subequations}
\label{eq:EnerDecay:Impro:Psiplus:l0:0to4}
\begin{align}
F^{(i)}(\reg,p,\tb_1,\Lxi^j\Psiplus)\lesssim_{\delta,j,\ell_0,\reg}{}&
\tb^{-2(\ell_0-1-i)-2j-4+\delta+p}
F^{(\ell_0-1)}(\reg+\regl(j,\ell_0-i),4-\delta,\tb_0,\Psiplus),\\
F(\reg,p,\tb,\Lxi^j\Psiplus)\lesssim_{\delta,j,\ell_0,\reg}
{}&\tb^{-2(\ell_0-1)-2j-2+p}
F^{(\ell_0-1)}(\reg+\regl(j,\ell_0),2,\tb/2,\Psiplus)\notag\\
\lesssim_{\delta,j,\ell_0,\reg}{}&\tb^{-2(\ell_0-1)-2j-4+\delta+p}
F^{(\ell_0-1)}(\reg+\regl(j,\ell_0),4-\delta,\tb_0,\Psiplus).
\end{align}
\end{subequations}
Consider equation \eqref{eq:rpinfty:Phiplus:Schw:l0:generalp} for $p\in [4,5)$. The integral on the RHS is bounded by
\begin{align}
\veps\int_{\Donetwo}r^{p} \tb^{-1-\delta}\chi^2_R \abs{V\tildePhiplusHigh{\ell_0-1}}^2 \di^4\mu +\veps^{-1}\sum_{j=0}^{\ell_0-1}\int_{\Donetwo}r^{p-6}\tb^{1+\delta} \chi^2_R \abs{\PhiplusHigh{j}}^2 \di^4\mu.
\end{align}
The $\veps$ part is absorbed by the LHS for small $\veps$, and in view of the estimates \eqref{eq:EnerDecay:Impro:Psiplus:l0:0to4}, the $\veps^{-1}$ part is dominated by
\begin{align}
\hspace{8ex}&\hspace{-8ex}
\int_{\tb_1}^{\tb_2}\tb^{1+\delta}
\bigg(F(\ell_0,p-4,\tb,\Psiplus)
+\sum_{i=1}^{\ell_0-1}F^{(i)}(1,p-4,\tb,\Psiplus)\bigg)\di\tb\notag\\
\lesssim_{\delta,j,\ell_0} {}& \tb_1^{-6+2\delta+p}F^{(\ell_0-1)}(\regl(\ell_0),4-\delta,\tb_0\Psiplus).
\end{align}
Similar to the proof of Proposition \ref{prop:BED:Psiplus:l=1}, we can thus obtain
$F^{(\ell_0-1)}(\reg,2,\tb,\Psiplus)\lesssim_{\delta,\reg} \tb^{-3+\delta}F^{(\ell_0-1)}(\reg+\regl,5-\delta,\tb/2,\Psiplus)$,
which together with the estimates \eqref{eq:EnerDecay:Impro:Psiplus:l0:0to4} closes the proof.
\end{proof}

\subsection{$\ell=1$ mode of spin $-1$ component}
\label{sect:spin-1:l=1:Schw}

\begin{prop}
Let $j\in \mathbb{N}$.  Let $\Psiminus$ be supported on $\ell=1$ mode. Let $F^{(1)}(\reg,p,\tb,\Psiminus)$ and $F^{(2)}(\reg,p,\tb,\Psiminus)$ for $p\in [-1,2]$ be defined as in Definition \ref{def:Fterm:Psiminus:1and2level}. Define additionally for any $p\in (2,5)$ that
\begin{align}
\label{def:Ffts:Psiminus:2to3}
F^{(2)}(\reg,p,\tb,\Psiminus)={}&\sum_{i=0,1}
\Big(\norm{rV\PsiminusHigh{i}}^2_{W_{0}^{\reg}(\Sigmatb)}
+\norm{\PsiminusHigh{i}}^2_{W_{-2}^{\reg+1}(\Sigmatb)}\Big)\notag\\
&+\norm{rV\Phiminus{2}}^2_{W_{p-2}^{\reg-1}(\Sigmatb^{3M})}
+\norm{\Phiminus{2}}^2_{W_{-2}^{\reg}(\Sigmatb^{3M})}.
\end{align}
Then,
\begin{enumerate}
\item\label{pt:Psiminus:l=1:0to3} if the first N-P constant $\NPCN{1}$ does not vanish,
there is a constant $\regl(j)$ such that for any small $\delta>0$, any $p\in [0,3-\delta]$ and any $\tb\geq\tb_0$,
\begin{subequations}
\label{eq:BEDC:Psiminus:l=1:0to3}
\begin{align}
F^{(2)}(\reg,p,\tb,\Lxi^j\Psiminus)\lesssim_{\delta,j,\reg}  {}&\tb^{-3+\delta-2j+p}
F^{(2)}(\reg+\regl(j),3-\delta,\tb_0,\Psiminus),\\
F^{(1)}(\reg,p,\tb,\Lxi^j\Psiminus)\lesssim_{\delta,j,\reg}  {}&\tb^{-5+\delta-2j+p}
F^{(2)}(\reg+\regl(j),3-\delta,\tb_0,\Psiminus),
\end{align}
\end{subequations}
i.e., the basic energy $\gamma$-decay condition with $\gamma=3-\delta$ holds for spin $-1$ component;
\item\label{pt:Psiminus:l=1:0to5}
if the first N-P constant $\NPCN{1}$ vanishes, there is a constant $\regl(j)$ such that for any small $\delta>0$, any $p\in [0,5-\delta]$ and any $\tb\geq\tb_0$,
\begin{subequations}
\label{eq:BEDC:Psiminus:l=1:0to5}
\begin{align}
F^{(2)}(\reg,p,\tb,\Lxi^j\Psiminus)\lesssim_{\delta,j,\reg}  {}&\tb^{-5+\delta-2j+p}
F^{(2)}(\reg+\regl(j),5-\delta,\tb_0,\Psiminus),\\
F^{(1)}(\reg,p,\tb,\Lxi^j\Psiminus)\lesssim_{\delta,j,\reg}  {}&\tb^{-7+\delta-2j+p}
F^{(2)}(\reg+\regl(j),5-\delta,\tb_0,\Psiminus),
\end{align}
\end{subequations}
i.e., the basic energy $\gamma$-decay condition with $\gamma=5-\delta$ holds for spin $-1$ component.
\end{enumerate}
\end{prop}

\begin{proof}
The system \eqref{eq:Phi-1012} of equations reduces to
\begin{subequations}
\label{eq:Phi-1012:Schw}
\begin{align}
\label{eq:Phi-10:Schw}
&-r^2YV\Phiminus{0} -2\Phiminus{0}
={}-\tfrac{2(r-3M)}{r^2}\Phiminus{1},\\
\label{eq:Phi-11:Schw}
&-r^2YV\Phiminus{1} -2\Phiminus{1}={}0,\\
\label{eq:Phi-12:Schw}
&-r^2YV\Phiminus{2}-2\mu^{-1}(r-3M)
{V\Phiminus{2}}-12Mr^{-1}\Phiminus{2}
={}0.
\end{align}
\end{subequations}
The last subequation \eqref{eq:Phi-12:Schw} is exactly of the same form as the equation \eqref{eq:Phi+1:Schw:l=1}. Thus the same way of arguing as in the proof of Proposition \ref{prop:BED:Psiplus:l=1} applies and yields the decay estimates for $F^{(2)}(\reg,p,\tb,\Lxi^j\Psiminus)$. One can further follow the proof in Proposition \ref{prop:BEDC:Psiminus} (in particular, the relation  $F^{(2)}(\reg,0,\tb,\Lxi^j\Psiminus)\sim F^{(1)}(\reg,2,\tb,\Lxi^j\Psiminus)$ is crucial) and arrive at the estimates for $F^{(1)}(\reg,p,\tb,\Lxi^j\Psiminus)$.
\end{proof}

\subsection{$\ell=\ell_0$ mode of spin $-1$ component with $\ell_0\geq 2$}
\label{sect:Maxwell:Schw:l=l0:minus}

\begin{prop}
Let $j\in \mathbb{N}$ and $\reg\geq \ell_0$.Let $F^{(1)}(\reg,p,\tb,\Psiminus)$ and $F^{(2)}(\reg,p,\tb,\Psiminus)$ for $p\in [-1,2]$ be defined as in Definition \ref{def:Fterm:Psiminus:1and2level}.  Define for any $p\geq 2$ that
\begin{align}
\label{def:Ffts:Psiminus:l0:generalp}
F^{(\ell_0+1)}(\reg,p,\tb,\Lxi^j\Psiminus)={}&\sum_{i=0,1}
\Big(\norm{rV\PsiminusHigh{i}}^2_{W_{0}^{\reg}(\Sigmatb)}
+\norm{\PsiminusHigh{i}}^2_{W_{-2}^{\reg+1}(\Sigmatb)}\Big)\notag\\
&+\sum_{i=2}^{\ell_0}\Big(
\norm{rV\Phiminus{i}}^2_{W_{0}^{\reg+1-i}(\Sigmatb^{3M})}
+\norm{\Phiminus{i}}^2_{W_{-2}^{\reg+2-i}(\Sigmatb^{3M})}\Big)\notag\\
&+\norm{rV\tildePhiminus{\ell_0+1}}^2_{W_{p-2}^{\reg-\ell_0}(\Sigmatb^{3M})}
+\norm{\Phiminus{\ell_0+1}}^2_{W_{-2}^{\reg-\ell_0+1}(\Sigmatb^{3M})}\Big).
\end{align}
Assume $\Psiminus$ is supported on a fixed $\ell=\ell_0$ mode with $\ell_0\geq 2$. Then,
\begin{enumerate}
\item\label{pt:Psiminus:l=l0:0to3} if the $\ell_0$-th N--P constant does not vanish,
there is a constant $\regl(j,\ell_0)$ such that for any small $\delta>0$, any $p\in [0,3-\delta]$ and any $\tb\geq\tb_0$,
\begin{subequations}
\label{eq:BEDC:Psiminus:l=l0:0to3}
\begin{align}
F^{(2)}(\reg,p,\tb,\Lxi^j\Psiminus)\lesssim_{\delta,j,\reg,\ell_0}  {}&\tb^{-3+\delta-2(\ell_0-1)-2j+p}
F^{(\ell_0+1)}(\reg+\regl(j,\ell_0),3-\delta,\tb_0,\Psiminus),\\
F^{(1)}(\reg,p,\tb,\Lxi^j\Psiminus)\lesssim_{\delta,j,\reg,\ell_0}  {}&\tb^{-5+\delta-2(\ell_0-1)-2j+p}
F^{(\ell_0+1)}(\reg+\regl(j,\ell_0),3-\delta,\tb_0,\Psiminus),
\end{align}
\end{subequations}
i.e., the basic energy $\gamma$-decay condition holds for spin $-1$ component for $\gamma=3-\delta+2(\ell_0-1)$;
\item\label{pt:Psiminus:l=l0:0to5}
if the $\ell_0$-th N--P constant vanishes, there is a constant $\regl(j)$ such that for any small $\delta>0$, any $p\in [0,5-\delta]$ and any $\tb\geq\tb_0$,
\begin{subequations}
\label{eq:BEDC:Psiminus:l=l0:0to5}
\begin{align}
F^{(2)}(\reg,p,\tb,\Lxi^j\Psiminus)\lesssim_{\delta,j,\reg,\ell_0}  {}&\tb^{-5+\delta-2(\ell_0-1)-2j+p}
F^{(\ell_0+1)}(\reg+\regl(j,\ell_0),5-\delta,\tb_0,\Psiminus),\\
F^{(1)}(\reg,p,\tb,\Lxi^j\Psiminus)\lesssim_{\delta,j,\reg,\ell_0}  {}&\tb^{-7+\delta-2(\ell_0-1)-2j+p}
F^{(\ell_0+1)}(\reg+\regl(j,\ell_0),5-\delta,\tb_0,\Psiminus),
\end{align}
\end{subequations}
i.e., the basic energy $\gamma$-decay condition holds for spin $-1$ component for $\gamma=5-\delta+2(\ell_0-1)$.
\end{enumerate}
\end{prop}

\begin{proof}
The system \eqref{eq:Phi-1012} of equations reduces to
\begin{subequations}
\label{eq:Phi-1012:Schw:l0}
\begin{align}
\label{eq:Phi-10:Schw:l0}
&-r^2YV\Phiminus{0} -\ell_0(\ell_0+1)\Phiminus{0}
={}-\tfrac{2(r-3M)}{r^2}\Phiminus{1},\\
\label{eq:Phi-11:Schw:l0}
&-r^2YV\Phiminus{1} -\ell_0(\ell_0+1)\Phiminus{1}={}0,\\
\label{eq:Phi-12:Schw:l0}
&-r^2YV\Phiminus{2}-(\ell_0(\ell_0+1)-2)\Phiminus{2}
-2\mu^{-1}(r-3M)
{V\Phiminus{2}}-12Mr^{-1}\Phiminus{2}
={}0.
\end{align}
\end{subequations}
Define for $p\in [0,2]$ that
\begin{align}
\label{def:Ffts:Psiminus:l0:0to2}
F^{(\ell_0+1)}(\reg,p,\tb,\Lxi^j\Psiminus)={}&\sum_{i=0,1}
\Big(\norm{rV\PsiminusHigh{i}}^2_{W_{0}^{\reg}(\Sigmatb)}
+\norm{\PsiminusHigh{i}}^2_{W_{-2}^{\reg+1}(\Sigmatb)}\Big)\notag\\
&+\sum_{i=2}^{\ell_0+1}\Big(
\norm{rV\Phiminus{i}}^2_{W_{p-2}^{\reg+1-i}(\Sigmatb^{3M})}
+\norm{\Phiminus{i}}^2_{W_{-2}^{\reg+2-i}(\Sigmatb^{3M})}\Big).
\end{align}
One finds the last subequation \eqref{eq:Phi-12:Schw:l0} is exactly of the same form as the equation \eqref{eq:Phi+1:Schw:l0}. Thus by arguing the same as in the proof of Proposition \ref{prop:BED:Psiplus:l=l0}, one can show the claimed estimates for $F^{(2)}(\reg,p,\tb,\Lxi^j\Psiminus)$. We then go back to the proof in Proposition \ref{prop:BEDC:Psiminus} and utilize the relation  $F^{(2)}(\reg,0,\tb,\Lxi^j\Psiminus)\sim F^{(1)}(\reg,2,\tb,\Lxi^j\Psiminus)$ to achieve the estimates for $F^{(1)}(\reg,p,\tb,\Lxi^j\Psiminus)$.
\end{proof}

\subsection{Closing the proof of Theorem \ref{thm:Schw}}
\label{sect:close:thm:Schw}

\begin{prop}
\label{prop:ptwFromBEDC:extremecomps:Schw}
Consider a Maxwell field in a Schwarzschild spacetime.
Let $j\in \mathbb{N}$.
\begin{enumerate}
\item \label{pt1:prop:ptwFromBEDC:extremecomps:Schw}
Let the basic energy $\gamma$-decay condition  with a $\gamma\geq 1$, a suitably large $\reg$ and $D_{+ 1}=D_{+ 1}(M,a,\reg,j)$ be satisfied for spin $+ 1$ component, then there exists a $\regl$ such that
\begin{align}
\absCDeri{\Lxi^j(r^{-2}\psiplus)}{\reg-\regl}
\lesssim{}&(D_{+1})^{\half}v^{-1}r^{-2}\tb^{-(\gamma-1)/2 -j}.
\end{align}
\item
Let the basic energy $\gamma$-decay condition  with a $\gamma\geq 1$, a suitably large $\reg$ and $D_{- 1}=D_{- 1}(M,a,\reg,j)$ be satisfied for spin $- 1$ component, then there exists a $\regl(j)$ such that
\begin{align}
\label{eq:ptwdecay:psiminus:gammadecay:Schw:l=1}
\absCDeri{\Lxi^j\psiminus}{\reg-\regl(j)}
\lesssim{}&(D_{-1})^{\half}v^{-1}\tb^{-(3+\gamma)/2 -j}.
\end{align}
\item Let the basic energy $\gamma$-decay condition  with a $\gamma\geq 1$, a suitably large $\reg$ and $D_{\pm 1}=D_{\pm 1}(M,a,\reg,j)$ be satisfied for spin $\pm 1$ components, then there exists a $\regl(j)$ such that
\begin{align}
\absCDeri{\Lxi^j\psiminus}{\reg-\regl(j)}
\lesssim{}&(D_{-1})^{\half}v^{-1}\tb^{-(3+\gamma)/2 -j},\\
\absCDeri{\Lxi^j(r^{-2}\psiplus)}{\reg-\regl(j)}
\lesssim{}&(D_{+1}+D_{-1})^{\half}v^{-3}\tb^{-(\gamma-1)/2 -j}.
\end{align}
\end{enumerate}
\end{prop}

\begin{proof}
The first point has been shown in  Section \ref{sect:vrdecay:psiplus}, and assuming the second point is valid, the third point is justified in Sections \ref{sect:vdecay:psiplus} and \ref{sect:vtdecay:psimiddle}. Hence, we only need to show the second point about spin $-1$ component. Moreover, we consider only the interior region $\{\rb\leq \tb\}$, since if in the exterior region $\{\rb>\tb\}$, the estimate \eqref{eq:ptwdecay:psiminus:gammadecay:Schw:l=1} follows from the already proven estimate \eqref{eq:evenimprovePTWDecay:phiminus}.

Consider first $\ell=1$ mode.  Equation \eqref{eq:W-3energynorm:spin-1:Schw} then becomes
\begin{align}\label{eq:W-3energynorm:spin-1:Schw:l=1}
\hspace{4ex}&\hspace{-4ex}
\int_{\Sigmatb}\Big(6\mu r^{-5}\abs{\Phiminus{1}-r\Psiminus}^2+\mu r
\abs{\prb(r^{-2}{\Phiminus{1}})}^2
+2\mu^2 r\abs{\prb(r^{-1}\Psiminus)}^2\Big)\di^3 \mu\notag\\
={}&\int_{\Sigmatb}\Big(\mu r^{-3}\abs{G(\Psiminus)}^2
+2r^{-1}\abs{ \mu H\Lxi \Psiminus}^2\Big)\di^3\mu.
\end{align}
The wave equation \eqref{eq:waveofPsiminus:intermsofPhi1} takes the form of
\begin{align}\label{eq:VmuYpsiminus:Schw:l=1}
\VR(\mu r^4 Y\psiminus)= {}6 r^3\Lxi\psiminus.
\end{align}
This implies in the interior region that
\begin{align}
\abs{\prb(\mu r^4 Y \psiminus)}\lesssim_{\delta}{}& (D_{-1})^{\half} \tb^{-\frac{7+\gamma}{2}+\delta}r^{3-\delta}.
\end{align}
Integrating from horizon gives
\begin{align}
\label{eq:Ypsiminus:Schw}
\abs{Y \psiminus}\lesssim_{\delta} {}&(D_{-1})^{\half}\tb^{-\frac{7+\gamma}{2}+\delta}r^{-\delta}.
\end{align}
This estimate and the estimate \eqref{eq:evenimprovePTWDecay:phiminus} together show that
\begin{align}
\label{eq:prbpsiminus:Schw}
\abs{\prb \psiminus}\lesssim_{\delta} {}&(D_{-1})^{\half}\tb^{-\frac{7+\gamma}{2}+\delta}r^{-\delta}.
\end{align}
Integrating from the point $\rb=\tb$ and in view of the fact that $\abs{\psiminus(\tb, \rb=\tb,\theta,\pb)}\lesssim(D_{-1})^{\half}\tb^{-\frac{5+\gamma}{2}}$, one immediately sees that for any $\rb\leq \tb$,
\begin{align}
\abs{\psiminus(\tb, \rb,\theta,\pb)}\lesssim{}&(D_{-1})^{\half}\tb^{-\frac{5+\gamma}{2}}.
\end{align}
Clearly, one can commute with $\Lxi$ to obtain
\begin{align}
\label{eq:psiminus:gooddecay:uptosecond:l=1}
\abs{\mu \prb\Lxi^j(r\prb\psiminus)}+\abs{\prb\Lxi^j\psiminus}
\lesssim_{\delta}{}(D_{-1})^{\half}\tb^{-\frac{7+\gamma}{2}+\delta-j}r^{-\delta}, \quad
\abs{\Lxi^j\psiminus}\lesssim{}(D_{-1})^{\half}\tb^{-\frac{5+\gamma}{2}-j}.
\end{align}
We shall show below that $\abs{\prb(r\prb\Lxi^j\psiminus)}
\lesssim_{\delta}(D_{-1})^{\half}\tb^{-\frac{7+\gamma}{2}+\delta-j}r^{-\delta}$.
In view of the following commutator for any scalar $\varphi$
\begin{align}
\label{eq:commmutator:rYandvariWave:Schw}
[rY,\VR(\mu r^4 Y)]\varphi={}&r^4\partial_r (\mu ) Y(rY\varphi) -2r^3\Lxi(rY\varphi)
-2\VR (\mu r^4 Y\varphi)\notag\\
&-2\mu^2 r^7 \partial_r (\mu^{-1}r^3)Y\varphi - r^2\partial_{r}^2(\mu r^3) Y\varphi,
\end{align}
commuting equation \eqref{eq:VmuYpsiminus:Schw:l=1} with $rY$ gives \begin{align}\label{eq:VmuYpsiminus:Schw:l=1}
\hspace{4ex}&\hspace{-4ex}\VR(\mu r^4 Y(rY\psiminus))
+r^4\partial_r (\mu ) Y(rY\psiminus)\notag\\
= {}&6 rY(r^3\Lxi\psiminus)
 +2r^3\Lxi(rY\psiminus)
+12r^3 \Lxi\psiminus\notag\\
&+2\mu^2 r^7 \partial_r (\mu^{-1}r^3)Y\psiminus + r^2\partial_{r}^2(\mu r^3) Y\psiminus.
\end{align}
This yields $\prb(\mu r^4 Y(rY\psiminus))
+r^4\partial_r (\mu ) Y(rY\psiminus)= F(\psiminus)$ where $\abs{F(\psiminus)}\lesssim_{\delta} (D_{-1})^{\half} \tb^{-\frac{7+\gamma}{2}+\delta}r^{3-\delta}$. For $\rb_0$ finite, one can multiply this equation by $\mu Y(rY\psiminus)$ and integrate over $\rb$ from $r_+$ to $\rb_0$. The integral involving $F(\psiminus)$ can be bounded via a Cauchy-Schwarz inequality by $\veps \int_{r_+}^{\rb_0}\mu \abs{Y(rY\psiminus)}^2 +\veps^{-1}\int_{r_+}^{\rb_0}\mu \abs{F(\psiminus)}^2$ where the first term can be absorbed by the LHS and the second term is bounded by $D_{-1}(\tb_0-r_+)^2\tb^{-7-\gamma+2\delta}$. This together with \eqref{eq:psiminus:gooddecay:uptosecond:l=1} thus gives for $i\in \{1,2\}$,
\begin{align}
\abs{r^{-1}(r\prb)^i\Lxi^j\psiminus}\lesssim_{\delta}{}&(D_{-1})^{\half}
\tb^{-\frac{7+\gamma}{2}+\delta-j}r^{-\delta}.
\end{align}
One can repeat the above discussions by using the commutator \eqref{eq:commmutator:rYandvariWave:Schw} to show the above estimate holds for any $i\in \mathbb{N}^+$, which then proves \eqref{eq:ptwdecay:psiminus:gammadecay:Schw:l=1} for $\ell=1$ mode.

We next consider $\ell\geq 2$ modes. As discussed in Remark \ref{rem:evenimproveEnerDecay:phiminus:Schw}, the RHS of \eqref{eq:W-3energynorm:spin-1:Schw} decays like $\tb^{-5-\gamma}$, and the LHS for $\ell\geq 2$ modes is larger than
\begin{align}
\hspace{4ex}&\hspace{-4ex}
c\int_{\Sigmatb}\big(
\mu r^{-5}\abs{\mathbf{S}^{\half}\Phiminus{1}}^2
+2M r^{-6}\abs{\mathbf{S}^{\half}\Phiminus{1}}^2
+\mu r
\abs{\prb(r^{-2}{\Phiminus{1}})}^2\notag\\
&\quad +\mu r^{-3}\abs{\mathbf{S}\Psiminus}^2
+\mu^2 r\abs{\prb(r^{-1}\mathbf{S}^{\half}\Psiminus)}^2\big)\di^3 \mu\notag\\
\gtrsim{}&\int_{\Sigmatb}\big(
r^{-5}\abs{\mathbf{S}^{\half}\Phiminus{1}}^2
+\mu r
\abs{\prb(r^{-2}{\Phiminus{1}})}^2
+\mu r^{-3}\abs{\mathbf{S}\Psiminus}^2\notag\\
&\qquad
+r^{-3}\abs{\mathbf{S}^{\half}\Psiminus}^2
+\mu^2 r\abs{\prb(r^{-1}\mathbf{S}^{\half}\Psiminus)}^2
\big)\di^3 \mu.
\end{align}
One can follow the same way of arguing in the proof of Proposition \ref{prop:evenimproveEnerDecay:phiminus}
and obtain spacetime integral decay
\begin{align}
\label{eq:evenimproveEnerDecay:phiminus:Schw}
\norm{\Lxi^j\Psiminus
}^2_{W_{-3}^{\reg-\regl}(\DOC_{\tb,\infty})}
\lesssim{}D_{-1}\tb^{-4-\gamma-2j},
\end{align}
It is then manifest from the estimate \eqref{eq:improvePTWDecay:phiminus} and inequality \eqref{eq:Sobolev:3} that
\begin{align}
\absCDeri{\Lxi^j\psiminus}{\reg-\regl}\lesssim{} (D_{-1})^{\half} \tb^{-\frac{5+\gamma}{2}-j}.
\end{align}
Together with \eqref{eq:weakdecay:spin-1:v3}, this proves \eqref{eq:ptwdecay:psiminus:gammadecay:Schw:l=1} for $\ell\geq 2$ modes.
\end{proof}
\emph{Proof of Theorem \ref{thm:Schw}}:
We consider only the spin $+1$ component, since the proof for spin $-1$ component is analogous and the estimates of the middle component then follow from the estimates of spin $\pm 1$ components and Proposition \ref{prop:estiofMiddlecomp:generalassupofextremecomps}.

In the case that the $\ell_0$-th N--P constant $\NPCP{\ell_0}$ of the $\ell=\ell_0$ mode $\Psiplus^{\ell=\ell_0}$ does not vanish, then from Proposition \ref{prop:BED:Psiplus:l=l0} the basic energy $\gamma$-decay condition with $\gamma=2\ell_0+1-\delta$ and $D_{+1}=F^{(\ell_0-1)}(\reg+\regl(j,\ell_0),3-\delta,\tb_0,\Psiplus)$ holds for $\ell=\ell_0$ mode. For the part $\Psiplus^{\ell\geq \ell_0+1}$ which are supported on $\ell\geq \ell_0+1$ modes, one can in fact obtain an $r^p$ estimate for $\PhiplusHigh{\ell_0}$ for $p\in [0,1-\delta]$. It is clear from the proof in Proposition \ref{prop:BED:Psiplus:l=l0} that they satisfy the basic energy $\gamma$-decay condition with $\gamma=2\ell_0+1-\delta$ and
\begin{align}
D_{+1}={}&\norm{\Psiplus}^2_{W_{-2}^{\reg+\regl(j,\ell_0)}(\Sigmazero)}
+\sum_{m=1}^{\ell_0}\norm{\PhiplusHigh{m}}^2_{W_{-2}^{\reg+\regl(j,\ell_0)-m}(\Sigmazero^{3M})}+
\norm{rV\PhiplusHigh{\ell_0}}^2_{W_{-1-\delta}^{\reg+\regl(j,\ell_0)-\ell_0-1}(\Sigmatb^{3M})}\notag\\
\lesssim{}&F^{(\ell_0-1)}(\reg+\regl(j,\ell_0),3-\delta,\tb_0,\Psiplus).
\end{align}
In total, the basic energy $\gamma$-decay condition with $\gamma=2\ell_0+1-\delta$ and $D_{+1}=F^{(\ell_0-1)}(\reg+\reg'(j,\ell_0),3-\delta,\tb_0,\Psiplus)$ holds.

In the other case that the $\ell_0$-th N--P constant $\NPCP{\ell_0}$ of the $\ell=\ell_0$ mode $\Psiplus^{\ell=\ell_0}$ vanishes, Proposition \ref{prop:BED:Psiplus:l=l0} implies the basic energy $\gamma$-decay condition with $\gamma=2\ell_0+3-\delta$ and $D_{+1}=F^{(\ell_0-1)}(\reg+\regl(j,\ell_0),5-\delta,\tb_0,\Psiplus)$ holds for $\ell=\ell_0$ mode. We then turn to $\ell=\ell_0+1$ mode, and find from Point \ref{pt1: prop:BED:Psiplus:l=l0} of Proposition \ref{prop:BED:Psiplus:l=l0} that the basic energy $\gamma$-decay condition with $\gamma=2\ell_0+3-\delta$ and \begin{align}
D_{+1}={}&\norm{\Psiplus}^2_{W_{-2}^{\reg+\regl(j,\ell_0)}(\Sigmazero)}
+\sum_{m=1}^{\ell_0}\norm{\PhiplusHigh{m}}^2_{W_{-2}^{\reg+\regl(j,\ell_0)-m}(\Sigmazero^{3M})}+
\norm{rV\tildePhiplusHigh{\ell_0}}^2_{W_{1-\delta}^{\reg+\regl(j,\ell_0)-\ell_0-1}(\Sigmatb^{3M})}\notag\\
\lesssim{}&F^{(\ell_0-1)}(\reg+\regl(j,\ell_0),5-\delta,\tb_0,\Psiplus).
\end{align}
Consider in the end $\Psiplus^{\ell\geq \ell_0 +2}$. One can achieve an $r^p$ estimate for $\PhiplusHigh{\ell_0+1}$ for $p\in [0,1-\delta]$, and hence it satisfies the basic energy $\gamma$-decay condition with $\gamma=2\ell_0+3-\delta$ and
\begin{align}
D_{-1}={}&\norm{\Psiplus}^2_{W_{-2}^{\reg+\regl(j,\ell_0)}(\Sigmazero)}
+\sum_{m=1}^{\ell_0+1}\norm{\PhiplusHigh{m}}^2_{W_{-2}^{\reg+\regl(j,\ell_0)-m}(\Sigmazero^{3M})}+
\norm{rV\PhiplusHigh{\ell_0+1}}^2_{W_{-1-\delta}^{\reg+\regl(j,\ell_0)-\ell_0-2}(\Sigmatb^{3M})}\notag\\
\lesssim{}&F^{(\ell_0-1)}(\reg+\regl(j,\ell_0),5-\delta,\tb_0,\Psiplus).
\end{align}
In summary, the basic energy $\gamma$-decay condition with $\gamma=2\ell_0+3-\delta$ and $D_{-1}=F^{(\ell_0-1)}(\reg+\regl(j,\ell_0),5-\delta,\tb_0,\Psiplus)$ holds.

Given the above basic energy $\gamma$-decay condition for spin $+1$ component, then from Point \ref{pt1:prop:ptwFromBEDC:extremecomps:Schw} of Proposition \ref{prop:ptwFromBEDC:extremecomps:Schw} and the following relations for $\beta\in\{0,2\}$
\begin{align}
F^{(\ell_0-1)}(\reg+\regl(j,\ell_0),3+\beta-\delta,\tb_0,\Psiplus)\lesssim{}&
\InizeroEnergyplus{\reg+\regl(j,\ell_0)}{-1-\delta+\beta},
\end{align}
the pointwise decay estimates stated in Theorem \ref{thm:Schw} follow.
\QED


\newcommand{\mnras}{Monthly Notices of the Royal Astronomical Society}
\newcommand{\prd}{Phys. Rev. D}
\newcommand{\apj}{Astrophysical J.}
\bibliographystyle{amsplain}

\end{document}